\documentclass[11pt]{article}
\usepackage{amsmath}
\usepackage{amssymb}
\usepackage[english]{babel}
\usepackage[utf8]{inputenc}
\usepackage{multicol}
\usepackage{mathtools}
\usepackage{bbm}
\usepackage{footnote}
\usepackage{graphicx}
\usepackage[letterpaper, portrait, margin=1 in]{geometry}
\usepackage{amsthm}
\usepackage{esint}
\usepackage{xcolor}
\usepackage{soul}
\usepackage{tikz}
\usepackage{hyperref}
\usepackage{indentfirst}
\usepackage{titlesec}
\usepackage{fancyhdr}

\providecommand{\keywords}[1]
{
  \small	
  \textbf{\textit{Keywords---}} #1
}

\newcommand{\BM}[1]{\begin{bmatrix} #1 \end{bmatrix}}
\newcommand{\R}{\mathbb{R}}
\newcommand{\nn}{\nonumber \\}
\newcommand{\vl}[1]{\left<#1\right>}

\newcommand{\dt}[1]{\frac{d}{d #1}}

\addto\captionsenglish{% Replace "english" with the language you use
  \renewcommand{\contentsname}%
    {Table of Contents}%
}

\newcommand\wordcount{
    \immediate\write18{texcount -sub=section \jobname.tex  | grep "Section" | sed -e 's/+.*//' | sed -n \thesection p > 'count.txt'}
(\input{count.txt}words)}

\newtheorem{theorem}{Theorem}[]

\newtheorem{remark}{Remark}[subsection]
\newtheorem{example}{Example}[subsection]
\newtheorem{proposition}{Proposition}[subsection]
\newtheorem{definition}{Definition}[subsection]
\newtheorem{lemma}{Lemma}[subsection]
\newtheorem{corollary}{Corollary}[subsection]

\begin{document}

\pagenumbering{roman}
\numberwithin{equation}{section}
	\begin{titlepage}
	\begin{center}
		\textbf{\LARGE Uniform Lifetime for Classical Solutions to the \\
		\vskip 2mm
		Hot, Magnetized, Relativistic Vlasov Maxwell System} 
		\\
		\vspace{5mm}
		\textbf{\large by}
		\\
		\vspace{5mm}
		\textbf{\large Christophe Cheverry\footnote{Univ Rennes, CNRS, IRMAR - UMR 6625, F-35000 Rennes, France}, 
		Slim Ibrahim\footnote{Department of Mathematics and Statistics, University of Victoria, British Columbia, Canada}, 
		Dayton Preissl\footnote{Department of Mathematics and Statistics, University of Victoria, British Columbia, Canada}}
	\end{center}

    %%\addcontentsline{toc}{section}{Abstract}
    \section*{}
    \begin{abstract}
     
\noindent        This article is devoted to the kinetic description in phase space of magnetically confined plasmas. It addresses the problem 
        of stability near equilibria of the Relativistic Vlasov Maxwell system. We work under the Glassey-Strauss compactly supported momentum 
        assumption on the density function $f(t,\cdot)$. Magnetically confined plasmas are characterized by the presence of a strong \textit{external} magnetic field $ x \mapsto 
        \epsilon^{-1} \mathbf{B}_e(x)$, where $\epsilon$ is a small parameter related to the inverse gyrofrequency of electrons. In comparison, the self consistent 
        \textit{internal} electromagnetic fields $(E,B) $ are supposed to be small. In the non-magnetized setting, local $ C^1 $-solutions 
        do exist but do not exclude the possibility of blow up in finite time for large data. Consequently, in the  strongly magnetized case, 
        since $ \epsilon^{-1} $ is large, standard results 
        predict that the lifetime $T_\epsilon$ of solutions may shrink to zero when $ \epsilon $ goes to $ 0 $. However, it has been proved recently \cite{cold} that 
        for {\it neutral}, {\it cold}, and {\it dilute} plasmas (like in the Earth's magnetosphere), smooth solutions corresponding to perturbations 
        of equilibria still exist on a uniform time interval $[0,T]$, with $ 0 < T < T_\epsilon$ independent of $\epsilon$. Here we investigate the {\it hot} situation
        which is more suitable for the description of fusion devices. The methods used in the cold case fail to control the larger current density coming from the hot assumption of large initial momentum. After straightening the external field, our new strategy in this paper is to take advantage of the rapid oscillations of the characteristics using a non-stationary phase argument. This allows for uniform estimates of the linearized solution by time averaging. Notice this cannot directly be done for the non-linear system without a loss of derivatives because the characteristics depend also on the internal fields. Therefore, we overcome this through a bootstrap argument to show the distribution $f$ remains close (at a distance of size $ \epsilon $) to the linear solution, while the fields $(E,B)$ can differ by order 1 for well prepared initial data.

    \end{abstract}
    \keywords{Vlasov Maxwell, Kinetic Equations, Analysis of PDEs, Plasma Physics}
    \end{titlepage}
    \newpage
    \addcontentsline{toc}{section}{Table of Contents}
    \section*{}
	\tableofcontents
	\newpage

\pagenumbering{arabic}

\section{Introduction and Main Result}
In this article we analyze the stability and well posedness of the Hot Magnetized Relativistic Vlasov Maxwell system (called the HMRVM system in 
abbreviated form) given by 
\begin{align} \label{vlasov 1}
 &\partial_t f + [v(\xi)\cdot \nabla_x] f  - \epsilon^{-1}[v(\xi)\times \mathbf{B}_e(x)]\cdot \nabla_\xi f    - \epsilon[ E + v(\xi)\times B]\cdot \nabla_\xi f =   
 \epsilon M' (|\xi|)\frac{E\cdot \xi}{|\xi|} ,
\\
&\nabla_x \cdot E = -\rho(f) \quad ; \quad \partial_t E - \nabla \times B = J(f), \label{maxwell 11} \\
&\nabla_x \cdot B = 0 \quad ; \qquad \, \quad \partial_t B + \nabla_x \times E = 0, \label{maxwell 22}
\end{align}
along with the following charge and current densities
\begin{equation}\label{maxwell 1} 
\rho(f) := \int f(t,x,\xi) \, d\xi \, , \qquad J(f) := \int v(\xi) \, f(t,x,\xi) \, d\xi. 
\end{equation}
This system is further supplemented with initial data
\begin{align}\label{inidatadata}
    (f,E,B)|_{t=0} = (f^{in}(x,\xi),E^{in}(x),B^{in}(x)) .
\end{align}

The density function $f$ depends on the time $ t \in \mathbb{R}_+ $ and on the coordinates $(x,\xi) \in \mathbb{R}_x^3 \times \mathbb{R}_\xi^3$ which  
are phase space position-momentum variables. The electromagnetic fields $(E,B)$ depend only on time and space ($\mathbb{R}_+\times \mathbb{R}_x^3$). 
The vector field $v(\xi) := \xi /\sqrt{1+|\xi|^2}$ is the relativistic velocity. Furthermore, $0 < \epsilon \ll 1$ is a small parameter controlling the strength of the 
inhomogeneous external applied magnetic field $x\mapsto \epsilon^{-1} \mathbf{B}_e(x)$. Finally, $M(|\xi|)$ is a radially symmetric equilibrium profile 
for the Relativistic Vlasov Maxwell system.

\smallskip

As will be seen, system (\ref{vlasov 1})-$\cdots$-(\ref{maxwell 1}) describes perturbations of some hot magnetized plasmas (with density function $\mathbf{f}$)
about stationary solutions to the Relativistic Vlasov Maxwell (RVM) system, see Subsection \ref{MRVM section}. Plasmas can be created when a substance is 
heated to high enough temperatures, such that the outer electrons of atoms can be stripped away from the nuclei leaving a mixture of positive and negative 
charges. They are electrically conductive and subject to long range electromagnetic fields generated by charged particle motion. 

\smallskip 

\indent In this article, we are interested in collisionless plasmas, on time scales for which the mean electromagnetic fields dominate the plasma behavior. 
This is well described mathematically by the Relativistic Vlasov Maxwell system. Furthermore, due to the large mass difference between the ions and electrons, 
we can be concerned with time scales for which only the motion of electrons is predominant in the dynamics, whereas the ions can be viewed as a stationary, neutralizing background.

\smallskip
The Relativistic Vlasov Maxwell Cauchy problem has been extensively studied in the last few decades. R. Glassey and W. Strauss were the first 
to determine sufficient conditions for local $C^1$-solutions to exist. The monograph \cite{glassey} is a complete review of their contributions. As long as 
the distribution function $\mathbf{f}$ has compact support in the momentum variable, local smooth solutions exist and can be extended to larger time 
intervals \cite{singularity}. Global $C^1$-solutions exist for nearly neutral ($|\rho|_{t=0}| \ll 1$) and sufficiently dilute plasmas \cite{neutral,dilute}. Other 
results (see \cite{DiPerna} and related references) deal with global weak $ L^1 \cap L^2 $-solutions. A resent result by X. Wang (\cite{global solutions 2}), 
uses energy methods and a new set of commuting vector fields to give global solutions for small data which is not compactly supported in $x$ or $\xi$, 
but has polynomial decay at infinity. This result also requires high regularity on initial data.  But, the existence of global classical solutions for large data, 
even in the case when $ \mathbf{B}_e  \equiv 0$, remains an open problem.

\smallskip 

\indent
More specifically, we study the RVM system in the presence of a strong \textit{external} predetermined magnetic field $\epsilon^{-1}\mathbf{B}_e$. 
This is relevant for many applications pertaining to plasma physics. For instance, the Van Allen Belts are regions of space surrounding Earth consisting of a 
hydrogen ion plasma, for which the Vlasov Maxwell system can be used for modeling. This plasma generates it's own electromagnetic field $ (E,B) $, but is 
subject to the (much stronger and independent) magnetic field of the Earth. The Van Allen Belts shield Earth from cosmic rays and solar flares, protecting the 
atmosphere from destruction. The presence of the external magnetic field of Earth acts to confine the plasma to within a few radii of Earth. Toroidal flux surfaces 
of Earth's magnetic field prevent particles from escaping radially away. This leads to drift of the particles along the magnetic field lines and then bouncing
back and forth between magnetic poles.

\smallskip 

\indent
In this article, we are interested in the effects of a strong, inhomogeneous external magnetic field, which is denoted by $ \epsilon^{-1}\mathbf{B}_e(x)$. In dimensionless units, 
the number $ \epsilon $ is of size $\epsilon \approx 10^{-5} $ in the case of both Van Allen belts and tokamaks. From now on, we consider that $0 < \epsilon \ll 1 $ is a small 
parameter. In practice, this number $ \epsilon $  is related to the inverse of a gyrofrequency. It controls the strength of the external applied magnetic field, and thereby the 
function $ \mathbf{B}_e(\cdot)$ has an amplitude of size one. On the other hand, the variations of  the vector field $\mathbf{B}_e(\cdot)$ account for the spatial inhomogeneities 
coming from physical geometries inside the problem. The number $\epsilon$ may also be associated with the period at which the particles tend to wrap around the magnetic 
field lines. 

\smallskip 

\indent As a matter of fact, under the action of $ \epsilon^{-1}\mathbf{B}_e(x)$, the charged particles starting from the position $ x $ tend  to follow deformed cylindrical 
paths of radius $\sim \mathcal{O}(\epsilon)$ orientated along the direction $ |\mathbf{B}_e(x)|^{-1} \, \mathbf{B}_e(x)$. For longer times, the motion become much more 
complicated. But, for well-adjusted functions $ \mathbf{B}_e(\cdot) $, they remain bounded in a compact set of phase space \cite{anomalous,whistler}. This is what could 
be meant by a ``dynamical particle confinement". This property plays a crucial role in fixing the Van Allen belts to Earth. It is also essential in tokomak reactors in containing 
the plasma by preventing particles from escaping radially outwards. 

\smallskip 

\indent In concrete situations, a self-consistent (internal) electromagnetic field $ (E,B) $ does appear. This phenomenon is well described through 
the coupling between the Vlasov equation and the Maxwell equations. This induces many extra phenomena which can change the preceding stabilized picture. 
In particular, the onset of a non trivial electric field $ E \not \equiv 0 $ may have disruptive effects.  However, it can be shown that the energy of the system remains 
uniformly bounded by initial energy. This is due to the fact that the magnetic field $\mathbf{B}_e $ does no work on charged particles. This is a key observation in order 
to obtain weak solutions, as well as a set of preliminary information (see the article \cite{Bostan} and related works). But this does not allow to describe sufficiently 
precisely the structure of the solutions $(f_\epsilon,E_\epsilon,B_\epsilon)(t) := (f,E,B)(t) $ when $ \epsilon $ is small. It is mathematically 
not clear whether the external applied field $\mathbf{B}_e $ 
can generate rapid oscillations which may degrade and destabilize the plasma (for instance through resonances) by generating some non trivial $ (E,B) $. 
To better understand what happens, it is first necessary to explain how the solutions behave and interact in the limit that $\epsilon$ tends to zero. This means to get a uniform 
lifespan $ T_\epsilon $, and to control the evolution of the solutions in norms leading to sufficiently accurate information, like $ L^\infty $ or $ W^{1,\infty} $.

\smallskip 

\indent In \cite{cold}, C. Cheverry and S. Ibrahim have initiated this program. They have derived a condition for which equilibria (or stationary solutions) of the perturbation
$(f,E,B)$ of the Magnetized RVM system are stable in the sense of  $C^1([0,T], W^{1,\infty})$. The authors assume that the momentum 
variable $\xi$ is initially confined to a set of size $\epsilon$. Physically, this means that the particle velocities are bounded far away from the speed of light. This is the notion of 
``coldness" of the plasma. In cold plasmas such as the Van Allen Belts, this is a reasonable assumption. Then, the solution exists on a uniform time interval $[0,T]$ with 
$ T \in \mathbb R_+^* $, implying that the lifespan $T_\epsilon $ (which is above $ T$) does not shrink to zero with $\epsilon$ going to zero. 

\smallskip 

\indent Moreover, the solutions stay close to the equilibrium profile and they remain uniformly bounded (in $L^\infty$). The goal here is to determine whether the 
stability conditions of \cite{cold} are necessary for more general initial data. In particular we are concerned with initial data having large momentum (present in hot 
plasmas), which corresponds better to the case of fusion reactors.
 
 \smallskip 
 
 \indent Our main results are reported below. They hold for compactly supported $ C^2 $-initial data that satisfy the natural compatibility 
 conditions (\ref{support consdtion}), (\ref{neutralityazero}) and (\ref{initial data 2}) introduced in Paragraph \ref{condiinidata}. Prepared data, in the sense of Definition \ref{prepared data}, is necessary to achieve the second
estimate (\ref{weighted Lip}) concerning the uniform Lipschitz norm of $f$ (see Example \ref{ example 1} in Section \ref{non linear approx}). For this well prepared data, the amplitude of the self-consistent electromagnetic field $ (E,B) $ 
becomes uniformly bounded in the sup-norm. Moreover, the density $f$ of the HMRVM solutions relax in some sense to the associated linear system given by
\begin{align} \label{linear vlasov 1}
 &\partial_t f_\ell + [v(\xi)\cdot \nabla_x] f_\ell  - \epsilon^{-1}[v(\xi)\times \mathbf{B}_e(x)]\cdot \nabla_\xi f_\ell     =   
 \epsilon M' (|\xi|)\frac{E_\ell\cdot \xi}{|\xi|} ,
\\
&\nabla_x \cdot E_\ell = -\rho(f_\ell) \quad ; \quad \partial_t E_\ell - \nabla_x \times B_\ell = J(f_\ell), \\ \label{linear maxwell end}
&\nabla_x \cdot B_\ell = 0 \quad ; \qquad \, \quad \partial_t B_\ell + \nabla_x \times E_\ell = 0.
\end{align}
This is equipped with the same initial data  (\ref{inidatadata}). We first derive the following preliminary result which follows from the methods of \cite{cold}. 

\begin{theorem} \label{main theorem 1}[Uniform lifetime of $ C^1 $-solutions for general data and $\epsilon$-weighted sup-norm estimates]
	Select initial data $(f^{in},E^{in}, B^{in})$ as in (\ref{initial data space}), (\ref{support consdtion}), (\ref{neutralityazero}) and (\ref{initial data 2}). 
	Fix some external magnetic field $\mathbf{B}_e\in C^2(\mathbb{R}^3;\mathbb{R}^3)$ satisfying (\ref{lowupboundve})-(\ref{Be assumption}). Then, there exists
	$T>0$ and $\epsilon_0\in (0,1]$ such that for all $\epsilon\in (0,\epsilon_0]$, there is a unique solution $(f_\epsilon, E_\epsilon, B_\epsilon) $ in $ C^1(
	[0,T];L_{x,\xi}^\infty)$ to the HMRVM system (\ref{vlasov 1})-$ \cdots$-(\ref{maxwell 1}) with initial data as in (\ref{inidatadata}). This solution 
	is subject to the Glassey-Strauss condition (\ref{support for later time}) for some $R_x^T>0$ and $R_\xi^T >0$, and it is such that, for all $t\in [0,T]$, 
	we have
	\begin{align} \label{epsilon E estimate}
		||(f_\epsilon, \epsilon E_\epsilon,  \epsilon B_\epsilon)(t,\cdot,\cdot)||_{L^\infty_{x,\xi}} \leq C \bigl( T,R^T_\xi, R_x^T,||\mathbf{B}_e||_{W^{1,\infty}_x}, ||( 
		E^{in}, B^{in})||_{L^\infty_{x}}, ||{f}^{in}||_{W^{2,\infty}_{x,\xi}} \bigr) < \infty,
	\end{align}
\end{theorem}

 Theorem \ref{main theorem 1} is proved in Section \ref{fields}. In particular, the estimate (\ref{epsilon E estimate}) is given by Lemma \ref{rough E estimate}. 
The primary purpose here is to get the following which, compared to \cite{cold}, implies a new approach. 
  
 \begin{theorem} \label{main theorem 2}[Uniform lifetime of $C^1 $-solutions for prepared data with sup-norm and $\epsilon$-weighted Lipschitz 
estimates; comparison  to the linear approximation]
Let $(f^{in},E^{in}, B^{in})$ as in (\ref{initial data space}), (\ref{support consdtion}), (\ref{neutralityazero}) and (\ref{initial data 2}) and $\mathbf{B}_e\in C^2(\mathbb{R}^3;\mathbb{R}^3)$ satisfies (\ref{lowupboundve}) - (\ref{Be assumption}). Further assume that $f^{in}$ is prepared in the sense of Definition \ref{prepared data}. Then, there exists
$T>0$ and $\epsilon_0\in (0,1]$ such that for all $\epsilon\in (0,\epsilon_0]$, there is a unique solution $(f_\epsilon, E_\epsilon, B_\epsilon) $ in $ C^1(
[0,T];L_{x,\xi}^\infty)$ to the HMRVM system (\ref{vlasov 1})-$ \cdots$-(\ref{maxwell 1}) with initial data adjusted as in (\ref{inidatadata}). This solution 
is subject to the Glassey-Strauss condition (\ref{support for later time}) for some $R_x^T>0$ and $R_\xi^T >0$, and it is such that, for all $t\in [0,T]$, 
we have
\begin{align}\label{theounifest}
   ||(f_\epsilon, E_\epsilon,  B_\epsilon)(t,\cdot,\cdot)||_{L^\infty_{x,\xi}} \leq C \bigl( T,R^T_\xi, R_x^T,||\mathbf{B}_e||_{W^{1,\infty}_x}, ||( 
   E^{in}, B^{in})||_{L^\infty_{x}}, ||{f}^{in}||_{W^{2,\infty}_{x,\xi}} \bigr) < \infty,
\end{align}
as well as the (partially) $\epsilon$-weighted Lipschitz norm
\begin{align} \label{weighted Lip}
    ||\partial_t(f_\epsilon,\epsilon E_\epsilon,\epsilon B_\epsilon)(t,\cdot)||_{L^\infty_{x,\xi}} + ||\nabla_x(f_\epsilon,\epsilon E_\epsilon, \epsilon B_\epsilon)(t,\cdot)||_{L^\infty_{x,\xi}} + ||\nabla_\xi f_\epsilon(t,\cdot)||_{L^\infty_{x,\xi}} \leq C_T.
\end{align}
Moreover, $(f_\epsilon,E_\epsilon,B_\epsilon)(t,\cdot)$ remains close to the solution of (\ref{linear vlasov 1})-$\cdots$-(\ref{linear maxwell end}) in the following sense
\begin{align} \label{difference estimates}
   & ||(f_\epsilon - f_\ell)(t,\cdot)||_{L^\infty_{x,\xi}} \leq \epsilon \, C_T,   \\
\label{difference estimates1}  &  || (E_\epsilon - E_\ell, B_\epsilon  -B_\ell)(t,\cdot)||_{L^\infty_{x}} \leq C_T.  
\end{align}
\end{theorem}

 \indent 
 Additional comments on the content of Theorem \ref{main theorem 2} as well as guidelines for the strategy of proof are given in Paragraph  \ref{Mainresultposex}. 
  Note that in addition to the estimate (\ref{difference estimates}), Lemma \ref{X lemma} gives (in terms of characteristics) an explicit asymptotic solution to $f_\ell$ 
  in the case when $\mathbf{B}_e \ || \ e_3$. This is an important ingredient of our method. When the direction of $\mathbf{B}_e$ is not fixed, a similar representation 
  is available and can be exploited (although it is less explicit, see the appendix). For application purposes this serves as a major computational advantage for 
  approximating the high dimensional solution $f$ which gives information on local (in $\xi$) behavior. Theorem \ref{main theorem 2} is proved in Section 
 \ref{non linear approx}.
  
 \smallskip 

\indent The general outline for this article is further detailed below. 

\smallskip 

\indent
In Section \ref{modeling}, we introduce the Hot, Magnetized, Relativistic Vlasov Maxwell (HMRVM) system and its underlying physical 
 assumptions. In this section we construct the system (\ref{vlasov 1})-(\ref{maxwell 1}) via a perturbation of the Magnetized Relativistic Vlasov Maxwell system about stationary solutions. Extending on the work of \cite{cold} we impose \textit{neutral} and \textit{dilute} assumptions, but remove the cold assumptions. Here we explain the condition of well prepared data and the prerequisites for notion of well posedness stated precisely by the main Theorem \ref{main theorem 2}. This is progress towards the open problem:
 {\it Does the perturbed HMRVM system remain stable for ill-prepared data ?}

\smallskip 

\indent
Section \ref{fields} reviews a number of techniques of \cite{cold} from a slightly different perspective. One difference is we avoid the use of scalar and vector potentials in deriving representation formulas for the electromagnetic fields. Furthermore, this section pinpoints exactly the mathematical difficulty faced for the hot plasma regime. It concludes by reformulating the HMRVM system in terms of a new set canonical cylindrical coordinates. Such coordinates are introduced in \cite{cold}, but are not used to full potential. This has the advantage of introducing a single, periodic, rapidly oscillating variable, which is 
useful to establish averaging procedures. This observation is the main motivation for our argument. 

 \smallskip 
 
 \indent
 Next, Section \ref{non linear approx} is entirely new. The approach is to completely study the associated linear Vlasov Maxwell system (\ref{linear vlasov 1})-$\cdots$-(\ref{linear maxwell end}). The introduction of the fast, periodic variable leads to an asymptotic approximation of the characteristic curves constructed using a non-stationary phase lemma. This is key to approximating the linear system. Then, we use a bootstrap argument to approximate the non-linear system for dilute equilibrium using the linear model. In essence, the non-linear term in the HMRVM system will remain small as long as the initial data is well prepared. In other words, in both the linear and non-linear systems, prepared data is necessary to ensure uniform estimates of $|\nabla_\xi f_\ell|$ appearing in the bilinear term of the bootstap approach. This is a great accomplishment, as it allows us to precisely understand and justify hot plasma dynamics using a reduced model which possesses a derived asymptotic expansion in terms of the parameter $\epsilon$. At the very end of the text, our a priori estimates are used to show a uniform lower bound on 
 the lifetime of solutions to the system (\ref{vlasov 1})-(\ref{maxwell 1}). This is accomplish using the continuation criterion of Glassey and Strauss (\cite{glassey}), which states classical solutions can be extended as long as $f$ remains compactly supported in $\xi$. 
 
 \begin{center}
     \textbf{Acknowledgment} 
     D.P and S.I were supported by NSERC grant (371637-2019).
 \end{center}
 
 %%%%%%%%%%%%%%%%%%%%%%%%%%%%%%
%%%%%%%%%%%%%%%%%%%%%%%%%%%%%%
\newpage
\section{Modeling of Magnetized Plasma}  \label{modeling}
This section is intended to define the Hot Magnetized Relativistic Vlasov Maxwell system (HMRVM in abbreviated form). Section \ref{MRVM section} 
introduces the Magnetized Relativistic Vlasov Maxwell (MRVM) system for a plasma consisting of electrons and stationary ions. Section \ref{ assumption section} 
then introduces some physically relevant assumptions pertaining to plasmas: the \textit{hot, cold and dilute assumptions}. Improving on the work of \cite{cold}, we no 
longer impose the cold assumption. Following this, we derive the HMRVM system by considering perturbations of equilibrium solutions to the MRVM system 
in the hot regime. 

%%%%%%%%%%%%%%
 
\subsection{The MRVM System} \label{MRVM section}
\smallskip \indent
This subsection is devoted to constructing our mathematical model and precisely outlining the assumptions necessary to prove the main result given by Theorem 
\ref{main theorem 2}. We work in dimension three, with 
spatial position $ x \in \mathbb R^3 $ and momentum $ \xi \in \mathbb R^3 $. We study properties of the Vlasov Maxwell system under the influence of a strong applied magnetic field. The strength of this inhomogeneous field is controlled by a large parameter $\epsilon^{-1}$, with $\epsilon \in (0,1]$. The parameter $\epsilon$ is related to the inverse gyro-frequency. As mentioned, here we consider a two particle system consisting of electrons and a single stationary ion type. We first define the relativistic velocity as a function of the momentum $\xi$, for electron mass $m_e$ as 
\begin{align*}
v_e(\xi) := \langle \frac{\xi}{m_ec} \rangle^{-1} \,  \frac{\xi}{m_e} , \quad 1 \leq \langle \xi \rangle := \sqrt{1 + |\xi|^2} ,\ \implies \ |v_e(\xi)| < c,  \quad \forall \xi \in \mathbb{R}^3.
\end{align*}
Since the ions are assumed stationary, we are free to choose units such that mass is measured in units of electron mass $m_e$. In other words, we simply set $m_e = 1$. Furthermore, we also take the speed of light $c$ to be set to unity ($c = 1$). The electron velocity then reduces to
\begin{align*}
    v(\xi) = v_e(\xi) = \frac{\xi}{\sqrt{1+|\xi|^2}}.
\end{align*}
Therefore, the Magnetized Relativistic Vlasov Maxwell (MRVM) system on the electron density $\mathbf{f}$, with charge $Z_e = -1$, is given by: 
\begin{align} 
\partial_t \mathbf{f} + [v(\xi) \cdot \nabla_x]\mathbf{f} - \frac{1}{\epsilon}[v(\xi) \times \mathbf{B}_e(x)]\cdot \nabla_\xi \mathbf{f} &= [\mathbf{E} + v(\xi)\times \mathbf{B})]\cdot \nabla_ \xi \mathbf{f}\label{vlasov}, \\
\nabla_x \cdot \mathbf{E} = \rho_i - \rho(\mathbf{f}) \quad &; \quad \partial_t \mathbf{E} - \nabla_x \times \mathbf{B} = \mathbf{J}(\mathbf{f}),  \label{maxw1}\\
\nabla_x \cdot \mathbf{B} = 0 \quad &; \quad \partial_t \mathbf{B} + \nabla_x \times \mathbf{E} = 0. \label{maxw2}
\end{align}
Equation (\ref{vlasov}) is known as the Vlasov equation, and (\ref{maxw1}) - (\ref{maxw2}) are Maxwell's equations governing propagation of the fields. The constant $\rho_i \in \mathbb{R}_+$ represents the background ion charge density. The current and charge densities of the electrons are defined respectively as 
\begin{align}
\mathbf{J}(\mathbf{f})(t,x) &:= \int v(\xi)\mathbf{f}(t,x,\xi) d\xi , \\
\rho(\mathbf{f})(t,x) &:= \int \mathbf{f}(t,x,\xi)d \xi.
\end{align}
The unknown in the above system is $\mathbf{U}:= {}^t (\mathbf{f}, \mathbf{E}, \mathbf{B})$. We impose a strong inhomogeneous exterior magnetic field that is smooth, non-vanishing, divergence free, and curl free. More specifically, for any compact set $K\subset \mathbb{R}^3$, there exists a constant $c(K) > 0 $ such that
\begin{align}\label{lowupboundve}
\forall x\in K, \quad c(K)\leq b_e(x) \leq c(K)^{-1} , \quad b_e(x) := |\mathbf{B}_e (x) |
\end{align}
and
\begin{align} \label{Be assumption}
\forall x\in \mathbb{R}^3, \quad \nabla_x \cdot \mathbf{B}_e(x) \equiv 0, \quad \nabla_x\times \mathbf{B}_e(x) = 0 .
\end{align}
The article \cite{cold} gives an extensive treatment of uniform estimates with respect to $\epsilon \in (0,1] $, as well as stability of $\mathbf{U}$ under particular technical assumptions 
related to a perturbed regime about stationary solutions. The aim of this article is to prove similar stability when these assumptions are removed, in particular the cold assumption.

%%%%%%%%%%%%%%

\subsection{Assumptions and Framework}
\label{ assumption section}
Before stating the physical assumptions, we first introduce  a family of equilibria denoted by $ \mathbf{U}^s := (\mathbf{f}^s, \mathbf{E}^s, \mathbf{B}^s)$, which have the form
\begin{align} 
\mathbf{f}^s(t,x,\xi) :=  M_\epsilon (|\xi|), \quad \mathbf{E}^s := 0, \quad \mathbf{B}^s := 0. \label{stationary}
\end{align}
Given any non-negative function $ M_\epsilon \in C^1_c(\mathbb{R}_+;\mathbb{R}_+)$, it will later be important to consider the gradient $\nabla_\xi M_\epsilon(|\xi|) 
= \xi \, M_\epsilon'(|\xi|) / |\xi | $. So for technical reasons, we impose that $ M'_\epsilon(|\xi|)  / |\xi | $ remains bounded as $|\xi|\rightarrow 0$. For this, it is sufficient 
to impose that $M_\epsilon(\cdot)$ has an even $C^1$-extension to the entire real line $\mathbb{R}$. Typically this is accomplished in the relativistic setting by 
considering the particular case $M_\epsilon(|\xi|) = \tilde{M}_\epsilon(\vl{\xi})$, so $\nabla_\xi\tilde{M}_\epsilon(\vl{\xi}) = v(\xi)\tilde{M}'(\vl{\xi})$ which has no singularity 
at $|\xi| = 0$ when $\tilde{M}_\epsilon \in C^1_c(\mathbb{R}_+;\mathbb{R}_+)$. Furthermore, We can always adjust $ \rho_i $ in such a way that  $\rho_i := 
\rho(M_\epsilon) $. Then, the expression $\mathbf{U}^s$ is sure to solve (\ref{vlasov})-(\ref{maxw2}). Thus, it is a stationary solution of (\ref{vlasov})-(\ref{maxw2}), 
hence the superscript``$ s $" while the subscript ``$ \epsilon $" is put to mark a possible dependence on $ \epsilon $. 

\smallskip

\noindent The goal is to perturb the stationary solutions $ \mathbf{U}^s $, and to examine their stability. To this end, we need to impose constraints on the 
data $ \rho_i  $ and $ M_\epsilon $. In \cite{cold}, the plasma was supposed to be \textit{globally neutral}, \textit{cold} and \textit{dilute}. In what follows below, we come back to the definitions of these three key assumptions.

%%%%%%%%%%%%%%
\smallskip

\smallskip
{\sc{\textbf{Global Neutrality:}}} The first important assumption is the {\it neutrality} assumption which describes the apparent charge neutrality of a plasma overall. This property is 
widely used when looking at plasmas. It is sometimes qualified as \href{https://www.plasma-universe.com/quasi-neutrality/}{quasi-neutrality} because, at smaller scales, the 
positive and negative charges may give rise to charged regions and electric fields. In the present context, for each equilibrium profile $M_\epsilon$, this means to fix the constant 
$\rho_i := \rho_i(M_\epsilon) =  ||M_\epsilon||_{L^1}$ in such a way that
\begin{align} \label{neutrality assumption}
\rho(\mathbf{f}^s) = \rho_i . \ \text{(Neutral background)}.
\end{align}

%%%%%%%%%%%%%%

\smallskip 

{\sc{\textbf{Coldness:}}} The next assumption that is involved in \cite{cold} is the notion of {\it coldness}. After rescaling, this condition limits particle momentum to be concentrated 
near the origin, i.e. $|\xi| \sim \mathcal{O}(\epsilon)$. This may be achieved by looking at equilibria such as
\[ \mathbf{f}^s (t,x,\xi) =  M_\epsilon (|\xi|) := \epsilon^{-2} \, M(\epsilon^{-1}|\xi|) , \]
where $ M \in C^1_c (\mathbb R^3) $ is adjusted in such a way that (for some constant $ R_M $)
\begin{align}\label{suppcomp1} 
\text{supp}(M) &\subset \{\xi \in \mathbb{R}^3 \ | \ |\xi| \leq R_M \} .
\end{align}
Next, \cite{cold} considered perturbed solutions having the form
\begin{align} \label{cold perturbation}
\mathbf{f}(t,x,\xi) = \epsilon^{-2} \, [M(\epsilon^{-1}|\xi|) + {f}(t,x,\epsilon^{-1}\xi)] .
\end{align}
Recall that a sufficient condition for local existence of smooth solutions of (\ref{vlasov})-(\ref{maxw2}) to exist on $ [0,T] $ with $ 0 < T $ is that 
$\mathbf{f}(t,x,\cdot)$ has compact support in the variable $\xi$ for $t \in [0,T] $. With this in mind,  in \cite{cold}, local $ C^1 $-solutions satisfying 
(for some constants $ R_x $ and $ R_\xi $)
\begin{align}
\text{supp} \,  f(t,\cdot,\cdot) &\subset \{(x,\xi)\in \mathbb{R}^3\times\mathbb{R}^3 \ | \ |x| \leq R_x,\text{ and } \ |\xi| \leq R_\xi \} \label{suppcomp2}
\end{align}
were constructed on $ [0,T] $. The combination of the three restrictions (\ref{suppcomp1}), (\ref{cold perturbation}) and (\ref{suppcomp2}) meant 
that for $|\xi| \geq \epsilon \max\{R_\xi, R_M\}$, there was $\mathbf{f}(t,x,\xi) = 0$.

%%%%%%%%%%%%%%

\medskip 
{\sc{\textbf{Dilute:}}}\label{subsec:Diluteassumption} The last assumption given in \cite{cold} is the {\it dilute} assumption which is given by the condition $\rho_i = \mathcal{O}(\epsilon)$. This may be 
viewed as a direct consequence of (\ref{neutrality assumption}) and (\ref{suppcomp1}) since we have
\begin{align}
    \rho_i = ||\epsilon^{-2}M(\epsilon^{-1}|\cdot|)||_{L^1} = \epsilon ||M||_{L^1} = \mathcal{O}(\epsilon) .
\end{align}

%%%%%%%%%%%%%%

The global neutrality condition is physically relevant at the scales under consideration.
It is therefore unavoidable, and we keep it. By contrast, the cold assumption is not suitable in the case of many applications like fusion devices. 
Here we remove this condition so that for most of the plasma we have $|\xi| \sim \mathcal{O}(1)$. Therefore in this article 
we consider equilibrium profiles of the form 
\begin{align} 
    (\mathbf{f}^s,\mathbf{E}^s,\mathbf{B}^s) = (\epsilon M(|\xi|),0,0), \quad & \rho_i = \epsilon ||M||_{L^1} \  \text{ (Neutral, Hot and Dilute)}. \label{neutral hot dilute}
\end{align}

%%%

{\sc{\textbf{The HMRVM System:}}} We consider a perturbation of the equilibrium solution as indicated below:  
\begin{align} \label{anzats}
\mathbf{f}(t,x,\xi) :=   \epsilon M (|\xi|) + \epsilon {f}(t,x,\xi), \quad \mathbf{E}(t,x) := \epsilon E(t,x), \quad \mathbf{B}(t,x) :=  \epsilon B(t,x).
\end{align}
Furthermore, let
 \begin{align} \label{initial data space}
 U^{in} := ({f}^{in},E^{in},B^{in}) \in C_c^2(\mathbb{R}^3\times\mathbb{R}^3) \times C^2_c (\mathbb{R}^3)\times C^2_c (\mathbb{R}^3) 
 \end{align}
 be some initial functions. Consider the system (\ref{vlasov})-(\ref{maxw2}) with initial data given by
\begin{align} \label{initial data 1}
\mathbf{f}|_{t=0} &= \mathbf{f}^{in} :=  \epsilon M(|\xi|) +  \epsilon {f}^{in}(x,\xi), \\  
\mathbf{E}|_{t=0} &= \mathbf{E}^{in} :=  \epsilon E^{in}(x), \\
\mathbf{B}|_{t=0} &= \mathbf{B}^{in} :=  \epsilon B^{in}(x). 
\end{align}
Substituting the expression (\ref{anzats}) into the system (\ref{vlasov})-(\ref{maxw2}) leads to the HMRVM system (\ref{vlasov 1})-$ \cdots$-(\ref{maxwell 1}) 
which is the main focus of our article.

%%%

\subsubsection{Conditions on the Initial Data}\label{condiinidata} 
Select $ (R_x^0 , R_\xi^0)  \in \mathbb R_+^* \times \mathbb R_+^* $, with $R_M \leq R_\xi^0$ and define 
\[ R^0 := \max \{ R_x^0 , R_\xi^0 \} \, . \]
In the sequel we impose 
\begin{align} \label{support consdtion}
\text{supp}({f}^{in}) \subset \{ (x,\xi) \ | \ |x|\leq R_x^0 \text{ and }|\xi| \leq R_\xi^0 \}.
\end{align}
Remark that if ${f}^{in}$ is compactly supported in $x$, then by the relation (\ref{initial data 1}) implies $\mathbf{f}^{in}$ is not, since for large 
enough $|x| > R_x^0$ we must have $\mathbf{f}^{in}(x,\xi) = \epsilon M(|\xi|)$.  To guarantee the neutrality at time $ t = 0 $, we have to adjust 
$ {f}^{in} $ in such a way that
\begin{align} \label{neutralityazero}
\forall x \in \mathbb R^3 , \quad \int {f}^{in} (x,\xi) \, d \xi = 0 . 
\end{align}
This is our notion of a perturbation. Although, $\epsilon M$ and $\epsilon f^{in}$ are of the same size in $L^\infty$, unlike $M \geq 0$, the perturbation $f^{in}$ is globally neutral. We also pay special attention to initial data that are prepared in the following sense.
\begin{definition} \label{prepared data}
Initial data, ${f}^{in} \equiv {f}^{in}_\epsilon $, is said to be \textbf{prepared} if there exists some $C>0$ such that for all $\epsilon\in (0,1]$
\begin{align}\label{concontr}
    ||[v(\xi)\times \mathbf{B}_e(x)]\cdot\nabla_\xi {f}^{in}_\epsilon ||_{L^\infty_{x,\xi}} \leq C \, \epsilon .
\end{align}
Data is said to be \textbf{ill-prepared} if the right hand side of (\ref{concontr}) must be replaced with $C$. 
\end{definition}

\noindent Remark that this condition arises naturally from equation (\ref{vlasov 1}) which yields 
\begin{align}
    \partial_t{f}|_{t=0} = \epsilon^{-1} \, [v(\xi)\times \mathbf{B}_e(x)]\cdot\nabla_\xi {f}^{in} + \mathcal{O}(1).
\end{align}
Thus, in the absence of (\ref{concontr}), the time derivative of ${f}$ at time $t=0$ is large. This means that (\ref{concontr}) is a necessary 
condition for uniform estimates in the Lipschitz norm. 

\smallskip

\noindent Given ${f}^{in}$ as above, we have to assume that the initial data $ E^{in} $ and $ B^{in} $ satisfy at time $ t = 0 $ 
the necessary compatibility conditions:
\begin{align} 
\nabla_x\cdot E^{in} = \rho({f}^{in}),\quad  \nabla_x \cdot B^{in} = 0. \label{initial data 2}
\end{align}

%%%

{\sc{\textbf{{Compact Support:}}}} Finally, we work under the classic Glassey-Strauss momentum condition. As in \cite{cold}, this means to look at a time interval $ [0,T] $,
with $ T $ below the maximal lifetime of $({f},E,B)(t,\cdot)$ solving the HMRVM system, such that
\begin{align} \label{support for later time}
    \forall t\in [0,T], \ \text{supp}({f}(t,\cdot)) \subset \{ (x,\xi) \ | \ |x|\leq R_x^T, \text{ and } |\xi|\leq R_\xi^T \}
\end{align}
for some $R_x^T>0$ and $R_\xi^T >0$. It is easy to show (in the relativistic context) that we may take $R_x^T = R_x^0 + T$. But there is 
no such evident control concerning $R_\xi^T$. In particular, we would like to show a uniform (in $ \epsilon $) positive lower bound for $ T $, 
as well as a uniform  (in $ \epsilon $) upper bound for $R_\xi^T$. This is what has been done in \cite{cold}. Define
\begin{align*}
    R^T := \max\{R_x^T, R_\xi^T\} ,
\end{align*}
and consider the set
\begin{align}
    \mathcal{A}_T := \{ (y,\eta) \ | \ |y| \leq R_x^0 + T, \text{ and } |\eta| \leq R_\xi^T \}.
\end{align}
With this in mind, in this article we define the norms
\begin{align} \label{norms}
    ||f(t,\cdot,\cdot)||_{L^\infty_{x,\xi}} &:= ||f(t,\cdot,\cdot)||_{L^\infty_{x,\xi}(\mathcal{A}_T)} = \sup\{ |f(t,x,\xi)| \ | \ (x,\xi)\in \mathcal{A}_T\}
    \nn
    ||(E,B)(t,\cdot)||_{L^\infty_x} &:= \sup\{ |(E,B)(t,x)| \ | \ |x|\leq R_x^0 + T \}
\end{align}
This is the precise interpretation of the estimates of Theorem \ref{main theorem 2}. Note that the very recent result \cite{global solutions 2} states that the compact support assumption (\ref{support for later time}) 
is not necessary for global well possedness of the RVM system. Indeed, using energy methods, X. Wang obtains global regularity for small initial data with decay rates $|\xi|^{-7}$ at infinity. However, here the smallness condition cannot be applied due to the large external magnetic field. 

%%%%%%%%%%%%%%

\subsection{Comments on Theorem \ref{main theorem 2} and Strategy of the Proof}\label{Mainresultposex}

In the new context of neutral, hot and dilute plasmas, Theorem \ref{main theorem 2} constructs solutions having a uniform lifespan and satisfying uniform
sup-norm estimates. It means that the dilute equilibrium given by (\ref{neutral hot dilute}) is a stable solution to the HMRVM system under 
prepared perturbed initial data.  Observe that Example \ref{ example 1} in Section \ref{non linear approx} gives a situation showing that the prepared data 
assumption is not necessary (although sufficient) to ensure the uniform estimate (\ref{theounifest}), while it is definitely required in view of (\ref{weighted Lip}).
In the sequel, for the sake of completeness, we plan to investigate (\ref{theounifest}) also in the case of large equilibrium profiles, when the profile 
$ \epsilon M(|\xi|)$  is replaced by $M(|\xi|)$ (which is of size $ 1 $).

 \smallskip

The strategy of the proof of Theorem \ref{main theorem 2} is as follows. We start Section \ref{fields} by first constructing representation formulas of the electromagnetic fields. This is accomplished using standard methods of solving a linear wave equation satisfied by $(E,B)$. But  this yields source terms with derivatives (in $t$ and $x$) on the current and charge densities involving $f$. Using a passing of derivatives argument (as done in \cite{classical}), involving vector field methods, we may replace these derivatives with $\xi$ derivatives using the Vlasov equation. An integration by parts in $\xi$ allows one to then achieve (non-uniform) a priori sup norm estimates of the fields in terms of the sup-norm of $f$ as long as $f$ remains compactly supported in $\xi$. 

 \smallskip 
 
 However, in the presence of the large external magnetic field, these rough a priori estimates involve a large amplitude term issued by $\epsilon^{-1}\mathbf{B}_e$. In the cold setting of the paper \cite{cold}, this problematic term is resolved by the assumption on the localization of initial data in the momentum variable. In other words, since the momentum is confined to a small set for which $\xi \sim \mathcal{O}(\epsilon)$, the non-local nature (from the integration in $\xi$) of the source terms of Maxwell's equations allow the authors to regain a factor of $\epsilon$. In the hot setting this cannot be done and a more optimal method is necessary. The main idea is to take advantage of the rapid oscillations issued by $\epsilon^{-1}\mathbf{B}_e$. To observe these oscillations, we consider a canonical coordinate system (Section \ref{field straigtening section} : Field Straightening). This aligns the external magnetic field along a single, fixed axis which introduces a single, fast, periodic variable $\theta$ (in cylindrical coordinates) for the momentum. Using these new variables, we can solve the Vlasov equation, as usual, via the method of characteristics. Thus, an integration by parts in time of the rapidly oscillating source term involved in solving for $(E,B)$ allows us to overcome the penalization of $\epsilon^{-1}\mathbf{B}_e$. 
 
 \smallskip 
 
That being said, in the non-linear system this posses a further difficulty. This integration by parts in time further imposes derivatives on the characteristics (and thus on the electromagnetic fields), and Gr\"{o}nwall-type estimates are no longer available. To avoid such a difficulty we first study the associated linear system given by (\ref{linear vlasov 1})-$\cdots$-(\ref{linear maxwell end}). For this system, the characteristics are completely determined by the asymptotic decomposition stated by Lemma \ref{X lemma} when $\mathbf{B}_e || e_3$ (and Lemma \ref{general field lemma} in the general setting) and solutions $(f_\ell,E_\ell, B_\ell)$ are uniformly well posed in the sup-norm even for ill-prepared initial data. For well prepared initial data, the linear system further possesses uniform bounds in the Lipschitz norm with respect to $t,x$ and $\xi$. From here we can then apply a bootstrap argument and then follow the lines of \cite{cold} to control the bilinear term of the Vlasov equation for the new variables $\epsilon f^\delta := (f - f_\ell)$ and $(E^\delta,B^\delta) := (E - E_\ell, B-B_\ell)$. From this choice of scaling, we can uniformly control $(E^\delta, B^\delta)$ in terms of $f^\delta$. However, in the 
corresponding Vlasov equation $ \nabla_\xi f_\ell$ and $\epsilon^{-1}M'_\epsilon = M' $ act as source terms. Thus as long as the equilibrium profile $M_\epsilon = \epsilon M$ is small and the initial data is well prepared (controlling $|\nabla_\xi f_\ell|$) we can uniformly control $f^\delta$ as well, implying $|f - f_\ell| =\mathcal{O}(\epsilon)$. In the sequel of this article it will be interesting to determine if this smallness assumption is necessary. 

\smallskip

%%%%%%%%%%%%%%%%%%%%%%%%%%%%%%%%%%%
%%%%%%%%%%%%%%%%%%%%%%%%%%%%%%%%%%%
\newpage
\section{Fundamental Solutions} \label{fields}
This section is devoted to constructing solutions of the HMRVM system. Section \ref{fundamental sol section} starts by deriving representation formulas for the electromagnetic fields $(E,B)$.  Here we take a direct approach of solving the electromagnetic fields and avoid the usage of scalar and vector potentials. Instead, we show directly that the fields $(E,B)$ solve a linear wave equation. The solutions to the fields $(E,B)$ are then represented using the fundamental solutions of the wave equation and Kirchhoff's formula. In doing so, this introduces a source term depending on the derivatives $\partial_t f$ and $\nabla_x f $. Section \ref{transfer of derivs section} then uses a classical division lemma of \cite{classical} to pass the time and spatial derivatives on this source term to the transport operator $T := \partial_t + v(\xi)\cdot \nabla_x$. This allows us to use the Vlasov equation and substitute $T(f)$ with a divergence in $\xi$ term. Since the current and charge densities posses an integration in the momentum, we can then integrate by parts to remove these derivatives from $f$ and estimate the fields $(E,B)$ in terms of $f$. This allows us to arrive at a similar expression presented in \cite{classical} now in the presence of an applied magnetic field. However, unlike \cite{cold}, we are not able to uniformly estimate the fields with respect to $\epsilon$. This is due to the fact, that we no longer have a cold plasma, and so cannot recover the factor of $\epsilon$ using the small momentum assumption. Section \ref{obstruction section} states precisely this difficulty and why the methods of \cite{cold} do not suffice in the hot plasma regime. Next in Section \ref{vlasov rep section}, we solve the Vlasov equation using the method of characteristics and Duhamel's principle. Then we prove 
the rough estimate (\ref{epsilon E estimate}) where, in comparison to (\ref{theounifest}), the weight $\epsilon$ is in factor of the fields. We finally conclude with Section \ref{field straigtening section} which reformulates the Vlasov equation in new canonical coordinates. The main idea is the operator $[v(\xi)\times \epsilon^{-1}\mathbf{B}_e]\cdot \nabla_\xi$ becomes $\epsilon^{-1}\vl{\xi}^{-1}|\mathbf{B}_e|\partial_\theta$ after this change of variables, where in the new momentum variables,  $\partial_\theta$ is the derivative in cylindrical coordinates. The remainder of Theorem \ref{main theorem 2} is addressed in Section \ref{non linear approx}.

%%%%%%%%%%%%%%%%%%%

\subsection{Fundamental Solution of the Wave Equation} \label{fundamental sol section}
One approach to obtain representation formulas of the electromagnetic fields $(E,B)$ is through a wave equation. We define the $3D$ D'Alembertian as follows.
\begin{align*}
    \square_{t,x} := \partial_t^2 - \Delta_x = \partial_t^2 - \sum_{i=1}^3\partial_{x_i}^2.
\end{align*}
Then the following lemma gives the precise relation between $(E,B)$ and the operator $\square_{t,x}$.
\begin{lemma} \label{E B wave equation lemma}
Let $f \in C^1([0,T]\times\mathbb{R}^6;\mathbb{R})$ with compact support in $\xi$. Then the self-consistent electromagnetic field $(E,B)$ solving (\ref{maxwell 11})-(\ref{maxwell 22}) is in 
$ C^1([0,T]\times \mathbb{R}^3;\mathbb{R}^3)$ and satisfies
\begin{align} \label{E B Wave equation}
    \square_{t,x} E&= \int v(\xi)\partial_t f + \nabla_xfd\xi, 
    \nn
    \square_{t,x}B &= - \int \nabla_x\times (v(\xi)f)d\xi.
\end{align}
\end{lemma}
\begin{proof}
Consider differentiating (\ref{maxwell 11})-(\ref{maxwell 22}) with respect to $t$. This gives
\begin{align} \label{fields wave equation step}
    \partial_t^2 E - \nabla_x \times \partial_t B &= \int v(\xi)\partial_t fd\xi,
    \nn
     \partial_t^2 B + \nabla_x \times \partial_t E &= 0.
\end{align}
We substitute $(\partial_tE,\partial_tB)$ from Maxwell's equations into (\ref{fields wave equation step}) and use the vector identity 
\begin{align*}
    \nabla \times( \nabla \times A) = \nabla(\nabla \cdot A) - \Delta A,
\end{align*}
to obtain
\begin{align*}
    \int v(\xi)\partial_tfd\xi &= \partial_t^2 E - \nabla_x\times (-\nabla_x\times E) =\partial_t^2 E -\Delta_x E + \nabla_x(\nabla_x\cdot E) 
    \nn
    &=\partial_t^2 E -\Delta_x E - \int\nabla_x fd\xi.
\end{align*}
Similarly,
\begin{align*}
    0 &= \partial_t^2B + \nabla_x\times \big(\nabla_x \times B + \int v(\xi)fd\xi \big)
    \nn
    &= \partial_t^2 B - \Delta_x B + \nabla_x(\nabla_x\cdot B) + \int \nabla_x\times (v(\xi)f)d\xi
    \nn
    &=\partial_t^2 B - \Delta_x B  + \int \nabla_x\times (v(\xi)f)d\xi.
\end{align*}
This is the desired result.
\end{proof}
Next, we introduce the fundamental solution of the wave equation. This will allow us to write a solution of (\ref{E B Wave equation}) in terms of the derivatives $\partial_tf$ and $\nabla_x f$. We first define a space of distributions.
\begin{definition}
We define the space $\mathcal{D}'(\mathbb{R}^n;\mathbb{R})$ to be the set of continuous linear functionals on $C_c^\infty(\mathbb{R}^n;\mathbb{R})$. For $\phi \in C_c^\infty(\mathbb{R}^n;\mathbb{R})$, and $S\in \mathcal{D}'(\mathbb{R}^n;\mathbb{R})$ we use the notation 
\begin{align}
    S(\phi) = \vl{S,\phi} = \int_{\mathbb{R}^n}S(x)\phi(x)dx,
\end{align}
where the rightmost term is imprecise, but will be used for formal computations used in place of density arguments. 
\end{definition}
Now we look at the fundamental solution of the wave equation solving the Cauchy problem
\begin{align} \label{Y problem}
    \square_{t,x} Y = \delta(t,x), \quad Y|_{t=0} = 0, \quad \partial_t Y(0,x) = 0.
\end{align}
Here $\delta(t,x)\in \mathcal{D}'(\mathbb{R}^4;\mathbb{R})$ is the Dirac mass on $\mathbb{R}\times\mathbb{R}^3$. The distribution $Y\in \mathcal{D}'(\mathbb{R}^4;\mathbb{R})$ solving (\ref{Y problem}) is given by
\begin{align} \label{fundamental solution of wave}
    Y(t,x) := \frac{\delta(t-|x|)}{4\pi t}\mathbbm{1}_{t>0}.
\end{align}
Moreover, we define a convolution with the distribution $Y$,  $Y*_{t,x}g$ (occasionally written as $Y*g$ when the context is clear), for an integrable function $g$, to be given by
\begin{align}
  Y*_{t,x} g = \int_{\mathbb{R}} \! \int_{\mathbb{R}^3} Y(s,y) \, g(t-s,x-y) \, dsdy.
\end{align}
Explicit formulas for this are given by Lemma \ref{wave integral lemma}. Using Lemma \ref{E B wave equation lemma} we obtain the representation formula for $(E,B)$ using Kirchhoff's formula from \cite{evans}. 
\begin{corollary}
Let $f \in C^1([0,T]\times\mathbb{R}^3\times\mathbb{R}^3;\mathbb{R})$ be a solution of (\ref{vlasov 1}), then $(E,B)(t,\cdot)$ solving (\ref{maxwell 11})-(\ref{maxwell 22}) has the solution 
\begin{align} \label{equations for E and B}
 E(t,x) &= K_1(E^{in}) + Y *_{t,x}\mathbbm{1}_{t>0}\int [ v(\xi)\partial_t f +  \nabla_x f]d\xi,
\\
B(t,x) &=  K_2(B^{in}) - Y *_{t,x}\mathbbm{1}_{t>0}\int [\nabla_x \times(v(\xi) f)]d\xi,
\end{align}
where $K_i$ depends on the initial data
\begin{align}  \label{K1}
K_1(E^{in})(t,x) &:= \frac{1}{4\pi t^2 }\int_{|y-x|=t} \big[t\partial_t   E|_{t=0}(y)  + E^{in}(y) + [(y-x)\cdot\nabla_y]E^{in}(y) \big] dS(y),
\\
K_2(B^{in})(t,x) &:= \frac{1}{4\pi t^2 }\int_{|y-x|=t} \big[t\partial_t B|_{t=0}(y) + B^{in}(y) +[(y-x)\cdot\nabla_y]B^{in}\big] dS(y),
\end{align}
with 
\begin{align} \label{dt E initial}
    \partial_tE|_{t=0} = J(f^{in}) +  \nabla\times B^{in}, \quad  
    \partial_t B|_{t=0} = - \nabla_x\times E^{in}.
\end{align}
\end{corollary}
Remark for $i=1,2$, we have the sup-norm estimate
\begin{align}
    ||K_i(t,\cdot)||_{L^\infty_{x}} \lesssim t \, (||f^{in}||_{L^\infty_{x}, L^1_\xi} + 2||\nabla_x(E^{in}, B^{in})||_{L^\infty_x}) + ||(E^{in}, B^{in})||_{L^\infty_x}.
\end{align}

\subsection{Transfer of Derivatives} \label{transfer of derivs section}
\smallskip \noindent
 The idea is to pass the derivatives $\partial_t f$ and $\nabla_x f$ in (\ref{equations for E and B}) to a derivative with respect to $\xi$ and integrate by parts in order to apply a Gr\"{o}nwall type lemma to estimate the fields $(E,B)$ in terms of $f$ only. This is done by using a corollary of the \textit{division lemma} from \cite{classical} to obtain a transport operator on $f$ for which the Vlasov equation can be substituted in (\ref{equations for E and B}). First define the spaces of smooth homogeneous functions $\mathcal{M}_k$ on $\mathbb{R}^n - \{0\}$,
 \begin{align}
     \mathcal{M}_k(\mathbb{R}^n - \{0\}) := \bigg\{ \phi \in C^\infty(\mathbb{R}^n-\{0\}) \ | \ \phi(\alpha x) = \alpha^{k}\phi(x), \  \forall\alpha>0 \bigg\}.
 \end{align}
 Then set $\mathfrak{M}_k(\mathbb{R}^n - \{0\})$ to be the space of homogeneous distributions on $\mathbb{R}^n - \{0\}$ of degree $k$. This means $S\in \mathfrak{M}_k(\mathbb{R}^n-\{0\})$ if for all $\lambda>0$ and $\phi\in C_c^\infty(\mathbb{R}^n - \{0\})$ we have
 \begin{align}
     \vl{S, M_\lambda \phi} = \lambda^{k+n}\vl{S,\phi},
 \end{align}
 where $M_\lambda\phi(x) := \phi(\lambda^{-1}x)$. In particular $\mathcal{M}_k\subset \mathfrak{M}_k$. Remark that we will not make the distinction here between homogeneous distributions on $\mathbb{R}^n$ and $\mathbb{R}^n-\{0\}$, since we will only consider distributions of degree $k>-n$. By a result in \cite{distribution} any homogeneous distribution on $\mathbb{R}^n-\{0\}$ of degree $k>-n$ has a unique homogeneous extension to a distribution on $\mathbb{R}^n$. Thus we simply identify the distributions on $\mathbb{R}^n-\{0\}$ with those on $\mathbb{R}^n$ 
 \begin{align*}
 \mathcal{M}_k(\mathbb{R}^n) \sim \mathcal{M}_k(\mathbb{R}^n - \{0\}) \text{ and } \mathfrak{M}_k(\mathbb{R}^n) \sim \mathfrak{M}_k(\mathbb{R}^n - \{0\}), \text{ for } k>-n.    
 \end{align*}
Also note that $Y\in \mathfrak{M}_{-2}(\mathbb{R}^4)$. Next we define the transport operator (also known as the convective derivative) $T$ as
 \begin{align}
     T := T(\xi) = \partial_t + v(\xi)\cdot \nabla_x.
 \end{align}
 The goal is to exchange $[v(\xi)\partial_t + \nabla_x ]f$ and $\nabla_x\times( v(\xi)f)$ in (\ref{equations for E and B}) by commuting the derivatives onto $Y$ through the convolution, and express each $\partial_iY$ in terms of $T$. This is given precisely in the following lemma from \cite{classical}. 
 
 \begin{lemma} \label{division lemma}
 Let $Y$ be the fundamental solution of the wave equation given by (\ref{fundamental solution of wave}). Then there exits homogeneous functions $p, a^0\in \mathcal{M}_0(\mathbb{R}^4)$ and $a^1, q\in \mathcal{M}_{-1}(\mathbb{R}^4)$ such that
 \begin{align}
     [v(\xi)\partial_t + \nabla_x]Y &= -T(\xi)(pY) + qY \in \mathfrak{M}_{-3}(\mathbb{R}^4),
     \nn
     \nabla_x\times (v(\xi)Y) &= T(\xi)(a^0Y) + a^1Y \in \mathfrak{M}_{-3}(\mathbb{R}^4).
 \end{align}
     In fact, we have the precise expressions
 \begin{align}\label{p formula}
    p(t,x,\xi) &:= \frac{v(\xi)t - x}{v(\xi)\cdot x - t},  &q(t,x,\xi) &:= \frac{1}{\vl{\xi}^2}\frac{v(\xi)t - x}{(v(\xi)\cdot x  -t)^2}
\end{align}
with similar expression for $a^0$ and $a^1$ given in \cite{cold}.
 \end{lemma}
 The proof of Lemma \ref{division lemma} is in \cite{cold}.  Lemma \ref{division lemma} can be physically interpreted as follows. The Vlasov equation has a speed of propagation, $v(\Xi)$, in the spatial variable, where $\Xi(t)$ solves the momentum component of the characteristic curves of the Vlasov equation. The main remark is that for compactly supported momentum, we have the control $|v(\Xi)| < C < 1$. This means that individual particle velocities are uniformly bounded away from the speed of light. On the other hand, the electromagnetic waves $(E,B)$ travel at a speed $c=1$, ahead of the transport of $f$. This feature that transport speed never surpasses the wave speed is crucial. It allows for the distributions $p$ and $q$ (as well as $a^0$ and $a^1$) to be well defined away from the light cone $\{|x| = t\}$.
 
The next two lemmas enable us to write (\ref{equations for E and B}) in a way that allows both the use of Lemma \ref{division lemma} and a way to estimate $(E,B)$.

\begin{lemma}\label{wave integral lemma}
Let $p\in \mathcal{M}_m(\mathbb{R}^4)$ with $m\geq -1$ and $f\in L^\infty(\mathbb{R}_+\times\mathbb{R}^3;\mathbb{R})$. Then the following expression can be written
\begin{align}
    \bar{u}(t,x) &:= (pY)*(f\mathbbm{1}_{t>0}) \nn
    &=\int_0^t \int_{\mathbb{S}^2} \frac{p(1,\omega)}{4\pi}f(t-s,x-s\omega)s^{1+m}d\omega ds,
\end{align}
where $\omega = \frac{y}{|y|}\in\mathbb{S}^2$. Furthermore, from this we obtain the estimate
\begin{align}
|\bar{u}(t,x)| \leq \frac{t^{1+m}}{3}||p(1,\cdot)||_{L^\infty(\mathbb{S}^2)}\int_0^t||f(s,\cdot)||_{L^\infty(\mathbb{R}^3_x)} ds.
\end{align}
\end{lemma}
\begin{proof}
By direct formal computation, upon converting to polar coordinates, we have
\begin{align*}
    \bar{u}(t,x) &= \int_{\mathbb{R}^4} p(s,y)\frac{\delta(s - |y|)}{4\pi s}\mathbbm{1}_{s>0}f(t-s,x-y)\mathbbm{1}_{t-s>0}dsdy
    \nn
    &= \int_0^t \int_{\mathbb{S}^2} \int_0^\infty p(s,\omega r)\frac{\delta(s-r)}{4\pi s}f(t-s,x-r\omega)r^2 drd\omega ds
    \nn
    &= \int_0^t \int_{\mathbb{S}^2} p(s,\omega s)\frac{1}{4\pi }f(t-s,x-r\omega)s d\omega ds
    \nn
    &= \int_0^t\int_{\mathbb{S}^2} \frac{p(1,\omega)}{4\pi}f(t-s,x-s\omega)s^{1+m}d\omega ds.  
\end{align*}
Then it is easy to conclude
\begin{align*}
    |\bar{u}(t,x)| \leq \frac{t^{1+m}}{4\pi} \, |\mathbb{S}^2| \, ||p(1,\cdot)||_{L^{\infty}(\mathbb{S}^2)}\int_0^t||f(s,\cdot)||_{L^\infty(\mathbb{R}^3)}ds.
\end{align*}
Given $|\mathbb{S}^2| = \frac{4\pi}{3}$, we are done.
\end{proof}
The next lemma allows us to commute the time derivative in (\ref{equations for E and B}) onto the distribution $Y$. Remark the challenge is to pass $\partial_t$ through the characteristic function $\mathbbm{1}_{t>0}$.

\begin{lemma} \label{time derive on step function lemma}
For $f\in W^{1,\infty}(\mathbb{R}^4;\mathbb{R})$ we have the identity
\begin{align}
    \partial_t(Y*_{t,x}\mathbbm{1}_{t>0}f) = Y *\mathbbm{1}_{t>0}\partial_t f + \frac{t}{4\pi}\int_{\mathbb{S}^2}f(0,x-t\omega)d\omega.
\end{align}
\end{lemma}
\begin{proof}
First note, from Lemma \ref{wave integral lemma} with $p\equiv 1 \in \mathcal{M}_0(\mathbb{R}^4)$, we have 
\begin{align*}
    Y*_{t,x}\mathbbm{1}_{t>0}f 
    &=\int_{\mathbb{S}^2}\int_0^t \frac{s}{4\pi}f(t-s,x-s\omega)dsd\omega,
\end{align*}
therefore 
\begin{align*}
    \partial_t(Y*_{t,x}\mathbbm{1}_{t>0}f) &= \int_{\mathbb{S}^2}\int_0^t \frac{s}{4\pi}\partial_t f(t-s,x-s\omega)ds d\omega + \int_{\mathbb{S}^2}\frac{t}{4\pi}f(0,x-t\omega)d\omega
    \nn &=Y *\mathbbm{1}_{t>0}\partial_t f + \frac{t}{4\pi}\int_{\mathbb{S}^2}f(0,x-t\omega)d\omega.
\end{align*}
\end{proof}
Lemmas \ref{division lemma}, \ref{wave integral lemma} and \ref{time derive on step function lemma} then allow us to manipulate (\ref{equations for E and B}) as follows  
\begin{align} \label{E pass computation}
    Y *_{t,x}\mathbbm{1}_{t>0}\int &[v(\xi) \partial_tf + \nabla_xf]d\xi 
    \nn
    &=
    \int (-T(pY) + qY)*_{t,x}\mathbbm{1}_{t>0}fd\xi - \frac{t}{4\pi}\int\int_{\mathbb{S}^2}v(\xi)f(0,x-t\omega,\xi)d\omega d\xi
 \nn
 &= -\int pY *_{t,x} \mathbbm{1}_{t>0}T(f)d\xi - \frac{t}{4\pi}\int\int_{\mathbb{S}^2}p(1,\omega,\xi)f(0,x-t\omega,\xi)d\omega d\xi
 \nn
 &\quad +\int qY*_{t,x}\mathbbm{1}_{t>0}fd\xi - \frac{t}{4\pi}\int\int_{\mathbb{S}^2}v(\xi)f(0,x-t\omega,\xi)d\omega d\xi.
\end{align}
Remark that the term in (\ref{K1}) involving $J(f^{in})$ can be written
\begin{align*}
\frac{1}{4\pi t^2}\int_{|y-x|=t}tJ(f^{in})dS(y) &= \frac{1}{4\pi t}\int \int_{|x-y|=t} v(\xi) f^{in}(y,\xi) dS(y)d \xi 
\nn
&=\frac{t}{4\pi}\int \int_{\mathbb{S}^2}v(\xi)f^{in}(x-t\omega,\xi)d\omega d\xi
\end{align*}
This cancels with the last term in  (\ref{E pass computation}). A similar computation for $B$, leads to a wonderful representation formula for the fields $(E,B)$:
\begin{align} \label{Full E B solution}
    E(t,x) &=-\int p(t,x,\xi)Y(t,x) *_{t,x}(\mathbbm{1}_{t>0}T(f))d\xi + \int q(t,x,\xi)Y(t,x) *_{t,x}(\mathbbm{1}_{t>0}f)d\xi
    \nn
     &+\frac{1}{4\pi t^2}\int_{|x-y|=t}\bigg[t\nabla_x\times B^{in}(y) + E^{in}(y) + [(y-x)\cdot\nabla_y]E^{in}(y)\bigg]dS(y) 
     \nn
     &- \frac{t}{4\pi}\int\int_{\mathbb{S}^2}p(1,\omega,\xi)f^{in}(x-t\omega,\xi)d\omega d\xi,
    \\ \label{full B solution}
    B(t,x) &= \int a^0(t,x,\xi)Y(t,x) *_{t,x}(\mathbbm{1}_{t>0}T(f))d\xi + \int a^1(t,x,\xi)Y(t,x) *_{t,x}(\mathbbm{1}_{t>0}f)d\xi
    \nn
    &+\frac{1}{4\pi t^2}\int_{|x-y|=t} \bigg[-t\nabla_x \times E^{in}(y) + B^{in}(y) + [(y-x)\cdot\nabla_y]B^{in}(y) \bigg]dS(y)
    \nn
    &+ \frac{t}{4\pi}\int\int_{\mathbb{S}^2}a^0(1,\omega,\xi)f^{in}(x-t\omega,\xi)d\omega d\xi.
\end{align}

%%%%%%%%%%%%%%

\subsection{Obstruction to Uniform Estimates} \label{obstruction section}
The difficulty for obtaining uniform in $\epsilon$ estimates of the fields comes from the first terms in (\ref{Full E B solution})-(\ref{full B solution}). That is when we replace the $T(f)$ using the Vlasov equation, this introduces the term of order $\epsilon^{-1}$, coming from the applied field. We will only consider computations for $E$ and simply state the final results for $B$ as they are similar. Using Lemma \ref{wave integral lemma}, the estimate shown in \cite{cold} for $p$ and $q$ with $|\xi|\leq R_\xi^T$ are given by
\begin{align} \label{p estimate}
||p(1,\cdot,\xi)||_{L^\infty(\mathbb{S}^2)} &\leq  2\frac{\sqrt{1 + (R_\xi^T)^2}}{\sqrt{1+(R_\xi^T)^2} - (R_\xi^T)} = 2\big(1 + (R_\xi^T)^2 + (R_\xi^T)\sqrt{1 + (R_\xi^T)^2}\big) < \infty,
\end{align}
and 
\begin{align} \label{q estimate}
    ||q(1,\cdot,\xi)||_{L^\infty(\mathbb{S}^2)} &\leq 2\frac{1 + (R_\xi^T)^2}{(\sqrt{1+(R_\xi^T)^2} - (R_\xi^T)^2} = 2\big(1 + (R_\xi^T)^2 + (R_\xi^T)\sqrt{1 + (R_\xi^T)^2}\big)^2 <\infty.
\end{align}
Remark, the non-integrability of $p(1,\omega,\xi)$ and $q(1,\omega,\xi)$ (in $L^1$) in the variable $\xi$ is the main difficulty in closing the open well-posedness RVM problem for large data.  We then immediately obtain the estimate for the field $E$
\begin{align} \label{E estimate 3}
    |E(t,x)| &\leq C(t,R_x^0,R_\xi^0, E^{in}, B^{in}, f^{in}) + \int_{|\xi|\leq R_\xi^T} ||q(1,\cdot,\xi)||_{L^\infty(\mathbb{S}^2)}d\xi \int_0^t ||f(s,\cdot,\cdot)||_{L^\infty_{x,\xi}}ds
    \nn
    &+ \bigg| \int p(t,x,\xi)Y(t,x) *_{t,x}(\mathbbm{1}_{t>0}T(f))d\xi \bigg| .
\end{align}
The idea to estimate this remaining term and apply Gr\"{o}nwall's lemma is to pass the derivative $T(f)$ to the Vlasov equation and integrate by parts in $\xi$ as follows
\begin{align} \label{good equation for E}
    &\int p(t,x,\xi)Y(t,x) *_{t,x}(\mathbbm{1}_{t>0}T(f))d\xi 
    \nn
    & \quad =
    \int p(t,x,\xi)Y(t,x) *_{t,x}\bigg(\mathbbm{1}_{t>0}\big\{\nabla_\xi\cdot ( [\epsilon E + v(\xi)\times (\epsilon B + \epsilon^{-1}\mathbf{B}_e)]f) + \epsilon M' (|\xi|)\frac{\xi}{|\xi|}\cdot E\big\}\bigg)d\xi
     \nn
     &\quad = 
     - \epsilon \int \nabla_\xi p(t,x,\xi)Y(t,x) *_{t,x}(\mathbbm{1}_{t>0}[ E + v(\xi)\times B ]f)d\xi
     \nn
     &\quad \quad  -\epsilon^{-1}\int \nabla_\xi p(t,x,\xi)Y(t,x) *_{t,x}(\mathbbm{1}_{t>0} [v(\xi)\times \mathbf{B}_e]f)d\xi
     \nn
     & \quad\quad  + \epsilon\int p(t,x,\xi)Y(t,x) *_{t,x}(\mathbbm{1}_{t>0}  M'(|\xi|)\frac{\xi}{|\xi|}\cdot E)d\xi.
\end{align}
Similarly for $B$. This is now in a suitable form to apply Gr\"{o}nwall estimates (after applying Lemma \ref{wave integral lemma} one more time of course), provided the solution $f$ has compact support in $\xi$ which allows the use of the estimates (\ref{p estimate})-(\ref{q estimate}). More specifically we apply a non-linear Gr\"{o}nwall estimate known as the Bihari-LaSalle inequality due to the quadratic term $[E+v\times B] f$. See for instance appendix A of \cite{thesis}. Assuming $f$ remains bounded in $L^\infty_{x,\xi}$, we do have the fields $(E,B)$ are uniformly bounded in $L^\infty_x$ with respect to $\epsilon$, but only on a time interval $T_\epsilon>0$ (the maximal lifetime of solutions), which may shrink to zero as $\epsilon$ tends to zero. Unlike the cold case in \cite{cold}, at this stage, it is not apparent that one can achieve uniform estimates on a times interval $T_\epsilon = \mathcal{O}(1)$ due to the penalization $\epsilon^{-1}\mathbf{B}_e(x)$ (see Remark \ref{cold remark} below). Therefore we pay special attention to the term
\begin{align*}
    -\epsilon^{-1}\int \nabla_\xi p(t,x,\xi)Y(t,x) *_{t,x}(\mathbbm{1}_{t>0} [v(\xi)\times \mathbf{B}_e]f)d\xi,
\end{align*}
which posses the very rough estimate using Lemma \ref{wave integral lemma}
\begin{align}
    \bigg|\epsilon^{-1}&\int \nabla_\xi p(t,x,\xi)Y(t,x) *_{t,x}(\mathbbm{1}_{t>0} [v(\xi)\times \mathbf{B}_e]f)d\xi\bigg|
    \nn
    &\leq \frac{t}{4\pi \epsilon}||\mathbf{B}_e(\cdot)||_{L^\infty({\{|x-y|\leq t\} })} \int_{|\xi|\leq R_\xi^t}||\nabla_\xi p(1,\cdot,\xi)||_{L^\infty(\mathbb{S}^2)}d\xi \int_0^t ||f(s,\cdot,\cdot)||_{L^\infty_{x,\xi}}ds.
\end{align}
For instance, assume a solutions $f(t)$ which is uniformly bounded in $L^\infty_{x,\xi}$ with respect to $\epsilon$ exists. The above estimate could imply a growth $|(E,B)(t)| = \mathcal{O}(\epsilon^{-1} t)$ and therefore uniform sup-norm estimates can only be achieved on a time interval $T_\epsilon$, with $T_\epsilon\rightarrow 0$ with $\epsilon$. Section \ref{non linear approx} is devoted to overcoming this difficulty of achieving a uniform lifetime $0 < T < T_\epsilon$ in the hot regime. To accomplish this, we also require representation formulas for the Vlasov equation using the method of characteristics. This is done in the next section. To take full advantage of the fast oscillations of the characteristics we conclude Section \ref{field straigtening section} by constructing a canonical set of coordinates which simplifies the analysis of Section \ref{non linear approx} by introducing a fast periodic variable for the characteristics.

\begin{remark} \label{cold remark}
In \cite{cold}, the cold assumption leads to the replacement of $p(\cdot,\cdot,\xi)$ with $p_\epsilon(\cdot,\cdot,\xi)$ given by the relationship $p_\epsilon(\cdot,\cdot,\xi) := p(\cdot,\cdot,\epsilon \xi)$ and hence $\nabla_\xi p$ is replaced with $\epsilon \nabla_\xi p_\epsilon$ which compensates the term $\epsilon^{-1}\mathbf{B}_e(x)$ allowing for uniform bounds in $L^\infty$ of $(E,B)$.
\end{remark}
\begin{remark}
One does however have the estimate
\begin{align} \label{dilute E estimate}
|\epsilon E(t,x)| &\leq \epsilon C +  C\int_0^t (1+\epsilon)||f(s,\cdot,\cdot)||_{L^\infty_{x,\xi}}ds + C||M'||_{L^\infty}\int_0^t ||\epsilon (E,B)(s,\cdot)||_{L^\infty_x}ds 
\nn
&+C\int_0^t ||\epsilon (E,B)(s,\cdot)||_{L^\infty_{x}}||\epsilon f(s,\cdot,\cdot)||_{L^\infty_{x,\xi}}ds,
\end{align}
where the constant $C$ depends on the  initial data, on $||\mathbf{B}_e||_{L^\infty}$, on the momentum support $\{ |\xi|\leq R_\xi^T\}$ of $f$ and on  
$||(\nabla_\xi p,q)(1,\cdot,\cdot)||_{L^\infty(\mathbb{S}^2\times \{ |\xi|\leq R_\xi^T\})}$. 
\end{remark}
In the next section we will derive representation formulas for the Vlasov equation.

%%%%%%%%%%%%%%%%%%%%%%%%%%%%%%%%%%%%

\subsection{Vlasov Representation Formula and Uniform Estimates of \texorpdfstring{$ ||(f,\epsilon E,\epsilon B)(t)||_{L^\infty_{x,\xi}}$}{TEXT}} \label{vlasov rep section}
The approach to solving the Vlasov Equation, a transport equation, is through the method of characteristics. Consider the ODE system, depending on given fields $(E,B)$,  defined as solutions of
\begin{align} \label{first x characteristic} 
    \dot{X} &= v(\Xi), &&X(0,x,\xi) = x,
    \\
    \dot{\Xi} &= -\epsilon^{-1}v(\Xi)\times \mathbf{B}_e(X) - \epsilon [E(t,X) + v(\Xi)\times B(t,X)], &&\Xi(0,x,\xi) = \xi. \label{first xi charactersitc}
\end{align}
Remark we always work on a time interval $t\in [0,T_\epsilon]$, where $T_\epsilon$ is the maximal lifetime of $(f,E,B)$. Then for as long as the solution $(X,\Xi)(t) := (X,\Xi)(t,x,\xi)$  exists (here we omit the dependence  on $(x,\xi)$ in our notation), it follows that
\begin{align} \label{bar f along flow}
    \dt{t}  f(t,X(t),\Xi(t)) =  \epsilon M' (|\Xi(t)|)\frac{\Xi}{|\Xi|}\cdot E(t,X).
\end{align}
Thus we must justify the flow map defined by
\begin{align*}
    \mathcal{F}_t : \mathbb{R}^3\times\mathbb{R}^3 &\mapsto \mathbb{R}^3\times\mathbb{R}^3
    \nn
    (x,\xi) &\mapsto (X(t,x,\xi),\Xi(t,x,\xi))
\end{align*}
is invertible up to some time $t$. First remark that $|\dot{X}|<1$ and therefore
\begin{align} \label{X1 estimate}
    |X(t) - x| < t.
\end{align}
Next we compute
\begin{align*}
    \dt{t}|\Xi|^2 = \Xi\cdot\dot{\Xi} = \epsilon \Xi\cdot E(t,X(t)).
\end{align*}
Therefore the Bahari-LaSalle inequality implies
\begin{align} \label{Xi charactersitic estimate}
    |\Xi|(t,x,\xi) \leq |\xi| + \epsilon C\int_0^t ||E(s,\cdot)||_{L^\infty(|x-y|\leq t)}ds.
\end{align}
Therefore as long as $||\epsilon E(s,\cdot)||_{L^\infty(|x-y|\leq t)} < \infty$ it follows that $|\Xi(t)| < \infty$. Therefore the characteristics $(X,\Xi)(t)$ remain in a compact set for any finite $t$. Furthermore, we have the right hand side of the vector field (\ref{first x characteristic})-(\ref{first xi charactersitc}) is divergence free $\nabla_{X,\Xi}\cdot(\dot{X},\dot{\Xi})\equiv 0$, and therefore the flow $\mathcal{F}_t$ is a volume preserving diffeomorphism. Thus the Duhamel Principal on (\ref{bar f along flow}) implies
\begin{align*}
     f(t,x,\xi) =  f^{in}(X(-t),\Xi(-t)) + \epsilon \int_0^t\bigg[M'(|\xi|)\frac{\xi}{|\xi|}\cdot E\bigg](s,X(t-s),\Xi(t-s))ds.
\end{align*}
Note that (\ref{X1 estimate}) implies $|X(t-s) - x| \leq t-s \leq t$ for $s\in [0,t]$. This then gives the immediate estimate
\begin{align} \label{first f estimate}
    | f(t,x,\xi)|\leq || f^{in}||_{L^\infty_{x,\xi}} +  ||M' ||_{L^\infty_\xi}\int_0^t ||\epsilon E(s,\cdot)||_{L^\infty(|x-y|\leq t)}ds.
\end{align}

\begin{lemma} \label{rough E estimate}
The estimate (\ref{epsilon E estimate}) holds.
\end{lemma}
\begin{proof}
    We simply add (\ref{dilute E estimate}) to  (\ref{first f estimate}) and apply Gr\"{o}nwall's (Bihari-LaSalle) Lemma. This gives gives for all $t \in [0,T]$
\begin{align}
    ||(f,\epsilon E,\epsilon B)(t,\cdot,\cdot)||_{L^\infty_{x,\xi}(\mathcal{A}_T)} \leq C(T, R_\xi^T, ||\mathbf{B}_e||_{L^\infty}).
\end{align}
 Note the remaining terms of (\ref{good equation for E}) that do not involve $\epsilon^{-1}$ are controlled the same as in \cite{cold}.
\end{proof}

\smallskip \noindent
 Finally, the estimate (\ref{epsilon E estimate}) and the arguments used in section (\ref{uniform time}) to extend solutions
 guarantee the uniform time of existence for ill-prepared data. This completes the proof of Theorem \ref{main theorem 1}.  
 
 \smallskip
 
 Remark that we required $M_\epsilon = \epsilon M = \mathcal{O}(\epsilon)$ in order to apply  Gr\"{o}nwall's lemma to (\ref{first f estimate}) and thus estimate $f$ in terms of $\epsilon E$. Thus, at this stage it is not apparent how even weighted, uniform estimates should be obtained when the system is not dilute. Moreover, these estimates do not show how one could remove the weight of $\epsilon$ to achieve uniform Sup-norm estimates of the fields. In this article we only address the latter issue. Before that, we will consider a canonical set of coordinates though a field straightening procedure. This will involve a rotation of the applied magnetic field to align with the $x_3$-axis. The advantage is to introduce a single oscillatory variable $\theta$ in cylindrical coordinates as the characteristic curve trajectories wrap around the $x_3$-axis. 

%%%%%%%%%%%%%%%%%%%%%%%%%%%%%%%%%%

\subsection{Field Straightening} \label{field straigtening section}
\smallskip \noindent
As mentioned it will be convenient to work with a single oscillatory variable. To do this, we will rotate our system in the following way. Let $O : \mathbb{R}^3 \mapsto SO(3)$ be a map defined by the relation
\begin{align*}
O^t(x) \mathbf{B}_e(x) = b_e(x) (0,0,1)^t.
\end{align*}
Remark the superscript $t$ is used to denote a matrix transpose and should not be confused with time. Thus, $O^t$ is a rotation by angle $\vartheta(x)\in [0,2\pi)$ defined by $\cos(\vartheta(x)) := \mathbf{B}_e^3(x) / b_e(x) $ about the axis $\mathbf{B}_e^\perp := (\mathbf{B}_e^2(x),-\mathbf{B}^1_e(x),0)^t = \mathbf{B}_e\times e_3$. Clearly when $\mathbf{B}_e^\perp(x) \equiv 0$, we take $\vartheta(x) =0$. Recall our assumption (\ref{lowupboundve}) that $b_e >0$. So more precisely, $O^t(\cdot)$ is determined by Euler-Rodrigues' formula 
\begin{align} \label{O exact}
O^t(x) := \frac{\mathbf{B}^3_e(x)}{b_e(x)}I_{3\time 3} + \frac{|\mathbf{B}_e^\perp(x)|}{b_e^2(x)}[\mathbf{B}_e\times] + \big(1 - \frac{\mathbf{B}^3_e}{b_e(x)}\big)\frac{\mathbf{B}_e^\perp \otimes \mathbf{B}_e^\perp}{b_e^2(x)},
\end{align}
with the cross product matrix and usual Euclidean outer product
\begin{align*}
[\mathbf{B}_e\times] := \begin{bmatrix} 0  &-\mathbf{B}_e^3 &\mathbf{B}_e^2 \\ \mathbf{B}_e^3&0&-\mathbf{B}_e^1 \\-\mathbf{B}_e^2 &\mathbf{B}_e^1&0 \end{bmatrix}, \quad \mathbf{B}_e^\perp \otimes \mathbf{B}_e^\perp := \mathbf{B}_e^\perp  (\mathbf{B}_e^\perp)^t = \begin{bmatrix} (\mathbf{B}_e^2)^2 & - \mathbf{B}_e^1\mathbf{B}_e^2 & 0 \\ - \mathbf{B}_e^1\mathbf{B}_e^2& (\mathbf{B}_e^1)^2 &0 \\ 0&0&0\end{bmatrix}.
\end{align*}
The precise construction of $O(x)$ is not of high importance, but retain that it is a smooth, rational function of the components of $\mathbf{B}_e$ with matrix norm $||O^t||_{L^\infty} = 1$. Next define a new distribution function $\bar{f}$ according to the following variable change
\begin{align}\label{OOOOcha}
    \bar{f}(t,x,\xi) := f(t,x,O(x)\xi).
\end{align}
It follows that $\bar{f}$ is a solution of
\begin{align}\label{fullsystem}
&\partial_t \bar{f} + v(O(x)\xi)\cdot \nabla_x \bar{f} - O^t(x)\nabla_x(O(x)\xi) v(O(x)\xi)\cdot \nabla_\xi \bar{f} -\epsilon^{-1}\frac{b_e(x)}{\vl{\xi}}\partial_\theta \bar{f} \nn
&\qquad \qquad \qquad =  \epsilon [O^t(x)E + v(\xi)\times O^t(x)B]\cdot \nabla_\xi \bar{f} +  \epsilon \frac{M'(|\xi|)}{|\xi|}O(x)\xi\cdot E,   \\
&\bar{f}(0,x,\xi) = f^{in}(x,O(x)\xi) := \bar{f}^{in}(x,\xi),
\end{align}
where 
\begin{align} \label{partial theta defin}
\partial_\theta := \xi_2\partial_{\xi_1} - \xi_1\partial_{\xi_2}  = [\xi \times O^t(x)\frac{\mathbf{B}_e(x)}{b_e(x)}]\cdot \nabla_\xi   = \BM{0 & 1 & 0 \\ -1 &0&0 \\ 0&0&0}\xi \cdot \nabla_\xi.
\end{align}
For now, we may think of $\partial_\theta$ defined above to be given in a Cartesian coordinate system as in the far right expression of (\ref{partial theta defin}). Later we will convert our new characteristic curves to a cylindrical coordinate system and the notation will become clear. Furthermore, to be unambiguous, the components of the matrix $\nabla_x(O(x)\xi)$ are defined by
\begin{align*}
[\nabla_x(O(x)\xi)]_{i,j} := \sum_{k=1}^3\partial_{x_j}O_{ik} (x) \, \xi_k, \quad (i,j)\in \{1,2,3\}^2.
\end{align*}
This convention will be used whenever we write the gradient of a vector valued function. Note that because $\det(O(x)) = 1$ for all $x\in\R^3$, it follows that the charge and current density become
\begin{align}
    \rho(f)(t,x) &=\rho(\bar{f})(t,x) =\int \bar{f}(t,x,\xi)d\xi,
    \nn
    J(f)(t,x) &= \int v(O(x)\xi)\bar{f}(t,x,\xi)d\xi.
\end{align}
So the compatibility conditions (\ref{support consdtion}), (\ref{neutralityazero}) and  (\ref{initial data 2}) are satisfied for $\bar{f}^{in}$ as well. The characteristic curves of (\ref{fullsystem}) are defined as solutions of
\begin{align}
\label{Xode} \dot{X} &= v(O(X)\Xi)  &X(0) &= x, \\
\label{endXode} \dot{\Xi} &= \vl{\Xi}^{-1} Q(X,\Xi) - \epsilon^{-1}v(\Xi)\times O^t(X)\mathbf{B}_e(X) 
\nn
&\qquad \qquad \quad \ -  \epsilon \, [O^t(x)E(t,X) + v(\Xi)\times O^t(x)B(t,X)] &\Xi(0) &= \xi,
\end{align}
where for more compact notation we have set the quadratic in $\xi$ term $Q$ to be given by
\begin{align} \label{Q formula}
Q(x,\xi):= - O^t(x)\nabla_x(O(x)\xi) O(x)\xi.
\end{align}
See remark \ref{Q formula remark} below for the derivation of $Q$. Note that the transformation $(t,x,\xi)\mapsto (t,x,O(x)\xi)$ is volume preserving with respect to $dxd\xi$ for all $t$. So it is expected that the flow
\begin{align} \label{exact flow}
\mathcal{F}_t(x,\xi) &:= (X(t,x,\xi),\Xi(t,x,\xi)),
\end{align}
should also preserve volume. The following lemma guarantees that this will be the case for any transformation $\eta$ with Jacobian one. 

\begin{lemma} Let $F \in C^1(\R^n;\R^n) $ be such that $ \nabla \cdot F \equiv 0 $. Suppose that $f:\R^n\mapsto \R$  satisfyes  
\begin{align}
    F(x)\cdot\nabla f(x) = 0 .
\end{align}
Consider a variable change $  y := \eta(x) $ with $ |D_x \eta| = 1 $, and define
\begin{align*}
    \tilde{f}(y) := \tilde{f}(\eta(x)) = f(x).
\end{align*}
Then it follows that 
\begin{align} \label{new variable change pde}
     \tilde{F}(y) := [(D_x\eta)^tF\circ (\eta^{-1}(y))], \quad  \tilde{F}(y)\cdot \nabla_y \tilde{f}(y) = 0,
\end{align}
with the convention that
\begin{align} \label{jacobian definition}
    (D_x\eta)_{i,j} := J_{i,j} = \frac{\partial \eta_i}{\partial x_j}, \quad 
    J^{-1}_{i,j} = \frac{\partial x_i}{\partial \eta_j}.
\end{align}
Moreover the following divergence free property is preserved for any constant Jacobian transform
\begin{align} \label{ divergence formula change}
     \nabla_y \cdot \tilde{F}(y(x)) = \nabla_x \cdot F(x) + F(x)\cdot \nabla_x \ln(|D_x\eta|(x)) = 0.
\end{align}
\end{lemma}
\begin{proof}
Using index notation (while not distinguishing between upper and lower indices) we have by the chain rule
\begin{align*}
0 &= F(x)\cdot \nabla f(x) = F_i(x)\frac{\partial f}{\partial x_i} = F_i(x(y)) \frac{\partial y_j}{\partial_{x_i}}\frac{\partial \tilde{f}}{\partial y_j}
\nn
&=\big[\frac{\partial y_j}{\partial x_i} F_i \big](x(y))\frac{\partial \tilde{f}}{\partial y_j}(y),
\end{align*}
which is exactly (\ref{new variable change pde}). Consider next the $y$-divergence of $\tilde{F}$
\begin{align*}
\nabla_y \cdot \tilde{F}(y) = \partial_{y_j}\big[\frac{\partial y_j}{\partial_{x_i}} F_i \big](x(y)) = \frac{\partial F_i}{\partial x_k}\frac{\partial x_k}{\partial y_j}\frac{\partial y_j}{\partial x_i} + F_i \frac{\partial}{\partial y_j}\big[\frac{\partial y_j}{\partial x_i}\big].   
\end{align*}
Then using (\ref{jacobian definition}) it follows that
\begin{align*}
\nabla_y \cdot \tilde{F}(y) &= \frac{\partial F_i}{\partial x_k} \delta_i^k + F_i \frac{\partial x_k}{\partial y_j} \frac{\partial^2 y_j}{\partial x_k \partial x_i}
\nn
&= \frac{\partial F_i}{\partial x_i} + F_i \big[ \frac{\partial y_j}{\partial x_k}\big]^{-1} \frac{\partial }{\partial x_i}\big[ \frac{\partial y_j}{\partial x_k}\big]
\nn
&= \nabla \cdot F + F_i \text{Tr}\big( J^{-1} \partial_{x_i} J\big).
\end{align*}
Then Jacobi's Formula gives that for any invertible matrix $A(t)$ we have 
\begin{align*}
    \partial_t \ln(|A|) = \frac{\partial_t |A|}{|A|} = \text{Tr}\big(A^{-1} \partial_t A \big).
\end{align*}
Thus we finally arrive at
\begin{align*}
    \nabla_y \cdot \tilde{F}(y) = \nabla_x \cdot F(x(y)) + F(x(y))\cdot \nabla_x \ln (|D\eta|)(x(y)).
\end{align*}

\end{proof}

\begin{remark} \label{Q formula remark}
Note that in our case we use the transformation $(t,x,\tilde{\xi}) := (t,x,O^t(x)\xi)$, so that $\bar{f}(t,x,\tilde{\xi}) = f(t,x,\xi)$, given by (\ref{OOOOcha}), and the term $Q$ comes from
\begin{align*}
    \frac{Q_j}{\vl{\xi}} &= \frac{\partial \tilde{\xi}_j}{\partial x_i}v_i(\xi) = \partial_{x_i}(O^t_{j,k}\xi_{k})\frac{\xi_i}{\vl{\xi}} 
    \nn
    &= \partial_{x_i}(O^t_{j,k})O_{k,\ell}\tilde{\xi}_{\ell}\frac{O_{i,m}\tilde{\xi}_m}{\vl{\tilde{\xi}}}
    \nn
    &= \underbrace{\partial_{x_i}(O^t_{j,k}O_{k,\ell}\tilde{\xi}_\ell)}_{=0}\frac{O_{i,m}\tilde{\xi}_m}{\vl{\tilde{\xi}}} - O^t_{j,k}\partial_{x_i}(O_{k,\ell}\tilde{\xi}_\ell) \frac{O_{i,m}\tilde{\xi}_m}{\vl{\tilde{\xi}}}
    \nn
    &=-\big[O^t(x)\nabla_x (O(x)\tilde{\xi})v(O(x)\tilde{\xi})\big]_j,
\end{align*}
where it is clear that $\vl{\xi} =\langle \tilde{\xi} \rangle $. This is exactly the equation given by (\ref{Q formula}). Furthermore $Q$ is orthogonal to $\xi$
\begin{align*}
    - \vl{\xi} \xi \cdot Q &= \xi_i O^t_{i,j}\partial_{x_k}(O_{j,\ell}\xi_\ell)O_{k,m}\xi_m
    \nn
    &= \partial_k(\xi_iO_{i,j}^t O_{j,\ell}\xi_\ell) O_{k,m}\xi_m - \partial_k(\xi_iO^t_{i,j})O_{j,\ell}\xi_\ell O_{k,m}\xi_m
    \nn
    &=  \underbrace{\nabla_x(O(x)\xi\cdot O(x)\xi)}_{= \nabla_x (\xi \cdot \xi) = 0} \cdot O(x)\xi  - \xi_\ell O^t_{\ell,j}\partial_k(O_{j,i}\xi_i)O_{k,m}\xi_m
    \nn
    &= \vl{\xi}\xi\cdot Q,
\end{align*}
where we have relabeled $\ell$ and $j$ in the last lines implying that $\xi \cdot Q \equiv 0$.
\end{remark}

This leads to the immediate corollary:
\begin{corollary}
The rotated flow of (\ref{Xode})-(\ref{endXode}) preserves volume with respect to the Liouville measure $dxd\xi$ as long as the characteristics do not cross.
\end{corollary}
Similarly to the non-rotated flow, the solution of (\ref{fullsystem}) will exist up to time $T>0$ provided the characteristics remain in a compact set up to time 
$T$. Suppose that $(E,B) \in C^1([0,T]\times\mathbb{R}^3\times\mathbb{R}^3)$ is a classical solution. It follows that for $t \leq T$
\begin{align*}
|\dot{X}| \leq 1 \implies |X(t,x,\xi)| \leq R_x^0 + T.
\end{align*}
Next consider the pointwise estimate of $|\Xi|$ using remark \ref{Q formula remark} that $\xi\cdot Q = 0$,
\begin{align}
\partial_t |\Xi|^2 = 2 \Xi\cdot\partial_t \Xi = \Xi\cdot  \epsilon O^t E  &\leq  \epsilon |\Xi|  ||E||_{L^\infty}. \label{xi charac}
\end{align}  
Then integrating gives
\begin{align}
|\Xi|^2 \leq (R_\xi^0)^2 +  \epsilon \int_0^t ||E(s,\cdot)||_{L^\infty}|\Xi(s)|ds. \label{xi bounded}
\end{align}
 If $ \epsilon E \in L^\infty([0,T]\times \{|x|\leq R_x^0\})$, then in fact , $|\Xi|$ can be controlled by the Bahari LaSalle inequality which gives us the estimate
 \begin{align*}
 |\Xi(t)|^2 \leq (R_\xi^0)^2(1 + Cte^{Ct}). 
 \end{align*}
 This means that the characteristics (\ref{Xode}) - (\ref{endXode}) remain in a bounded set on $[0,T]$ and are thus globally defined. 
 
 %%%%%%%%%%%%%%%%%%%%%%%%%%%%
 %%%%%%%%%%%%%%%%%%%%%%%%%%%%

 \newpage
 \section{Proof of Theorem \ref{main theorem 2}} \label{non linear approx}

\indent This section is devoted to the proof of Theorem \ref{main theorem 2}. Section \ref{Asymptotic apprx section} begins by considering an external  
inhomogeneous magnetic field orientated along a fixed direction. Furthermore, we study a linearized version of the Vlasov Maxwell system and derive 
an asymptotic approximation of the associated characteristics in terms of $\epsilon$. This approximation is given by Lemma \ref{X lemma} and is accomplished 
using a strategy similar to the methods of \cite{fastrotations}, involving a non-stationary phase argument for the rapidly oscillating characteristics. This Lemma is 
essential. The general procedure for applied fields with variable direction is handled in the appendix. In Section \ref{linear inhomog section} we prove the well 
posedness of the linear system with respect to uniform Sup and Lipschitz-norms. For the linear system, the Sup-norm is uniform in $\epsilon$, while a weight 
$\epsilon$ is necessary for a uniform Lipschitz norm unless the data is well prepared in the sense of Definition \ref{prepared data}. When the direction of the 
magnetic field is fixed, uniform estimates of the linear system are explicit. For demonstration, we leave the general case to the appendix with the inclusion of 
Lemma \ref{bad term change of variables}. Finally, Section \ref{proof of theorem 1 section} uses the linear system described in 4.1 to prove the estimates in Theorem \ref{main theorem 2}. 
The linear solution serves as a good $\mathcal{O}(\epsilon)$-approximation of $f$ in the Sup-norm, while only an $\mathcal{O}(1)$ approximation of the fields $(E,B)$, 
which is still enough to deduce a priori estimates for Theorem \ref{main theorem 2}. We also estimate the derivatives on the fields to ensure the characteristic equations 
for $f$ can be solved (via the Picard Lindel\"{o}f Theorem). Finally in Section \ref{uniform time}, under the Glassey Strauss assumption, using these a priori estimates we 
show solutions exist on a uniform time interval $0<T<T_\epsilon$. In other words, we establish well posedness on a uniform time interval $[0,T]$ of the HMRVM system 
for dilute equilibrium and well prepared data. 

\subsection{Asymptotics of Characteristic Curves of Linear System} \label{Asymptotic apprx section}
\smallskip
For simplicity, we first consider the case of an inhomogeneous, magnetic field with constant direction aligned along the $x_3$-axis and leave the general case for the appendix. Therefore assume
\begin{align*}
    \mathbf{B}_e(x) = b_e(x) \, {}^t (0,0,1) .
\end{align*}
This implies $Q\equiv 0$ and $O(x) = Id_{3\times 3}$ and $\bar{f} \equiv f$. The goal will be to first study the dilute, linearized system in an inhomogeneous magnetic field with 
fixed direction.  We define the linear system by dropping the non-linear term of order $\epsilon$ from (\ref{vlasov 1}), namely:
\begin{align} \label{linear dilute approx}
\begin{cases}
&\partial_t f_{\ell} + v(\xi)\cdot\nabla_x f_{\ell} - \epsilon^{-1} \, \vl{\xi}^{-1} \, b_e(x) \, \partial_\theta f_{\ell} =  \epsilon M'(|\xi|) \, |\xi|^{-1} \, \xi \, \cdot E_\ell
 \\
&\partial_t E_\ell -\nabla_x \times B_\ell = J(f_\ell), \quad \nabla_x\cdot E_\ell = -\rho(f_\ell)
\\
&\partial_t B_\ell + \nabla_x\times  E_\ell = 0, \ \ \ \qquad   \nabla_x \cdot B_\ell = 0
\end{cases}
\end{align}
together with
\begin{align}
(f_\ell,E_\ell,B_\ell)|_{t=0} = (f^{in}, E^{in}, B^{in}) .
\end{align}

Furthermore, consider the characteristic curves $(X_\ell,\Xi_\ell)(t,x,\xi) = (X_\ell,\Xi_\ell)(t)$ of the linearized system solving
\begin{align} \label{inhomogeneous linear characteristics }
     \dot{X}_{\ell} &= \frac{\Xi_\ell}{\vl{\Xi_\ell}}, & &X_\ell(0) = x,
     \nn
     \dot{\Xi}_\ell &= - \frac{b_e(X_\ell)}{\epsilon\vl{\Xi_\ell}} \, {}^t (\Xi_{\ell 2} , - \Xi_{\ell 1} , 0 ) , &&\Xi_\ell(0) = \xi. 
 \end{align}
Then the solution $f_\ell$ can be expressed using Duhamel's principal in terms of these characteristic curves as
\begin{align*}
    f_\ell(t,x,\xi) = f^{in}(X_\ell(-t),\Xi_\ell(-t)) + \epsilon \int_0^t\bigg[ M'(|\xi|)\frac{\xi}{|\xi|}\cdot E_\ell \bigg](s,X_\ell(t-s),\Xi_\ell(t-s))ds.
\end{align*}
Remark that system (\ref{inhomogeneous linear characteristics }) is divergence free and the flow is therefore volume preserving for all times.  
Define the horizontal and perpendicular momentum variables as
\begin{align*} 
\bar \xi := {}^t (\xi_1,\xi_2,0) , \quad \xi^\perp := {}^t (\xi_2,- \xi_1,0) , 
\end{align*}
as well as the following phase $ \Phi $ and remainder functions $ R_\epsilon $ as follows.
  \begin{align} \label{ phi formula}
     \Phi(t,x,\xi) 
     &:=b_e(x)t  -\epsilon \nabla_xb_e(x)\cdot \bigg( t \frac{1}{b_e(x)}\xi^\perp - t^2 \bigg(\frac{\nabla_x b_e(x)\cdot \xi^\perp}{4 \vl{\xi}b_e(x)^2}\bigg)\bar{\xi}  + t^2\bigg(\frac{\nabla_x b_e(x)\cdot \bar{\xi}}{4\vl{\xi}b_e(x)^2}\bigg) \xi^\perp\bigg),
     \\
     R_\epsilon (t,x,\xi) 
     &:=  \frac{1}{b_e(x)} \bigg(\sin\big(\frac{\Phi(t,x,\xi)}{\epsilon\vl{\xi}}\big)\bar{\xi} + \cos\big(\frac{\Phi(t,x,\xi)}{\epsilon\vl{\xi}}\big)\xi^\perp - \xi^\perp \bigg)
     \nn
     &\quad + t \bigg(\frac{\nabla_x b_e(x)\cdot \xi^\perp}{2 \vl{\xi}b_e(x)^2}\bigg)\bar{\xi}  -  t\bigg(\frac{\nabla_x b_e(x)\cdot \bar{\xi}}{2\vl{\xi}b_e(x)^2}\bigg) \xi^\perp .
 \end{align}
Then we have the following approximation for the linear, inhomogeneous characteristics. 

 \begin{lemma} [Approximation for the linear flow] \label{X lemma}
 Consider (\ref{inhomogeneous linear characteristics }). For any $T>0$ and $R_\xi^0 >0$, there exists $C := C(T,||b_e||_{W^{2,\infty}}, R_\xi^0) \geq 0$ 
 such that for all $t\in[0,T]$ and $|\xi|\leq R_\xi^0$, we have
 \begin{align} \label{inhomog x char approx}
     \bigg|X_\ell(t,x,\xi) - x - \frac{t  \xi_3}{\vl{\xi}}e_3 - \epsilon R_\epsilon(t,x,\xi) \bigg| \leq \epsilon^2 C ,
 \end{align}
 \begin{align} \label{inhomog xi char approx}
     \bigg|\Xi_\ell(t,x,\xi)  - \cos(\frac{\Phi(t,x,\xi)}{\epsilon\vl{\xi}})\bar{\xi} + \sin(\frac{\Phi(t,x,\xi)}{\epsilon\vl{\xi}})\xi^\perp  - \xi_3e_3\bigg| \leq \epsilon C.
 \end{align}
 \end{lemma}
 \begin{remark}
 The general case, when the direction of $\mathbf{B}_e$ is not fixed, is handled in the appendix and the necessary results are reported in Lemma \ref{general field lemma}. Furthermore, in the general case $\sqrt{\Xi_1(t) + \Xi_2(t)}$ and $|\Xi_3(t)|$ also vary with time (although by remark \ref{ Q remark} we still have $|\Xi(t)| = |\xi|$). So this must be considered in Lemma \ref{general field lemma} as well.
 \end{remark}
 
 \begin{proof}
For neatness, we omit the subscript $\ell$, but note that $(X,\Xi)(t)$ should not be confused with (\ref{Xode})-(\ref{endXode}). Remark that $\nabla_x\cdot \mathbf{B}_e = 0$ and $\mathbf{B}_e(x) = b_e(x)e_3$ imply that $b_e$ depends only on the horizontal spatial components, $b_e(x) = b_e(x^1,x^2)$. Furthermore $\dt{t}|\Xi|^2 = 0$ and $|X(t)| \leq x +t$ so the solution $(X,\Xi)(t)$ is globally defined. Moreover, the solution $\Xi(t)$ in (\ref{inhomogeneous linear characteristics }) can be expressed as 
 \begin{align} \label{ Xi formula}
     \Xi(t,x,\xi) = \cos(\frac{\theta_\epsilon(t,x,\xi)}{\epsilon\vl{\xi}})\bar{\xi} - \sin(\frac{\theta_\epsilon(t,x,\xi)}{\epsilon\vl{\xi}})\xi^\perp + e_3 \xi_3,
 \end{align}
 where 
 \begin{align*}
     \theta_\epsilon(t,x,\xi) := \int_0^t b_e(X(s,x,\xi))ds .
  \end{align*}
  Retain that, due to (\ref{lowupboundve}), we have 
 \begin{align}     \label{mindederithbe}
    \dt{t}\theta_\epsilon(t,x,\xi) = b_e(X(t,x,\xi)) \geq c(K) > 0 . 
 \end{align}
This part is similar to the setting of \cite{fastrotations}. Note that since $|\Xi| = |\xi|$ we also have $\vl{\Xi} = \vl{\xi}$. Hence we can integrate to obtain an expression for $X(t)$
 \begin{align*}
     X(t,x,\xi) = x + \frac{t  \xi_3}{\vl{\xi}}e_3 + \frac{1}{\vl{\xi}}\int_0^t \Bigl( \cos(\frac{\theta_\epsilon(s,x,\xi)}{\epsilon\vl{\xi}})\bar{\xi} - 
     \sin(\frac{\theta_\epsilon(s,x,\xi)}{\epsilon\vl{\xi}})\xi^\perp \Bigr) \, ds.
 \end{align*}
 The time integral is rapidly oscillating, so an integration by parts gives
 \begin{align} \label{order eps approx of X}
     X - x - \frac{t  \xi_3}{\vl{\xi}}e_3 &=  \epsilon \int_0^t\frac{1}{b_e(X(s))}\partial_s\big(\sin(\frac{\theta_\epsilon(s,x,\xi)}{\epsilon\vl{\xi}})\bar{\xi} + \cos(\frac{\theta_\epsilon(s,x,\xi)}{\epsilon\vl{\xi}})\xi^\perp \big) ds
     \nn
     &= \epsilon \frac{1}{b_e(X(t))} \big(\sin(\frac{\theta_\epsilon(t,x,\xi)}{\epsilon\vl{\xi}})\bar{\xi} + \cos(\frac{\theta_\epsilon(t,x,\xi)}{\epsilon\vl{\xi}})\xi^\perp \big) - \epsilon \frac{1}{b_e(x)}\xi^\perp
     \nn
     &\qquad + \epsilon \int_0^t\frac{\nabla_x b_e(X)\cdot\dot{X}}{b_e(X(s))^2}\big(\sin(\frac{\theta_\epsilon(s,x,\xi)}{\epsilon\vl{\xi}})\bar{\xi} + \cos(\frac{\theta_\epsilon(s,x,\xi)}{\epsilon\vl{\xi}})\xi^\perp \big) ds.
 \end{align}
 Therefore we have the estimate
 \begin{align} \label{ eps X estimate}
     |X(t) - x - \frac{t \xi_3}{\vl{\xi}}e_3| \leq \epsilon|\xi|\big( \frac{3}{b_-} + 2 t ||\nabla_x (\frac{1}{b_e})||_{L^\infty}\big),
 \end{align}
 where
 \begin{align} \label{ b minus}
   0 < c(K) \leq  b_- = b_-(t,x) := \min_{|x-y|\leq t} b_e(y).
 \end{align}
 We can then Taylor expand $ b_e(X)^{-2} \, \nabla_x b_e(X) $ in the last line of (\ref{order eps approx of X}) with respect to $X$ 
 about the point $x+ t \vl{\xi}^{-1} \xi_3 e_3 $. When doing this, since $b_e(x) = b_e(x^1,x^2)$ does not depend on $ x^3 $, the shift 
 $ t \vl{\xi}^{-1} \xi_3 e_3 $ does not appear, so that:
 \[ b_e(X)^{-2} \, \nabla_x b_e(X) = b_e(x)^{-2} \, \nabla_x b_e(x) + \mathcal{O} (\epsilon) . \]
 and integrate by parts once more after substituting $\dot{X} = v(\Xi)$. The only terms of size $\epsilon$ which remain are the `slow terms' with non-zero mean. For instance, using standard trig identities and substituting (\ref{ Xi formula}) we have
\begin{align*}
    \dot{X}&\sin(\frac{\theta_\epsilon(t,x,\xi)}{\epsilon\vl{\xi}}) = \frac{\Xi}{\vl{\xi}}\sin(\frac{\theta_\epsilon(t,x,\xi)}{\epsilon\vl{\xi}})
    \nn
    &=\frac{1}{\vl{\xi}}\bigg[\cos(\frac{\theta_\epsilon(t,x,\xi)}{\epsilon\vl{\xi}})\bar{\xi} - \sin(\frac{\theta_\epsilon(t,x,\xi)}{\epsilon\vl{\xi}})\xi^\perp + e_3 \xi_3\bigg]\sin(\frac{\theta_\epsilon(t,x,\xi)}{\epsilon\vl{\xi}})
    \nn
    &=
    \frac{1}{2\vl{\xi}}\bigg[\sin(\frac{2\theta_\epsilon(t,x,\xi)}{\epsilon\vl{\xi}}) \bar{\xi} + \cos(\frac{2\theta_\epsilon(t,x,\xi)}{\epsilon\vl{\xi}})\xi^\perp + 2e_3\xi_3\sin(\frac{\theta_\epsilon(t,x,\xi)}{\epsilon\vl{\xi}})\bigg] - \frac{1}{2\vl{\xi}}\xi^\perp.
\end{align*}
Similarly,
\begin{align*}
   \dot{X}&\cos(\frac{\theta_\epsilon(t,x,\xi)}{\epsilon\vl{\xi}}) = \frac{\Xi}{\vl{\xi}}\cos(\frac{\theta_\epsilon(t,x,\xi)}{\epsilon\vl{\xi}})
   \nn
   &=\frac{1}{2\vl{\xi}}\bigg[  \cos(\frac{2\theta_\epsilon(t,x,\xi)}{\epsilon\vl{\xi}})\bar{\xi} - \sin(\frac{2\theta_\epsilon(t,x,\xi)}{\epsilon\vl{\xi}})\xi^\perp + 2e_3\xi_3 \cos(\frac{\theta_\epsilon(t,x,\xi)}{\epsilon\vl{\xi}})\bigg] + \frac{1}{2\vl{\xi}}\bar{\xi}.
\end{align*}
Therefore, Taylor expanding the first term in the integrand of (\ref{order eps approx of X}) gives
\begin{align} \label{X osc 1}
    &\frac{\nabla_xb_e(X)\cdot \dot{X}}{b_e(X)^2}\sin(\frac{\theta_\epsilon(t,x,\xi)}{\epsilon\vl{\xi}})\bar{\xi} = \frac{\nabla_x b_e(x)\cdot \dot{X}}{b_e(x)^2}\sin(\frac{\theta_\epsilon(t,x,\xi)}{\epsilon\vl{\xi}})\bar{\xi} + \mathcal{O}(\epsilon)
    \nn
    &= \bigg[ \frac{\nabla_x b_e(x)}{2 \vl{\xi}b_e(x)^2}\cdot \bigg(\sin(\frac{2\theta_\epsilon(t,x,\xi)}{\epsilon\vl{\xi}}) \bar{\xi} + \cos(\frac{2\theta_\epsilon(t,x,\xi)}{\epsilon\vl{\xi}})\xi^\perp + 2e_3\xi_3\sin(\frac{\theta_\epsilon(t,x,\xi)}{\epsilon\vl{\xi}})\bigg)\bigg] \bar{\xi} 
    \nn
    &\quad - \bigg(\frac{\nabla_x b_e(x)\cdot \xi^\perp}{2 \vl{\xi}b_e(x)^2}\bigg)\bar{\xi} + \mathcal{O}(\epsilon) ,
\end{align}
and the other term in (\ref{order eps approx of X}) becomes
\begin{align} \label{X osc 2}
    &\frac{\nabla_xb_e(X)\cdot \dot{X}}{b_e(X)^2}\cos(\frac{\theta_\epsilon(t,x,\xi)}{\epsilon\vl{\xi}}) \xi^\perp = \frac{\nabla_x b_e(x)\cdot \dot{X}}{b_e(x)^2}\cos(\frac{\theta_\epsilon(t,x,\xi)}{\epsilon\vl{\xi}}) \xi^\perp + \mathcal{O}(\epsilon)
    \nn
    &= \bigg[ \frac{\nabla_x b_e(x)}{2\vl{\xi}b_e(x)^2}\cdot \bigg( \cos(\frac{2\theta_\epsilon(t,x,\xi)}{\epsilon\vl{\xi}})\bar{\xi} - \sin(\frac{2\theta_\epsilon(t,x,\xi)}{\epsilon\vl{\xi}})\xi^\perp + 2e_3\xi_3 \cos(\frac{\theta_\epsilon(t,x,\xi)}{\epsilon\vl{\xi}})\bigg)\bigg] \xi^\perp
    \nn
    &\quad + \bigg(\frac{\nabla_x b_e(x)\cdot \bar{\xi}}{2\vl{\xi}b_e(x)^2}\bigg) \xi^\perp + \mathcal{O}(\epsilon) .
\end{align}
Therefore after substituting (\ref{X osc 1}) and (\ref{X osc 2}) into (\ref{order eps approx of X}) and integrating the oscillating terms by parts, up to order $\epsilon^2$, we have
\begin{align} \label{order eps 2 approx of X}
     X - x - \frac{t  \xi_3}{\vl{\xi}}e_3 
     &= \epsilon \frac{1}{b_e(X(t))} \big(\sin(\frac{\theta_\epsilon(t,x,\xi)}{\epsilon\vl{\xi}})\bar{\xi} + \cos(\frac{\theta_\epsilon(t,x,\xi)}{\epsilon\vl{\xi}})\xi^\perp \big) - \epsilon \frac{1}{b_e(x)}\xi^\perp
     \nn
     &-\epsilon \int_0^t \bigg[-\bigg(\frac{\nabla_x b_e(x)\cdot \xi^\perp}{2 \vl{\xi}b_e(x)^2}\bigg)\bar{\xi}  + \bigg(\frac{\nabla_x b_e(x)\cdot \bar{\xi}}{2\vl{\xi}b_e(x)^2}\bigg) \xi^\perp \bigg]ds + \mathcal{O}(\epsilon^2)
     \nn
     &=\epsilon \frac{1}{b_e(x)} \big(\sin(\frac{\theta_\epsilon(t,x,\xi)}{\epsilon\vl{\xi}})\bar{\xi} + \cos(\frac{\theta_\epsilon(t,x,\xi)}{\epsilon\vl{\xi}})\xi^\perp \big) - \epsilon \frac{1}{b_e(x)}\xi^\perp
     \nn
     &\quad +\epsilon t \bigg(\frac{\nabla_x b_e(x)\cdot \xi^\perp}{2 \vl{\xi}b_e(x)^2}\bigg)\bar{\xi}  - \epsilon t\bigg(\frac{\nabla_x b_e(x)\cdot \bar{\xi}}{2\vl{\xi}b_e(x)^2}\bigg) \xi^\perp  + \mathcal{O}(\epsilon^2) . \end{align}
Similarly we can Taylor expand $\theta_\epsilon$ and integrate the oscillating terms by parts
\begin{align} \label{order eps 2 approx of theta eps}
    &\theta_\epsilon(t,x,\xi) = \int_0^t b_e(x) + \nabla_xb_e(x)\cdot \Bigl \lbrack X - x - \frac{t  \xi_3}{\vl{\xi}}e_3 - \epsilon R_\epsilon \Bigr \rbrack \,  ds + \mathcal{O}(\epsilon^2)
    \nn
    &= b_e(x)t  +\nabla_xb_e(x)\cdot \bigg(-\epsilon t \frac{1}{b_e(x)}\xi^\perp +\epsilon t^2 \bigg(\frac{\nabla_x b_e(x)\cdot \xi^\perp}{4 \vl{\xi}b_e(x)^2}\bigg)\bar{\xi}  - \epsilon t^2\bigg(\frac{\nabla_x b_e(x)\cdot \bar{\xi}}{4\vl{\xi}b_e(x)^2}\bigg) \xi^\perp\bigg) + \mathcal{O}(\epsilon^2)
    \nn
    &=\Phi(t,x,\xi) + \mathcal{O}(\epsilon^2).
\end{align}
After replacing this inside (\ref{ Xi formula}), we get (\ref{inhomog xi char approx}). Finally, we can replace $\theta_\epsilon$ inside (\ref{order eps 2 approx of X}) 
as indicated in (\ref{order eps 2 approx of theta eps}) to recover (\ref{inhomog x char approx}).
 \end{proof}
Lemma \ref{X lemma} gives the immediate corollary which follows.

 \begin{corollary} \label{Dx cor}
 There exists $\epsilon_0$ and $T>0$ independent of $\epsilon_0$, such that for all $\epsilon \in (0,\epsilon_0]$, $t\in[0,T]$ and $|\xi| \leq R_\xi^0$ the solution maps $x \mapsto X(t,x,\xi)$ of (\ref{inhomogeneous linear characteristics }) is a diffeomorpshism.
 \end{corollary}
 \begin{proof}
 The proof easily follows by computing $D_xX$ and taking the operator sup-norm 
 \begin{align*}
     ||(D_xX - Id_{3\times 3})(t,\cdot)||_{L^\infty_x(|x-y|\leq t)} \leq Ct +\mathcal{O}(\epsilon),   
 \end{align*}
 where $C$ depends only on $R_\xi^0$ and $||b_e||_{W^{2,\infty}}$. so that for $t$ and $\epsilon$ small enough one has
 \begin{align*}
     ||(D_xX -Id_{3\times3})(t,\cdot)||_{L^\infty_x(|x-y|\leq t)} <1.
 \end{align*}
 The map $ X(t,\cdot,\xi) $ is therefore a local diffeomorphism. It is injective (uniqueness part of Cauchy-Lipschitz Theorem) and it is surjective (it suffices to integrate 
 the  flow in the opposite direction, from $ t$ to $ 0 $). It is bijective, and thereby it is a global diffeomorphism.
 \end{proof}
 
 \begin{remark} \label{ Q remark}
 A similar approximation to (\ref{inhomog x char approx}) and (\ref{inhomog xi char approx}) holds when we include the quadratic term $Q$ coming from the linearized version of the characteristics (\ref{Xode})-(\ref{endXode}). That is to say, the normal form procedure of Lemma \ref{X lemma} holds when we allow the direction of the applied field $\mathbf{B}_e$ to vary. However, the procedure and approximation is much less explicit as it depends on the matrix $O(x)$. Furthermore, we no longer have $b_e(x)$ independent of $x_3$, so the 3rd component of the approximation (\ref{inhomog x char approx}) is less trivial. But since $\xi\cdot Q\equiv 0$, we still have $|\Xi(t)| = |\xi|$. This result is given in the appendix
 \end{remark}
 
 %%%%%%%%%%%%%%%%%%%%

 \subsection{Uniform Bounds of Dilute Linear Model in Inhomogeneous Magnetic Field} \label{linear inhomog section}
 In this section we derive uniform estimates for a linear model of a dilute plasma in an inhomogeneous magnetic field given by the following Cauchy problem: 
 \begin{equation}  \label{inhomog lin dilute model}
\left \{ \begin{array}{ll}
\partial_t f_\ell + v(O(x)\xi)\cdot\nabla_x f_\ell - \epsilon^{-1}\vl{\xi}^{-1}b_e(x)\partial_\theta f_\ell  \! \! \! & \\
\qquad + \, \vl{\xi}^{-1}Q(x,\xi)\cdot\nabla_\xi f_\ell = \epsilon M'(|\xi|)|\xi|^{-1}O(x)\xi\cdot E_\ell , & \\
\displaystyle \partial_t E_\ell -\nabla_x \times B_\ell = \int v(O(x)\xi) \, f_\ell d\xi, \quad & \nabla_x\cdot E = -\rho(f_\ell) , \\
\partial_t B_\ell + \nabla_x\times  E_\ell = 0, \quad  & \nabla_x \cdot B_\ell = 0 .
\end{array} \right.
\end{equation}
Again we assume the compatibility conditions on the initial data $(f^{in}, E^{in}, B^{in})$:
 \begin{align} \label{linear compatibilty}
     \nabla_x \cdot E^{in} = - \rho(f^{in}), \qquad \nabla_x \cdot B^{in} = 0.
 \end{align}
Note that (\ref{inhomog lin dilute model}) is the linearized version of the straightened system (\ref{fullsystem}).  In this section we prove the following proposition when 
$\mathbf{B}_e || e_3$. The general procedure when $\mathbf{B}_e$ has variable direction is handled in the Appendix. 

\begin{proposition} \label{linear inhomog theorem}(Uniform estimates for the solutions of the inhomogeneous linear problem) 
Let $(f^{in}, E^{in}, B^{in})\in C^2_c(\mathbb{R}^6; \mathbb{R})\times[C^2_c(\mathbb{R}^3;\mathbb{R}^3)]^2$ satisfying the compatibility conditions 
(\ref{linear compatibilty}). Suppose there is $R_\xi^0 >0$ such that $\text{supp}(f^{in}(x,\cdot))\subset \{|\xi|\leq R_\xi^0\}$ and denote by 
$(f_\epsilon, E_\epsilon, B_\epsilon) $ the solution in $ C^1(\mathbb R_+;L_{x,\xi}^\infty) $ of the linear Cauchy problem 
(\ref{inhomog lin dilute model})-(\ref{linear compatibilty}).
Then, there exists $T>0$ and $\epsilon_0(T)\in (0,1]$ such that for all $\epsilon \in (0,\epsilon_0]$ we can find a constant $C_T$ depending on $T, R_\xi^0$ 
and $||(f^{in}, E^{in}, B^{in})||_{W^{1,\infty}_{x,\xi}}$ such that for all $ t\in [0,T]$, we have
\begin{align}
||(f_\epsilon,E_\epsilon,B_\epsilon)(t)||_{L^\infty_{x,\xi}} &+||\epsilon \partial_t (f_\epsilon,E_\epsilon,B_\epsilon)(t)||_{L^\infty_{x,\xi}} + 
||\partial_{x_3}(f_\epsilon,E_\epsilon,B_\epsilon)(t)||_{L^\infty_{x,\xi}} 
\nn
&+|| \epsilon\bar{\nabla}_x(f_\epsilon,E_\epsilon,B_\epsilon)(t)||_{L^\infty_{x,\xi}}  + ||\epsilon \nabla_\xi f_\epsilon(t)||_{L^\infty_{x,\xi}}  \leq C_T,
 \end{align}
 where $\bar{\nabla}_x := (\partial_{x_1},\partial_{x_2},0)$. Furthermore, in the case of prepared data, that is when
 \begin{align} \label{linear prepared homog}
     ||\partial_\theta f^{in}||_{L^\infty_{x,\xi}} \leq \epsilon C,
 \end{align}
 the preceding Lipschitz norm control becomes
\begin{align}
||(f_\epsilon,E_\epsilon,B_\epsilon)(t)||_{L^\infty_{x,\xi}} & +|| \partial_t(f_\epsilon,E_\epsilon,B_\epsilon)(t)
||_{L^\infty_{x,\xi}} + ||\partial_{x_3}(f_\epsilon,E_\epsilon,B_\epsilon)(t)||_{L^\infty_{x,\xi}} 
\nn
&+|| \bar{\nabla}_x(f_\epsilon,\epsilon E_\epsilon,\epsilon B_\epsilon)(t)||_{L^\infty_{x,\xi}}  + ||\nabla_\xi f_\epsilon(t)
||_{L^\infty_{x,\xi}}  \leq C_T.  \label{oifcvjb}
 \end{align}
\end{proposition}

 Moreover, when $b_e$ is constant, we can also achieve uniform estimates of $\nabla_x (f_\epsilon,E_\epsilon,B_\epsilon)$ for ill-prepared data since the system 
 (\ref{inhomog lin dilute model})-(\ref{linear compatibilty}) becomes homogeneous in $x$. Before proving Proposition 
 \ref{linear inhomog theorem}, we would like to illustrate the optimality of its estimates. Indeed, the prepared data 
 assumption is necessary for uniform Lipschitz estimates in both the linear and non-linear system. The underlying 
 mechanism is local (we can forget the condition on the support), and it does not involve the spatial variable $ x $. 
 Thus, we can explain it below by looking at functions $f $ depending only on $ t $ and $ \xi $.

\begin{example} \label{ example 1}
Consider the initial data given by
\begin{align*}
    f^{in}(\xi(r,\theta,z)) = \chi(r,z)\cos(2\theta), \quad  (E^{in}, B^{in}) = (0,0),
\end{align*}
where $(r,\theta,z)$ are the cylindrical coordinates for $\xi$, and take $\chi \in C^1_c(\mathbb{R}_+\times\mathbb{R})$. Then 
\begin{align*}
    & f(t,x,\xi) = \chi(r,z) \, \cos \bigl( 2(\theta + \frac{t}{\epsilon\vl{\xi}}) \bigr),
    \nn
   & (E,B)(t,x)\equiv (0,0),
\end{align*}
solves the nonlinear problem with $M\equiv 0$:
\begin{align*}
\begin{cases}
&\partial_t f + v(\xi)\cdot\nabla_x f - \frac{1}{\epsilon\vl{\xi}}\partial_\theta f -\epsilon[E+v(\xi)\times B]\cdot \nabla_\xi f  =  0
 \\
&\partial_t E -\nabla_x \times B = J(f), \quad \nabla_x\cdot E = -\rho(f)
\\
&\partial_t B + \nabla_x\times  E = 0, \ \  \qquad   \nabla_x \cdot B = 0
\\
&(f,E,B)|_{t=0} = (f^{in}, 0, 0)
\end{cases}
\end{align*}
This follows by construction, as we have $ \rho(f) = 0 $ and $ J(f) = 0 $ (due to the factor $ 2 $ in front of $ \theta $ in the definition of $ f $). 
When $ \chi \not \equiv 0 $, we do not have 
(\ref{linear prepared homog}), and we see that the control (\ref{oifcvjb}) is not satisfied since both $|\partial_t f|$ and $|\nabla_\xi f|$ 
are of order $\epsilon^{-1}$. These estimates become uniform if instead the initial data was prepared.
\end{example}

\begin{proof} [Proof of Proposition \ref{linear inhomog theorem}.] The non singular terms inside (\ref{good equation for E}) can be 
handled as in \cite{cold} or as briefly explained in Subsection \ref{obstruction section}. Thus, we can focus on the more problematic term 
implied by (\ref{good equation for E}), the one which is of order $\epsilon^{-1}$. This involves $ {\bar f} $ and not the expression $ f $ 
obtained through the change of variables (\ref{OOOOcha}). The change of variable $ \xi = O(x) \eta $ allows to remedy this, yielding
\begin{align}\label{firstaftersubaj}
-\epsilon^{-1}\int \nabla_\xi p \bigl(t,x, O(x) \eta \bigr) Y(t,x) *_{t,x} \bigl(\mathbbm{1}_{t>0} [v\bigl(O(x) \eta \bigr)\times \mathbf{B}_e] 
f_\ell (t,x,\eta) \bigr) \, d\eta .
\end{align}
The aim of Lemma \ref{bad term change of variables} is to reformulate (\ref{firstaftersubaj}). Thanks to (\ref{firstaftersubaj}), we can 
handle the solution $f_\ell$ of (\ref{inhomog lin dilute model}). The interest is that we can solve $f_\ell$ in (\ref{inhomog lin dilute model}) 
using Duhamel's Principle
\begin{align} \label{ f ell equaiton}
     f_\ell(t,x,\xi) = f^{in}(X(-t),\Xi(-t)) + \epsilon \int_0^t\bigl(M'(|\xi|)\frac{O(X)\Xi}{|\xi|}\cdot E_\ell\bigr)(s,X(t-s),\Xi(t-s))ds,
\end{align}
where $(X,\Xi)(t)$ solves the characteristics of the linear Vlasov equation in the system (\ref{inhomog lin dilute model}). Our goal is to absorb 
the singular factor $ \epsilon^{-1} $ inside (\ref{firstaftersubaj}). To accomplish this, the strategy is to substitute (\ref{ f ell equaiton}) into 
(\ref{firstaftersubaj}). This yields a sum of two terms. The first is
\begin{align}\label{firstaftersub}
-\epsilon^{-1} \int \nabla_\xi p \bigl(t,x, O(x) \eta \bigr) Y(t,x) *_{t,x} \bigl(\mathbbm{1}_{t>0} [v\bigl(O(x) \eta \bigr) \times \mathbf{B}_e] 
f^{in}(X(-t),\Xi(-t)) \bigr) \, d \eta .
\end{align}
The second is
\begin{align}\label{secondaftersub}
- \int \nabla_\xi p \bigl(t,x, O(x) \eta \bigr) & *_{t,x} \Bigl( \mathbbm{1}_{t>0} [v \bigl(O(x) \eta \bigr) \times \mathbf{B}_e] \nonumber \\
\ & \times \int_0^t\bigl(M'(|\eta|)\frac{O(X)\Xi}{|\eta|}\cdot E_\ell\bigr)(s,X(t-s),\Xi(t-s))ds \Bigr) \, d\eta .
\end{align}

\noindent Lemma \ref{f initial non stationary} is devoted to estimate (\ref{firstaftersub}). We accept a loss of derivatives of the data $ f^{in} $,
and therefore we can apply non-stationary phase arguments. The idea is to take advantage of the rapid oscillations implied by the 
characteristics. 

\smallskip

\noindent Lemma \ref{ dilute E source lemma} deals with (\ref{secondaftersub}). We are saved by the dilute equilibrium condition 
$M'_\epsilon(|\xi|) =  {\mathcal O} (\epsilon)$ which explains why the singular factor $ \epsilon^{-1} $ has disappeared from 
(\ref{secondaftersub}).

\smallskip

\noindent Briefly, Proposition \ref{linear inhomog theorem} is proved in four stages made of three lemmas \ref{bad term change of variables},
\ref{f initial non stationary}  and \ref{ dilute E source lemma}, followed by a closing paragraph ``{\it End of proof or Proposition \ref{linear inhomog theorem}}".
\end{proof}

\smallskip

\noindent A preliminary step is to reformulate (\ref{firstaftersub}) using $ \bar{f} $ instead of $ f $, with $ \bar{f} $ as in (\ref{OOOOcha}).

\begin{lemma} \label{bad term change of variables}[Transfer of derivatives after straightening of the field lines]
Let $\bar{g}(t,\cdot)\in W^{1,\infty}_{x,\xi}$. Then, under the change of variables
\begin{align*}
    g(t,x,\xi) := \bar{g}(t,x,O(x)\xi), \quad \bar{g}(t,x,\xi) = g(t,x,O^t(x)\xi),
\end{align*}
we have the following identity
\begin{align}
    &\int p(t,x,\xi) Y(t,x) *_{t,x}\mathbbm{1}_{t>0}\nabla_\xi\cdot [(v(\xi)\times \frac{1}{\epsilon}\mathbf{B}_e(x) )\bar{g}(t,x,\xi)]d\xi 
    \nn
    &= -\frac{1}{4\pi \epsilon}\int_0^t \int_{\mathbb{S}^2}\int   s\big(\dt{\theta}[p(1,\omega,O(s\omega)\eta)]\big)\frac{b_e(x-s\omega)}{\vl{\eta}} g(t-s,x-s\omega,\eta)d\eta dS(\omega)ds,
\end{align}
where $ d / d \theta $ is the total derivative with respect to the new variable $\eta := O^t(x)\xi$,
\begin{align*}
    \dt{\theta} := \eta_2 \partial_{\eta_1} - \eta_1 \partial_{\eta_2} = \eta^\perp\cdot\nabla_\eta.
\end{align*}
\end{lemma}
\begin{proof}
First consider the variable change, with $x$ as a parameter
\begin{align}
    \xi := O(x) \eta, \quad \eta := O^t(x)\xi, \quad d\xi = |\det(O(x))|d\eta = d\eta.
\end{align}
Then using the formula (\ref{ divergence formula change}), remarking that $|\det(O(x))| =1$, and again that $O^t\mathbf{B}_e(x) = b_e(x) e_3$ we have the divergence in the new variable becomes 
\begin{align*}
    \forall x\in \mathbb{R}^3, \ \nabla_\xi \cdot \big[ \big(\frac{\xi}{\vl{\xi}}\times \mathbf{B}_e(x) \big)\bar{g}(t,x,\xi) \big] &
    = \nabla_\eta \cdot  \big[ O^t (x) \big( \frac{O(x)\eta}{\vl{\eta}}\times \mathbf{B}_e(x) \big) \bar{g}(t,x,O(x)\eta) \big]
    \nn
    &= \nabla_\eta \cdot  \big[ O^t(x) O(x)\big( \frac{\eta}{\vl{\eta}}\times O^t\mathbf{B}_e(x) \big) g(t,x,\eta) \big]
    \nn
    &=\nabla_\eta \cdot\big[ b_e(x) \frac{\eta^\perp}{\vl{\eta}} g(t,x,\eta)\big]
\end{align*}
Therefore we can change variables, apply Lemma \ref{wave integral lemma}, then integrate by parts in the new momentum variable $\eta$ to arrive at the conclusion
\begin{align*}
&\int p(t,x,\xi) Y(t,x) *_{t,x}\mathbbm{1}_{t>0}\nabla_\xi\cdot [(v(\xi)\times \frac{1}{\epsilon}\mathbf{B}_e(x) )\bar{g}(t,x,\xi)]d\xi
\nn
&= \int_0^t \int \int p(t-s,x-y,\xi)Y(t-s,x-y) \nabla_\xi \cdot [(v(\xi)\times \frac{1}{\epsilon}\mathbf{B}_e(y) )\bar{g}(s,y,\xi)] d\xi dy ds
\nn
&=\epsilon^{-1}\int_0^t \int \int p(t-s,x-y,O(y)\eta)Y(t-s,x-y) \nabla_\eta \cdot\big( b_e(y) \frac{\eta^\perp}{\vl{\eta}} g(s,y,\eta)\big) d\eta dy ds
\nn
&=\frac {1}{4\pi \epsilon}\int_0^t \int_{\mathbb{S}^2}\int sp(1,\omega, O(s\omega)\eta) \nabla_\eta \cdot\big( b_e(x-s\omega) \frac{\eta^\perp}{\vl{\eta}} g(t-s,x-s\omega,\eta)\big) d\eta dS(\omega) ds
\nn
&=-\frac {1}{4\pi \epsilon}\int_0^t \int_{\mathbb{S}^2}\int s\eta^\perp \cdot \nabla_\eta \bigg[p(1,\omega, O(s\omega)\eta)\bigg] \frac{b_e(x-s\omega)}{\vl{\eta}} g(t-s,x-s\omega,\eta) d\eta dS(\omega) ds
\nn 
&=-\frac {1}{4\pi \epsilon}\int_0^t \int_{\mathbb{S}^2}\int s\dt{\theta} \bigg[p(1,\omega, O(s\omega)\eta)\bigg] \frac{b_e(x-s\omega)}{\vl{\eta}}  g(t-s,x-s\omega,\eta) d\eta dS(\omega) ds.
\end{align*}
Observe that there are no more derivatives on $ g $, but instead derivatives on the symbol $ p (\cdot) $ given by (\ref{p formula}). We can check that 
this derivative on $ p (\cdot) $ is non zero.
\end{proof}

\smallskip

\noindent Next we will state lemma \ref{f initial non stationary}, which allows us to regain a factor of $\epsilon$, by taking advantage of the time averaged rapid oscillations coming from the characteristics. Given Lemma \ref{bad term change of variables}, the method will hold when $O(x) \not\equiv Id_{3\times 3}$. However, as stated in Remark \ref{ Q remark}, the approximations of Lemma \ref{X lemma} are less explicit.
Therefore from now on we assume
\begin{align*}
    \mathbf{B}_e(x) = b_e(x) ^t(0,0,1), \ O(x) \equiv Id_{3\times 3}, \ Q(x,\xi)\equiv 0.
\end{align*}
Then we have the following lemma, where $ \xi $ comes to replace $ \eta $ to fit with the presentation of (\ref{linear dilute approx}).

\begin{lemma} \label{f initial non stationary}[Impact of the oscillating flow]
For any $T>0$, there exists $\epsilon_0(T) \in (0,1]$, and constant $C_T := C_T(R_\xi^0, ||f^{in}||_{W^{1,\infty}_{x,\xi}}, ||b_e||_{W^{2,\infty}})$ such that for all $\epsilon\in (0,\epsilon_0]$, and all $t\in [0,T]$, the following estimate holds
\begin{align} \label{ f in flow estimate}
     \bigg|\int \int_{\mathbb{S}^2}\int_0^t & \frac{ s \partial_{\theta} p(1,\omega,\xi)}{4\pi}\frac{b_e(x-s\omega)}{\vl{\eta}} \nonumber \\
     & \times f^{in}\circ\big[(X_\ell,\Xi_\ell)(-(t-s),x-s\omega,\xi)\big] ds dS(\omega) d\xi\bigg| \leq \epsilon C_T . 
\end{align}
\end{lemma}
\begin{proof}
We can first simplify our analysis by Taylor expanding $f^{in}$ composed with the flow with respect to $X$ using Lemma \ref{X lemma}
\begin{align*}
    f^{in}(X(t),\Xi(t)) = f^{in}(x+\frac{t\xi_3}{\vl{\xi}}e_3,\Xi(t)) + \mathcal{O}(\epsilon ||\nabla_x f^{in}||_{L^\infty}).
\end{align*}
This remaining term of order $\mathcal{O}(\epsilon ||\nabla_x f^{in}||_{L^\infty})$, when substituted into (\ref{ f in flow estimate}), is controlled using Lemma \ref{wave integral lemma} by  the constant
\begin{align} \label{ nabla f in est}
C_1(t) := \frac{t^2}{3}||p(1,\cdot,\cdot)||_{L^\infty(\mathbb{S}^2\times \{|\xi|\leq R_\xi^0\})} ||b_e||_{L^\infty(|x|\leq R_x^0+t)}||\nabla_xf^{in}||_{L^\infty}||R_\epsilon||_{L^\infty(|x|\leq R_x^0+t, \ |\xi|\leq R_\xi^0)}.
\end{align}
Next remark that the momentum component $\Xi(t)$ given by (\ref{ Xi formula}) can be viewed as a rotation as follows
\begin{align*}
    \Xi(t,x,\xi) = \mathcal{R}(\frac{\Phi(t,x,\xi)}{\epsilon\vl{\xi}})\xi + \mathcal{O}(\epsilon), 
\end{align*}
where $\mathcal{R}$ is the rotation matrix about the $\xi_3$-axis
\begin{align*}
   \mathcal{R}(\theta) = \BM{\cos(\theta) & -\sin(\theta) & 0 \\ \sin(\theta) & \cos(\theta) & 0 \\ 0 & 0 & 1}.
\end{align*}
Therefore we may convert $\Xi(t)$ to cylindrical coordinates
\begin{align*}
    \Xi(t,x,\xi(r,\theta,z)) = \BM{r \cos(\theta +\frac{\Phi(t,x,\xi)}{\epsilon\vl{\xi}}) \\  r\sin( \theta + \frac{\Phi(t,x,\xi)}{\epsilon\vl{\xi}}) \\  z} +\mathcal{O}(\epsilon).
\end{align*}
where $z = \Xi_3 = \xi_3$ and $r = \sqrt{\Xi_1^2 + \Xi_2^2} = \sqrt{\xi_1^2+\xi_2^2}$ are independent of time. Then with a slight abuse of notation on the 
dependence of $r$ and $z$ we consider the Fourier series
\begin{align} \label{ f in approx 1}
    f^{in}(X(t),\Xi(t)) = \sum_{n\in\mathbb{Z}}f^{in}_n(x+\frac{t \xi_3}{\vl{\xi}}e_3, r,z)e^{in(\theta + \frac{\Phi(t,x,\xi)}{\epsilon\vl{\xi}})} + \mathcal{O} \big(\epsilon||\nabla_{\xi,x} f^{in}|| \big).
\end{align}
Remark a similar estimate to (\ref{ nabla f in est}) holds for the order $\epsilon||\nabla_\xi f^{in}||$ term, now including a momentum derivative. Then 
substituting the order 1 term of (\ref{ f in approx 1}), which must be evaluated at the position $(-(t-s),x-s\omega,\xi)$, into (\ref{ f in flow estimate}), it 
remains to consider
\begin{align}
\label{fourier series bad term 1}
    \int \int_{\mathbb{S}^2}\int_0^t \frac{s\partial_{\theta}p(1,\omega,\xi)}{4\pi}&\frac{b_e(x-s\omega)}{\epsilon\vl{\xi}}\times 
    \nn
    &\sum_{n\in \mathbb{Z}^*} f^{in}_n(x+\frac{(s-t)z}{\vl{\xi}}e_3 - s\omega,r,z)e^{in(\theta +  \frac{\Phi(s-t,x-s\omega,\xi)}{\epsilon\vl{\xi}})}dsdS(\omega)d\xi.
\end{align}
In the above sum, the integer $ n = 0 $ does not appear because the integration with respect to the variable $ \theta $ of the derivative $ \partial_{\theta}p $ 
is simply zero. The next step is to gain back the factor of $\epsilon$ by a time integration along the lines of the proof of Lemma \ref{X lemma}. First note
\begin{align} \label{ 1 over dt phi}
    e^{in \frac{\Phi(s-t,x-s\omega,\xi)}{\epsilon\vl{\xi}}} = \frac{\epsilon \vl{\xi}}{in[\partial_t\Phi - \omega\cdot\nabla_x\Phi](s-t,x-s\omega,\xi)}\partial_s(e^{in \frac{\Phi(s-t,x-s\omega,\xi)}{\epsilon\vl{\xi}}}),  \ n\neq 0.
\end{align}
Therefore we can integrate by parts in time $s$, as long as the denominator of (\ref{ 1 over dt phi}) does not vanish. This is the notion of the \textit{non-stationary phase}. Recalling (\ref{ phi formula}) we have
\begin{align*}
    &(\partial_t\Phi - \omega \cdot \nabla_x\Phi)(t,x,\xi) 
    = b_e(x) - \omega \cdot \nabla_x b_e(x) t 
    \nn
    &-\epsilon\big[\partial_t - \omega\cdot \nabla_x \big]\bigg( \nabla_xb_e(x)\cdot \big(  t\frac{1}{b_e(x)}\xi^\perp  - t^2 \big(\frac{\nabla_x b_e(x)\cdot \xi^\perp}{4 \vl{\xi}b_e(x)^2}\big)\bar{\xi}  + t^2\big(\frac{\nabla_x b_e(x)\cdot \bar{\xi}}{4\vl{\xi}b_e(x)^2}\big) \xi^\perp\big)\bigg) . 
\end{align*}
Therefore choose $T>0$ small enough such that
\begin{align} \label{ T inequality}
    b_- - T||\nabla_x b_e||_{L^\infty} > \frac{3}{4}b_- > 0,
\end{align}
where $b_-$ is defined by (\ref{ b minus}). Note this choice remains independent of $\epsilon$. Furthermore, choose $\epsilon_0(T)$ small enough, such that for $|\xi|\leq R_\xi^0$ fixed, we can bound  the reciprocal according to
\begin{align*}
     \forall t\in [0,T], \  \big|\frac{1}{\partial_t\Phi - \omega \cdot\nabla_x \Phi}\big | \leq \frac{2}{b_-}.
\end{align*}
Thus integrating (\ref{fourier series bad term 1}) by parts in time using (\ref{ 1 over dt phi}), we find the term (\ref{fourier series bad term 1}) remains uniformly bounded, depending only on initial data as long as the Fourier series is absolutely convergent in the sense that 
\begin{align*}
    \sum_{n\in \mathbb{Z}-\{0\}} \frac{||f^{in}_n||_{W^{1,\infty}_x, L^\infty_{r,z}}}{|n|}   \leq C , 
\end{align*}
which is guaranteed by the $ C^2 $-smoothness and compact support of $f^{in}$.
\end{proof}

\noindent Remark this procedure requires 2 derivatives of the phase $\Phi$. For the non-linear characteristics, this would imply 2 derivatives on $\epsilon E$. Next, we consider (\ref{secondaftersub}). 

\begin{lemma} \label{ dilute E source lemma}
There exists $C_T = C_T(||b_e||_{L^\infty},||M'||_{L^\infty}, R_\xi^0)$ such that the following estimate holds 
\begin{align}
    & \int \partial_{\theta}pY*_{t,x}(H\mathbbm{1}_{t\geq 0})d\xi \leq C_T \int_0^t \sup_{s'\in[0,s]}||E(s',\cdot)||_{L^\infty_x}ds ,
\end{align}
where
\begin{align}
    H(t,x,\xi) := \frac{b_e(x)}{\vl{\xi}} \int_0^t\bigg(M'(|\xi|)\frac{\Xi}{|\xi|}\cdot E\bigg)(s,X(t-s),\Xi(t-s))ds .
\end{align}
\end{lemma}
\begin{proof}
Once again Lemma \ref{wave integral lemma} implies 
\begin{align}
   & \int \partial_{\theta}pY*_{t,x}(H\mathbbm{1}_{t\geq 0})d\xi = 
    \nn
    &\int \int_{\mathbb{S}^2}\int_0^t  
    \frac{s \partial_\theta p(1,\omega,\xi)}{4\pi}\frac{b_e(x-s\omega)}{\vl{\xi}} \times
    \nn
    &\qquad \qquad \qquad \bigg[\int_0^{t-s}\bigg(M'(|\xi|)\frac{\Xi}{|\xi|}\cdot E\bigg)(s',X(t-s-s'),\Xi(t-s-s'))ds' \bigg]dsdS(\omega)d\xi
    \nn
    &\qquad \lesssim \frac{||b_e||_{L^\infty(|x-y|\leq t)}t^2}{3}\int\frac{1}{\vl{\xi}}|M'(|\xi|)|||\partial_{\theta}p(1,\cdot,\xi)||_{L^\infty(\mathbb{S}^2)}d\xi \int_0^t \sup_{s'\in[0,s]}||E(s',\cdot)||_{L^\infty_{x}}ds.
\end{align}
\end{proof}

\noindent As an immediate corollary to lemma \ref{f initial non stationary} and \ref{ dilute E source lemma} we obtain the uniform Sup norm estimate in Proposition \ref{linear inhomog theorem}.  We now prove the remainder of Proposition \ref{linear inhomog theorem}, which is related to the information involving derivatives .

\begin{proof}[End of proof of Proposition \ref{linear inhomog theorem}]
There is nothing left to be done to estimate the Sup-norm. To estimate the Lipschitz norm we use a trick from geometric optics. We first compute $\partial_i$ with $i = t,x_1,x_2$ or $x_3$, by first dividing the Vlasov equation by $b_e(x)$, apply $\partial_i$, then re-multiply by $b_e(x)$. By this way there is no term of size $\epsilon^{-1}$ acting as a source:
\begin{align} \label{inhomog lin partial deriv eqn}
    \begin{cases}
    &\partial_t (\partial_if) + v(\xi)\cdot\nabla_x (\partial_if) - \frac{b_e(x)}{\epsilon\vl{\xi}}\partial_\theta (\partial_i f) = \epsilon M'(|\xi|)\frac{\xi}{|\xi|}\cdot (\partial_iE) 
    \\
    &\qquad \qquad \qquad  \qquad \qquad  + \partial_i\ln(b_e)[\partial_t f + v(\xi)\cdot\nabla_x f - \epsilon M'(|\xi|)\frac{\xi}{|\xi|}\cdot E]
 \\
&\partial_t (\partial_iE) -\nabla_x \times (\partial_iB) = J(\partial_if), \quad \nabla_x\cdot (\partial_iE) = -\rho(\partial_if)
\\
&\partial_t (\partial_iB) + \nabla_x\times  (\partial_iE) = 0, \qquad  \quad \  \nabla_x \cdot (\partial_iB) = 0
    \end{cases}
\end{align}
Remark that if $\partial_{x_i}\ln(b_e) = \mathcal{O}(\epsilon)$ (such as the constant case when $\partial_{x_i} b_e\equiv 0$), the sup-norm proof would follow once more since the new source term introduced into the first equation in (\ref{inhomog lin partial deriv eqn}) would be of size $\epsilon$. However, one still has $\partial_t f|_{t=0}$ and $\partial_t^2(E,B)|_{t=0}$ are both of order $\epsilon^{-1}$ for ill-prepared data. So prepared data is essential for uniform estimates. Now, in the case when $i = t$ or  $x_3$ we have $\partial_i\ln(b_e)\equiv 0$  and therefore, we can once again integrate along the flow
\begin{align}
    \partial_{x_3} f(t,x,\xi) = \partial_{x_3}f^{in}(X(-t),\Xi(-t)) + \epsilon \int_0^t\bigg(M'(|\xi|)\frac{\Xi}{|\xi|}\cdot \partial_{x_3}E\bigg)(s,X(t-s),\Xi(t-s))ds \nonumber
\end{align}
and
\begin{align}\label{claqdt}
    \partial_{t} f(t,x,\xi) = \partial_{t}f|_{t=0}(X(-t),\Xi(-t)) + \epsilon \int_0^t\bigg(M'(|\xi|)\frac{\Xi}{|\xi|}\cdot \partial_{t}E\bigg)(s,X(t-s),\Xi(t-s))ds.
\end{align}
As noted, for general (possibly ill-prepared) data, the time derivative $ \partial_{t} f $ in the right hand side of (\ref{inhomog lin partial deriv eqn}) and in (\ref{claqdt}) are of size $ \epsilon^{-1} $.
The weight $ \epsilon $ must be put in factor of $ \partial_{t} f $, $ \partial_1 f $ and $ \partial_2 f $ ($\partial_1$ is interpreted as $\partial_{x_1}$ and so forth) to compensate this. Then, by the previous argument, we have
\begin{align}
    ||\epsilon \partial_t (f,E,B)(t)||_{L^\infty_{x,\xi}} + ||\partial_{x_3}(f,E,B)(t)||_{L^\infty_{x,\xi}} \leq C_T,
\end{align}
and if the data is prepared to ensure $||\partial_tf|_{t=0}||_{L^\infty_{x,\xi}} \leq \epsilon C$, then we have uniform estimates
\begin{align} \label{ f E B t x3}
    ||\partial_t (f,E,B)(t)||_{L^\infty_{x,\xi}} + ||\partial_{x_3}(f,E,B)(t)||_{L^\infty_{x,\xi}} \leq C_T.
\end{align}
When we consider $i = x_1$ and $x_2$ we can integrate along the flow to estimate $\partial_i f$ as follows 
\begin{align} \label{lin inhomog di f estimate}
    |\partial_if(t,x,\xi)| &\leq ||\partial_{i}f^{in}(\cdot)||_{L^\infty_{x,\xi}} + C\int_0^t ||\epsilon E(s,\cdot)||_{L^\infty_x}  + ||\epsilon \partial_i E(s,\cdot)||_{L^\infty_x}ds
    \nn
    &+C\int_0^t ||\partial_t f(s,\cdot)||_{L^\infty_{x,\xi}} + ||\nabla_x f(s,\cdot)||_{L^\infty_{x,\xi}}ds.
\end{align}
For the fields, similar to (\ref{Full E B solution}), we again have
\begin{align*}
    \square \partial_i E = \int v(\xi)\partial_t(\partial_if) + \nabla_x(\partial_i f)d\xi.
\end{align*}
Thus
\begin{align*}
     \partial_i E &=K_1(\partial_i E^{in})   
     - \frac{t}{4\pi}\int\int_{\mathbb{S}^2}p(1,\omega,\xi)\partial_if^{in}(x-t\omega,\xi)d\omega d\xi
     \nn
     &-\int p(t,x,\xi)Y(t,x) *_{t,x}(\mathbbm{1}_{t>0}T(\partial_if))d\xi + \int q(t,x,\xi)Y(t,x) *_{t,x}(\mathbbm{1}_{t>0}\partial_if)d\xi,
\end{align*}
and it follows, after multiplying by $\epsilon$ and replacing $T(\partial_if)$, that $\partial_iE$ can be estimated as follows
\begin{align} \label{lin inhomog di E estimate}
    |\epsilon \partial_i E(t,x)| &\leq C_T(R_\xi^0)||\epsilon \partial_i (f^{in}, E^{in}, B^{in})||_{L^\infty_{x,\xi}} 
    \nn
    &+ C_T(R_\xi^0)\int_0^t (1+\epsilon)||\partial_i f(s,\cdot)||_{L^\infty_{x,\xi}} + ||\epsilon\partial_t f(s,\cdot)||_{L^\infty_{x,\xi}} + ||\epsilon\nabla_x f(s,\cdot)||_{L^\infty_{x,\xi}} ds
     \nn
    &+ C_T(R_\xi^0)\int_0^t ||\epsilon \partial_i E(s,\cdot)||_{L^\infty_{x}} + 
  ||\epsilon^2E(s,\cdot)||_{L^\infty_{x}} ds 
\end{align}
and similarly for $B_i$. Thus adding (\ref{lin inhomog di f estimate}) and (\ref{lin inhomog di E estimate}) (and the similar expression for  $\partial_i B$) and applying Gr\"{o}nwall's lemma we have for prepared data
\begin{align} \label{ f E B x1 x2}
    ||\partial_{x_1}(f,\epsilon E,\epsilon B)(s,\cdot)||_{L^\infty_{x,\xi} } + ||\partial_{x_2}(f,\epsilon E,\epsilon B)(s,\cdot)||_{L^\infty_{x,\xi} } \leq C_T
\end{align}
For the momentum derivatives we again use a method from geometric optics. We can estimate $\nabla_\xi f$ using the cylindrical operator
\begin{align*}
    \nabla_\xi = e_\theta \frac{1}{r}\partial_\theta + e_r\partial_r + e_z\partial_z.
\end{align*} 
Remark that although the commutator $[\frac{1}{r} \partial_\theta,\partial_r] = \frac{1}{r^2}\partial_\theta \neq 0$, we do in fact have $[\partial_\theta,\partial_r] = 0$. Therefore, we first multiply (\ref{inhomog lin dilute model}) by $\vl{\xi}$, then apply $\partial_j$ with $j\in \{r,z\}$ and then divide once again by $\vl{\xi}$ leading to the expression
\begin{align} \label{ f j derivative}
    \partial_t(\partial_j f) &+ v(\xi)\cdot\nabla_x(\partial_j f) - \frac{b_e(x)}{\epsilon\vl{\xi}}\partial_{\theta}(\partial_j f) = \epsilon \partial_j\big( M'(|\xi|)\frac{\xi}{|\xi|}\big)\cdot E 
    \nn
    &-\partial_j(v(\xi))\cdot \nabla_x f +\partial_j(\ln(\vl{\xi}))\big[ \partial_t f + v(\xi)\cdot \nabla_x f - \epsilon M'(|\xi|)\frac{\xi}{|\xi|}\cdot E\big].
\end{align}
Assume that $M'(0) = 0$ and $M'(|\xi|) = \mathcal{O}(|\xi|)$ at $\xi = 0$ to ensure $|\frac{M'(|\xi|)}{|\xi|}|$ remains bounded at $|\xi| = 0$. This is satisfied as long as there is a differentiable extension of $M(\cdot)$ from $\mathbb{R}_+$ to $\mathbb{R}$. Thus integrating along the flow gives the estimate
\begin{align} \label{ r z derivs 1}
    |\partial_j f(t,x,\xi)| \leq ||\partial_j f^{in}||_{L^\infty_{x,\xi}} + C(R_\xi^0)\int_0^t ||\partial_t f(s,\cdot)||_{L^\infty_{x,\xi}} + ||\nabla_x f(s,\cdot)||_{L^\infty_{x,\xi}} + ||\epsilon E(s,\cdot)||_{L^\infty_{x}}ds.
\end{align}
For the $\theta$ derivative we simply apply the operator $\frac{1}{r}\partial_{\theta}$ directly since $\partial_\theta (\frac{1}{\vl{\xi}}) = 0$. This gives 
\begin{align*}
    \partial_t(\frac{1}{r}\partial_\theta f) + v(\xi)\cdot \nabla_x(\frac{1}{r}\partial_\theta f) - \frac{1}{\epsilon \vl{\xi}}\partial_\theta (\frac{1}{r}\partial_\theta f) = -\epsilon M'(|\xi|)\frac{\xi^\perp}{r|\xi|}\cdot E - \frac{\xi^\perp}{r\vl{\xi}}\cdot \nabla_x f.
\end{align*}
With the same assumptions on $M$, we then arrive at
\begin{align} \label{ theta derivs 1}
    |\frac{1}{r}\partial_{\theta}f(t,x,\xi)| \leq ||\frac{1}{r}\partial_{\theta}f^{in}||_{L^\infty_{x,\xi}} + ||\frac{M'(|\xi|)}{|\xi|}||_{L^\infty_\xi}\int_0^t||\epsilon E(s,\cdot)||_{L^\infty_x}ds + \int_0^t ||\nabla_x f(s,\cdot)||_{L^\infty_{x,\xi}}.
\end{align}
Thus adding (\ref{ r z derivs 1}) and (\ref{ theta derivs 1}) and using the established estimates (\ref{ f E B x1 x2}) and (\ref{ f E B t x3}) we can 
conclude for prepared data that $ ||\nabla_\xi f(t,\cdot)||_{L^\infty_{x,\xi}} \leq C_T $.
\end{proof}

\begin{remark} \label{derivative remark}
The previous proof relies heavily on the dilute assumption to ensure no loss of derivatives when estimating the fields. In order to control the initial data for $f$ along the flow a non-stationary phase argument was used. This implied derivatives on the initial data of $f$. However, the singular term coming from the source term of the Vlasov equation was required to be small (to ensure Lemma \ref{ dilute E source lemma}) in order to avoid this integration by parts step which would give a loss of derivatives when estimating the fields $(E,B)$.
\end{remark}

 %%%%%%%%%%%%%%%%%%%%%%%%%%%%%%%%%%%%%%%%%%%%

\subsection{Approximation of Dilute HMRVM System} \label{proof of theorem 1 section}
\smallskip
In this section we prove Theorem \ref{main theorem 2}. The goal is to use the results for the linear models in the previous section to deduce results 
for complete non-linear problem under the dilute assumption. Only the approximation given by Proposition \ref{main theorem 2} is needed for the following bootstrap 
argument. In this section we will denote $(f,E,B)$ as a solution to the Cauchy problem
\begin{align} \label{non linear dilute}
\begin{cases}
&\partial_t f + v(\xi)\cdot\nabla_x f - \frac{1}{\epsilon}[v(\xi)\times\mathbf{B}_e]\cdot\nabla_\xi f -\epsilon[E+v(\xi)\times B]\cdot \nabla_\xi f  =  \epsilon M'(|\xi|)\frac{\xi}{|\xi|}\cdot E
 \\
&\partial_t E -\nabla_x \times B = J(f), \quad \nabla_x\cdot E = -\rho(f)
\\
&\partial_t B + \nabla_x\times  E = 0, \ \  \qquad   \nabla_x \cdot B = 0
\\
&(f,E,B)|_{t=0} = (f^{in}, E^{in}, B^{in})
\end{cases}
\end{align}
Then let $(f_\ell,E_\ell, B_\ell)$ denote the solution to the associated linear system 
\begin{align} \label{linear dilute approx 2}
\begin{cases}
&\partial_t f_{\ell} + v(\xi)\cdot\nabla_x f_{\ell} -  \frac{1}{\epsilon}[v(\xi)\times\mathbf{B}_e]\cdot\nabla_\xi f =  \epsilon M'(|\xi|)\frac{\xi}{|\xi|}\cdot E_\ell
 \\
&\partial_t E_\ell -\nabla_x \times B_\ell = J(f_\ell), \quad \nabla_x\cdot E_\ell = -\rho(f_\ell)
\\
&\partial_t B_\ell + \nabla_x\times  E_\ell = 0, \ \  \qquad   \nabla_x \cdot B_\ell = 0
\\
&(f,E,B)|_{t=0} = (f^{in}, E^{in}, B^{in})
\end{cases}
\end{align}
In the dilute case, the ability to approximate the non-linear system with the linear version is due to the fact the first equation in (\ref{non linear dilute}) only $(\epsilon E, \epsilon B)$ appear instead of $(E,B)$. This allows for a linearization in the variable $f$ when the data is prepared. This guarantees the non linear term $\epsilon[E+v\times B] \cdot \nabla_\xi f$ remains small. Proposition \ref{ non linear approx with lin theorem} gives a precise relationship between (\ref{non linear dilute}) and (\ref{linear dilute approx 2}). Furthermore, Theorem \ref{main theorem 2} follows as a corollary of Proposition \ref{ non linear approx with lin theorem} following the continuation discussion of Section \ref{uniform time}.

\begin{proposition} \label{ non linear approx with lin theorem} (Local in time solution of non-linear system for prepared data) Let $(f_\ell,E_\ell, B_\ell)$ be a solution of 
(\ref{linear dilute approx 2}) with initial data $(f^{in}, E^{in}, B^{in})$ satisfying the compatibility (\ref{linear compatibilty}) and be prepared in the sense of 
(\ref{linear prepared homog}). Then there exits $T>0$ and $\epsilon_0 \in (0,1]$ such that for all $\epsilon \in (0,\epsilon_0]$ there is a unique solution $ (f_\epsilon,E_\epsilon,B_\epsilon)\in C^1([0,T];L^\infty_{x,\xi})$ to (\ref{non linear dilute}) such that $(f_\epsilon,E_\epsilon,B_\epsilon)|_{t=0} = (f^{in}, E^{in}, B^{in})$ 
and a constant $C_T$ depending on $||(f^{in}, E^{in}, B^{in})||_{W^{1,\infty}_{x,\xi}}$ such that for all $t \in [0,T]$,
\begin{align} \label{linear dilute approximation estimate}
    ||(f_\epsilon - f_\ell)(t)||_{L^\infty_{x,\xi}} &\leq \epsilon C_T,
    \nn
    ||(E_\epsilon - E_\ell,B_\epsilon- B_\ell) (t)||_{L^\infty_{x,\xi}} &\leq C_T.
\end{align}
Moreover,
\begin{align}
    ||\partial_t (f_\epsilon - f_\ell, \epsilon (E_\epsilon - E_\ell), \epsilon(B_\epsilon - B_\ell))(t,\cdot)||_{L^\infty_{x,\xi}}&+ ||\nabla_x (f_\epsilon - f_\ell, \epsilon (E_\epsilon - E_\ell), \epsilon(B_\epsilon - B_\ell))(t,\cdot)||_{L^\infty_{x,\xi}} 
    \nn
    &+ ||\nabla_\xi (f_\epsilon - f_\ell)(t,\cdot)||_{L^\infty_{x,\xi}} \leq C_T
\end{align}
\end{proposition}
\begin{proof}
First consider the anzatz
\begin{align} \label{ delta anzatz}
    f_\epsilon &= f_\ell + \epsilon f^\delta,
    \nn
    (E_\epsilon,B_\epsilon) &= (E_\ell,B_\ell) + (E^\delta,B^\delta).
\end{align}
Here, the $\delta$ in $(f^\delta,E^\delta,B^\delta)$ is not a parameter, but instead symbolizes a difference of solutions. Therefore $(f^\delta,E^\delta,B^\delta)$ satisfies
\begin{align} \label{dilute f delta}
\begin{cases}
&\partial_t f^\delta + v(\xi)\cdot\nabla_x f^\delta - \epsilon^{-1}[v(\xi)\times\mathbf{B}_e(x)]\cdot\nabla_\xi  f^\delta -\epsilon[E^\delta +  E_\ell +v(\xi)\times (B^\delta + B_\ell)]\cdot \nabla_\xi f^\delta  
\\&\qquad =   M'(|\xi|)\frac{\xi}{|\xi|}\cdot E^\delta  +  [E_\ell + E^\delta + v(\xi)\times (B_\ell+B^\delta)]\cdot\nabla_\xi f_\ell 
 \\
&\partial_t E^\delta -\nabla_x \times B^\delta = \epsilon J(f^\delta), \quad \nabla_x\cdot E^\delta = -\epsilon \rho(f^\delta)
\\
&\partial_t B^\delta + \nabla_x\times  E^\delta = 0, \ \  \qquad   \nabla_x \cdot B^\delta = 0
\end{cases}
\end{align}
Remark that the dilute assumption is key here, otherwise the right hand side of the equation on $f^\delta$ would involve $\epsilon^{-1}M'(|\xi|)\frac{\xi}{|\xi|}\cdot E^\delta$. Furthermore, the difference of scaling between $f^\delta$ and $(E^\delta, B^\delta)$, introduces $\epsilon J$ and $\epsilon \rho$ in the current and charge density. The methods of \cite{cold} can then be repeated to obtain uniform estimates without the difficulty involved in passing the transport operator $T(f^\delta)$ to a $\xi$ derivative which can be integrated by parts. First note that $f^\delta$ can be easily integrated along the full, non-linear flow $\mathcal{F}$ associated with the characteristics of (\ref{dilute f delta}). Recall  that $f^\delta|_{t=0}\equiv 0$, then it follows by the Duhamel Principle 
\begin{align*}
f^\delta(t,x,\xi) = \int_0^t\bigg[  M'(|\xi|)\frac{\xi}{|\xi|}\cdot E^\delta  +  [E_\ell + E^\delta + v\times (B_\ell+B^\delta)]\cdot\nabla_\xi f_\ell \bigg]\circ \big(s,\mathcal{F}(t-s,x,\xi)\big)ds.
\end{align*}
This gives the estimate
\begin{align} \label{ f delta estimate}
    |f^\delta(t,x,\xi)|&\leq  ||M'||_{L^\infty} \int_0^t ||E^\delta(s,\cdot)||_{L^\infty_x}ds 
    \nn
    &+ \int_0^t||\nabla_\xi f_\ell(s,\cdot)||_{L^\infty_{x,\xi}}\big(||(E^\delta, B^\delta)(s,\cdot)||_{L^\infty_{x}} + ||(E_\ell, B_\ell)(s,\cdot)||_{L^\infty_{x}}\big)ds.
\end{align}
This is precisely where the prepared data assumption is needed. It is to ensure uniform control on $||\nabla_\xi f_\ell(s,\cdot)||_{L^\infty_{x,\xi}}$. Next, due to the compensation of $\epsilon$ on the current and charge density, the fields $(E^\delta,B^\delta)$ satisfy the wave equation
\begin{align*}
    \square E^\delta &= \epsilon \int v(\xi)\partial_t f^\delta + \nabla_x f^\delta d\xi,
    \nn
    \square B^\delta &= \epsilon \int \nabla_x\times (v(\xi)f^\delta )d\xi.
\end{align*}
Again recalling $(f^\delta,E^\delta,B^\delta)|_{t=0}\equiv 0$, the solution to the fields is given by 
\begin{align*}
    E^\delta(t,x) &= - \epsilon \int p(t,x,\xi)Y(t,x)*_{t,x}(\mathbbm{1}_{t>0}T(f^\delta))d\xi + \epsilon \int q(t,x,\xi)Y(t,x)*_{t,x}(\mathbbm{1}_{t>0}f^\delta)d\xi, 
    \\
     B^\delta(t,x) &=  \epsilon \int a^0(t,x,\xi)Y(t,x)*_{t,x}(\mathbbm{1}_{t>0}T(f^\delta))d\xi + \epsilon \int a^1(t,x,\xi)Y(t,x)*_{t,x}(\mathbbm{1}_{t>0}f^\delta)d\xi.
\end{align*}
Replacing $T(f^\delta)$ with the Vlasov equation and integrating by parts in $\xi$ we can estimate $E$ by
\begin{align} \label{E delta estimate}
    |E^\delta(t,x)| &\leq C(R_\xi^T, ||b_e||_{L^\infty})\int_0^t ||f^\delta(s,\cdot,\cdot)||_{L^\infty_{x,\xi}}ds 
    \nn
    &+  \epsilon C(R_\xi^T) \int_0^t  || f^\delta(s,\cdot,\cdot)||_{L^\infty_{x,\xi}}\big(1 + \epsilon|| (E^\delta,B^\delta)(s,\cdot)||_{L^\infty_x} + \epsilon|| (E_\ell,B_\ell)(s,\cdot)||_{L^\infty_x}\big) ds
    \nn
    &+  \epsilon C(R_\xi^T) \int_0^t  || f_\ell(s,\cdot,\cdot)||_{L^\infty_{x,\xi}}\big(1 + \epsilon|| (E^\delta,B^\delta)(s,\cdot)||_{L^\infty_x} + \epsilon|| (E_\ell,B_\ell)(s,\cdot)||_{L^\infty_x}\big) ds
    \nn
    &+\epsilon C(R_\xi^0) ||M'||_{L^\infty}\int_0^t ||E^\delta(s,\cdot)||_{L^\infty}ds.
\end{align}
And similarly of $B^\delta$. Therefore adding (\ref{E delta estimate}) and (\ref{ f delta estimate}) and applying Gr\"{o}nwall's lemma gives the result 
(\ref{linear dilute approximation estimate}).

Next we must control the Lipschitz estimates to justify solving the characteristic curves of $f^\delta$. This can be done for $f^\delta$ after straightening 
to allow for the commutation of spatial and momentum derivatives with the variable coefficient $\epsilon^{-1}[v(\xi)\times\mathbf{B}_e(x)]$. Therefore we consider
\begin{align}
    \bar{f}^{\delta}(t,x,\xi) := f^\delta(t,x,O(x)\xi),
\end{align}
which satisfies 
\begin{align}
    \partial_t \bar{f}^\delta &+ v(O(x)\xi)\cdot\nabla_x \bar{f}^\delta - \frac{b_e(x)}{\epsilon\vl{\xi}}\partial_\theta  \bar{f}^\delta -\epsilon[O^t(x)(E^\delta +  E_\ell) +v(\xi)\times O^t(x)(B^\delta + B_\ell)]\cdot \nabla_\xi \bar{f}^\delta  
\\&\qquad =   M'(|\xi|)\frac{O(x)\xi}{|\xi|}\cdot E^\delta  +  [O^t(x)(E_\ell + E^\delta) + v(\xi)\times O^t(x)(B_\ell+B^\delta)]\cdot\nabla_\xi f_\ell .
\end{align}
Differentiating $\bar{f}^\delta$ with respect to $i$, for $i = t,x_1, x_2$ or $x_3$, multiplying by $\epsilon$ and using our trick from geometric optics, we have the following 
\begin{align}
    &\partial_t(\epsilon \partial_i \bar{f}^\delta) + v(O(x)\xi)\cdot \nabla_x(\epsilon \partial_i \bar{f}^\delta) - \frac{b_e(x)}{\epsilon\vl{\xi}}\partial_\theta (\epsilon \partial_i \bar{f}^\delta) 
    \nn
    &- \epsilon [O^t(x)(E^\delta + E_\ell) + v(\xi)\times(O^t(x)(B^\delta + B_\ell))]\cdot \nabla_\xi (\epsilon \partial_i \bar{f}^\delta) 
    \nn
    &= \epsilon [O^t(x)(E^\delta + E_\ell) + v(\xi)\times(O^t(x)(B^\delta + B_\ell))]\cdot \nabla_\xi ( \partial_i f_\ell ) 
    \nn
    &+ \epsilon[\partial_i(O^t(x)( E^\delta + E_\ell)) + v(\xi)\times(\partial_i(O^t(x)( B^\delta + B_\ell))]\cdot \nabla_\xi (\epsilon \bar{f}^\delta + f_\ell) 
    \nn
    &+ \partial_i \ln(b_e)\bigg[ \partial_t (\epsilon \bar{f}^\delta) + v(O(x)\xi)\cdot\nabla_x (\epsilon \bar{f}^\delta) -\epsilon[O^t(x)(E^\delta +  E_\ell) +v(\xi)\times O^t(x)(B^\delta + B_\ell)]\cdot \nabla_\xi (\epsilon \bar{f}^\delta) 
\nn&\qquad -  \epsilon M'(|\xi|)\frac{O(x)\xi}{|\xi|}\cdot E^\delta  -  \epsilon [O^t(x)(E_\ell + E^\delta) + v(\xi)\times O^t(x)(B_\ell+B^\delta)]\cdot\nabla_\xi f_\ell \bigg] 
\\
&+\epsilon M'(|\xi|)\frac{O^t(x)\xi}{|\xi|}\cdot(\partial_iE^\delta + \partial_i E_\ell) + +\epsilon M'(|\xi|)\frac{\partial_i O^t(x)\xi}{|\xi|}\cdot(\partial_iE^\delta + \partial_i E_\ell)+ v(\partial_iO(x)\xi)\cdot\nabla_x \epsilon \bar{f}^\delta.
\end{align}
Integrating along the flow (the left-hand-side of the above, involving no derivatives on the fields) allows us to estimate $\epsilon \partial_i \bar{f}^\delta$ in terms of $(E^\delta, B^\delta)$, $\epsilon \partial_i(E^\delta, B^\delta)$ and derivatives of $\epsilon \bar{f}^\delta$. But remark the main difference is we now must also control
\begin{align}
    ||\epsilon  \nabla_\xi (\partial_i f_\ell)||_{L^\infty_{x,\xi}} \lesssim C_T.
\end{align}
However, this is easily verified by repeating the proof given in the paragraph \textit{End of proof of Proposition 1} to justify weighted uniform estimates of $\epsilon\nabla_\xi (\partial_i f_\ell)$ for prepared data.
For the fields, we return to $f^\delta$, so similarly we have
\begin{align}
    \square_{t,x} (\epsilon \partial_i E^\delta) &= \epsilon \int \big[v(\xi)\partial_t(\epsilon \partial_i f^\delta) + \nabla_x(\epsilon \partial_i f^\delta)\big] d\xi.   
\end{align}
Moreover, repeating the the arguments used in the paragraph \textit{End of proof of Proposition 1} for the equations on $ (\epsilon\partial_r\bar{f}^\delta )$, $ (\epsilon\partial_z \bar{f}^\delta )$ and $  (\epsilon\frac{1}{r}\partial_\theta \bar{f}^\delta)$ (i.e. multiplying by $\vl{\xi}$, differentiating and multiplying again by $\vl{\xi}$) and applying Gr\"{o}nwall's lemma then yields 
 \begin{align}
     ||\partial_t (\epsilon f^\delta, \epsilon E^\delta, \epsilon B^\delta)(t,\cdot)||_{L^\infty_{x,\xi}} + ||\nabla_x (\epsilon f^\delta, \epsilon E^\delta, \epsilon B^\delta)(t,\cdot)||_{L^\infty_{x,\xi}} + ||\nabla_\xi (\epsilon f^\delta)(t,\cdot)||_{L^\infty_{x,\xi}} \leq C_T.
 \end{align}
Recalling the additional weight $\epsilon f^\delta := f - f_\ell$, this proves the estimate (\ref{weighted Lip}).

\end{proof}
Remark that although for prepared data one does have $|\nabla_\xi f_\ell(t)| \lesssim C$, it is not true in general that $|\partial_{\xi_i}\partial_{\xi_j}f_\ell(t)|$ will also remain uniformly bounded. In fact we require $|\partial_{\theta}f^{in}| \lesssim \epsilon^2C$ to obtain such estimates. This can be seen from Example \ref{ example 1}.

%%%%%%%%%%%%%%%%%%%%%%%%%%%%%%%%%%%%%%%

\subsection{Uniform Lower Bound on Time of Existence} \label{uniform time}
In this subsection we use the a priori uniform estimates of Proposition \ref{ non linear approx with lin theorem} to prove a uniform time of existence of the system 
(\ref{vlasov 1})-(\ref{maxw2}) to complete the proof of Theorem \ref{main theorem 2}. Following the lines of \cite{glassey}, $C^1([0, T_\epsilon);L^\infty_{x,\xi})$ 
solutions do exist on a time interval $[0,T_\epsilon)$ as long as $f$ remains compactly supported. Again retain that we assume $T_\epsilon$ is the maximal time 
of existence. The proof in \cite{glassey} (in the absence of external fields) is accomplished through a Picard iterative scheme, which constructs a Cauchy sequence 
of linear PDEs converging in some Banach space $C^1([0, \tilde{T}];L^\infty_{x,\xi})$ to the non-linear solution.  To complete the proof of Theorem \ref{main theorem 2} 
we set $R^\infty > R_\xi^0$. Then choose $\mathcal{T}_\epsilon >0 $ to be the maximal time in $[0, T_\epsilon)$ such that 
\begin{align} \label{supp f}
    \forall t \leq \mathcal{T}_\epsilon, \ \text{supp}(f(t,\cdot)) \subset \{ (x,\xi) \ | \ |x| \leq R_x^0  + t, \ |\xi| \leq R^\infty \} .
\end{align}
Thus for all $t \leq \mathcal{T}_\epsilon$, the a priori bounds of Proposition \ref{ non linear approx with lin theorem} hold with the uniform estimate depending on $R^\infty$. In particular, the momentum characteristic curve solving (\ref{endXode}) satisfies
\begin{align*}
    \dt{t}|\Xi|^2(t) = -\Xi\cdot \epsilon E(t,X)
\end{align*}
and therefore by continuity of the flow we have 
\begin{align*}
    \forall t\in [0,\mathcal{T}_\epsilon), \quad   |\Xi|^2(t) \leq |\xi|^2 + \epsilon  tC \sup_{s\in[0,t]}||E(s,\cdot)||_{L^\infty} \leq |\xi|^2 + \epsilon C_t(f^{in}, E^{in},B^{in}) t,
\end{align*}
Such that $C_t$ is a continuous function (for fixed initial data) and satisfies 
\begin{align}
    \lim_{t\rightarrow 0^+}  C_t(f^{in}, E^{in},B^{in}) t = 0.
\end{align}
Thus we may define $T$ satisfying
\begin{align*}
    0 < T  <\frac{(R^\infty)^2 - (R_\xi^0)^2}{ C_T(f^{in}, E^{in},B^{in})} <  \frac{(R^\infty)^2 - (R_\xi^0)^2}{\epsilon C_T(f^{in}, E^{in},B^{in})}. 
\end{align*}
So in particular for all $t < \min\{T, \mathcal{T}_\epsilon\}$ we must have (\ref{supp f}) holds.  Although this may imply we can choose $T$ very large, we must further adjust $T$ such that the inequality (\ref{ T inequality}) holds to ensure the use of the non-stationary phase argument. Since we have assumed $\mathcal{T}_\epsilon$ to be maximal this implies $0 < T \leq \mathcal{T}_\epsilon \leq T_\epsilon$. 

%%%%%%%%%%%%%%%%%%%%%%%%%%%%%%%%%%%
%%%%%%%%%%%%%%%%%%%%%%%%%%%%%%%%%%%

\section{Appendix : Characteristic approximation for Spatially Varying Fields} \label{Appendix}
\subsection{General Formulation of Lemmas \ref{X lemma} and \ref{f initial non stationary}}
In this appendix we prove a similar result to Lemma \ref{X lemma} and its use in Lemma \ref{f initial non stationary}, when the direction of the magnetic field is not fixed. Remark \ref{ Q remark} is crucial as it implies the size of the momentum characteristics remain constant. Furthermore, we can use this result to show Proposition \ref{linear inhomog theorem} still remains valid when the direction of $\mathbf{B}_e$ is allowed to vary. We begin with the following ODE of the characteristics after straightening.
\begin{align}
    \dot{X} &= \frac{O(X) \Xi}{\vl{\Xi}} &&X(0) = x
    \\
    \dot{\Xi} &= -\frac{b_e(X)}{\epsilon\vl{\Xi}}\Xi^\perp + \frac{Q(X,\Xi)}{\vl{\Xi}} &&\Xi(0) = \xi
\end{align}

Then consider converting this system into polar coordinates with respect to $\Xi$, by letting
\begin{align}
    \Xi(t) := (R(t)\cos(\Theta(t)),R(t)\sin(\Theta(t)),Z(t)), \quad \Xi(0) = \xi = (r\cos(\theta),r\sin(\theta),z).
\end{align}

Remark that since $\xi\cdot Q\equiv 0$, we have $|\Xi(t)| =\sqrt{R(t)^2 + Z(t)^2} = \sqrt{r^2+z^2} = |\xi|$ for all $t\geq 0$. This means we may ignore the discontinuity (in $\theta$) at $\xi = 0$ since $(X,\Xi)(t) = (x,0)$ is the unique solution when $\xi = 0$. Thus we only consider $\xi \neq 0$. Note, in the fixed direction case, we further have $R(t) = r$ and $Z(t) = z$ for all $t$. Furthermore $|X -x|\leq t$, so the solution is globally defined in time. In these new variables our system then becomes 
\begin{align} \label{polar vary start}
\dot{X} &= \frac{1}{\vl{\xi}} O(X)\BM{R\cos(\Theta) \\ R\sin(\Theta) \\ Z}, && X(0) =x,
    \\
    \dot{R} &= \BM{\cos(\Theta) \\ \sin(\Theta) \\ 0}\cdot \frac{Q(X,\Xi)}{\vl{\xi}}, && R(0) = r,
    \\
    \dot{Z} &= e_3\cdot \frac{Q(X,\Xi)}{\vl{\xi}}, && Z(0) =z,
    \\
    \dot{\Theta} &= -\frac{b_e(X)}{\epsilon\vl{\xi}} + \BM{\sin(\Theta) \\ -\cos(\Theta) \\ 0}\cdot \frac{Q(X,\Xi)}{\vl{\xi}},&& \Theta(0) = \theta. \label{polar vary end}
\end{align}
In these variables, we have the following approximating result.
\begin{lemma} \label{general field lemma}
	Fix $T>0$. Then there exists a constant $0 < C = C(|\xi|,x,||b_e||_{W^{2,\infty}}, T) < \infty $ and approximations $(X_2,R_2,Z_2)\in C^2$ and $\Theta_1 \in C^2$ such that for all $t\in [0,T]$ 
	\begin{align} \label{char est 2}
		|(X,R,Z)(t) - (X_2,R_2,Z_2)(t)| \leq \epsilon^2 C, \quad |\Theta(t) - \Theta_1(t)| \leq \epsilon C.
	\end{align}
	Furthermore, we have the Lipschitz estimates
	\begin{align} \label{X 2 lip}
		|\partial_t (X_2,R_2,Z_2)(t)| + |D_x(X_2,R_2,Z_2)(t)| \leq C. 
	\end{align}
	Moreover, for any $\omega\in \mathbb{S}^2$,  $\Theta_1$ satisfies 
	\begin{align} \label{Theta 1 est 1}
		\forall t < \min\{C^{-1}, T\}, \quad  \big|\frac{1}{\partial_t \Theta_1 + \omega \cdot \nabla_x \Theta_1}\big| \leq \frac{\epsilon C}{1 - tC}
	\end{align}
and
\begin{align} \label{Theta 1 est 2}
	|\epsilon \partial_t^2 \Theta_1| + |\epsilon D_x^2 \Theta_1| \leq C.
\end{align}
Together, (\ref{Theta 1 est 1}) - (\ref{Theta 1 est 2}) imply the crucial estimate 
\begin{align}
	\bigg|\dt{s}\bigg(\frac{1}{\partial_t \Theta_1 + \omega \cdot \nabla_x \Theta_1}(t-s,x-s\omega,\xi)\bigg)\bigg| \leq \frac{\epsilon C^2}{(1-tC)^2} .
\end{align}
As a consequence for any $g\in C^2(\mathbb{R}_+\times\mathbb{R}^3\times\mathbb{R}_+\times\mathbb{R};\mathbb{R})$ and $n\in \mathbb{Z}-\{0\}$ we have
\begin{align} \label{integral average}
	\forall t< \min\{T,C^{-1}\}, \ \int_0^t g(s,(X,R,Z)&(t-s,x-s\omega,\xi))e^{in\Theta(t-s,x-s\omega,\xi)}ds
	\nn
	&\leq \frac{\epsilon t}{n} \bigg[ ||D_{x,r,z}g||_{L^{\infty}}(n+ \frac{ C^2}{(1-tC)^2}) + ||g||_{L^\infty}\frac{ C}{1 - tC} \bigg].
\end{align}
\end{lemma}
\begin{proof}
\textbf{Step 1: A first order approximation of $(X,R,Z)$:} The idea is the same as Lemma \ref{X lemma}. But instead we now must consider $(X,R,Z)$ instead of just $X$ in our decomposition since $R$ and $Z$ are no longer constant. For the right hand side of (\ref{polar vary start})-$\cdots$-(\ref{polar vary end}), consider the decomposition into the mean and periodic part with respect to $\Theta$.
\begin{align}
\frac{1}{\vl{\xi}} O(X)\BM{R\cos(\Theta) \\ R\sin(\Theta) \\ Z} &:= \bar{V}(X,Z) + \partial_{\theta}V^*(X,R,Z,\Theta),
\nn
    \BM{\cos(\Theta) \\ \sin(\Theta) \\ 0}\cdot \frac{Q(X,\xi)}{\vl{\xi}} &:= \vl{\xi}^{-1}[\bar{Q}_r(X,R,Z) + \partial_{\theta}Q^*_r(X,R,Z,\Theta)],
    \nn
    e_3\cdot \frac{Q(X,\xi)}{\vl{\xi}} &:= \vl{\xi}^{-1}[\bar{Q}_z(X,R,Z) + \partial_{\theta}Q^*_z(X,R,Z,\Theta)],
    \nn
    -\frac{b_e(X)}{\epsilon\vl{\xi}} + \BM{\sin(\Theta) \\ -\cos(\Theta) \\ 0}\cdot \frac{Q(X,\xi)}{\vl{\xi}} &:= -\frac{b_e(X)}{\epsilon\vl{\xi}} + \vl{\xi}^{-1}[\bar{Q}_\theta(X,R,Z) + \partial_{\theta} Q^*_\theta(X,R,Z,\Theta)],
\end{align}
where $\partial_{\theta}V^*, \ \partial_{\theta}Q_r^*, \  \partial_{\theta}Q_z^*$ and $\partial_{\theta}Q_\theta^*$ are $2\pi-$periodic in $\theta$ and of zero mean. The derivative $\partial_\theta$, will be used for convenience of notation and we may also assume the antiderivatives $V^*$, $Q_r^*$ and $Q_\theta^*$ also have zero mean. The terms involving $Q$ are less explicit, but for $\bar{V}$ and $\partial_\theta V^*$ we have a similar representation as before
\begin{align} \label{V defin}
    \bar{V}(X,Z) &= \frac{Z}{\vl{\xi}}O(X)e_3 = \frac{Z}{\vl{\xi}}\frac{\mathbf{B}_e(X)}{b_e(X)},
    \\ \label{V star def}
    \partial_\theta V^*(Z,R,Z,\Theta) &= \frac{R}{\vl{\xi}}O(X)\BM{\cos(\Theta) \\ \sin(\Theta) \\ 0}.
\end{align}
For a general $2\pi-$periodic function $\partial_\theta G^* \in \{\partial_\theta V^*, \ \vl{\xi}^{-1}\partial_\theta Q^*_r,\ \vl{\xi}^{-1}\partial_\theta Q^*_z\} $, the main idea to decompose the system (\ref{polar vary start})-$\cdots$-(\ref{polar vary end}) is to write the oscillating parts, composed with $(X,R,Z,\Theta)$, as follows
\begin{align}
    \partial_{\theta}G^*(X,R,Z,\Theta) = \frac{1}{\dot{\Theta}}\dt{t}[G^*(X,R,Z,\Theta)] - \frac{1}{\dot{\Theta}} \big[(D_{x,r,z}G^*)(X,R,Z,\Theta)\big](\dot{X},\dot{R},\dot{Z}).
\end{align}
Again we have that $\dot{\Theta}^{-1}$ is small, of size $ \epsilon $. To see why, we consider the Taylor expansion 
\begin{align}
    |u| <1, \ (1-u)^{-1} = 1+u+u^2+u^3 +...
\end{align}
Thus for $\epsilon$ small enough we have
\begin{align} \label{1 over theta dot}
    \frac{1}{\dot{\Theta}} &= -\epsilon\frac{\vl{\xi}}{b_e(X)}\bigg[\frac{1}{1 - \epsilon(\bar{Q}_\theta + \partial_\theta Q^*_\theta)/ b_e(X)}\bigg](X,R,Z,\Theta)
    \nn
    &= -\epsilon \frac{\vl{\xi}}{b_e(X)}\bigg[ 1 + \epsilon(\bar{Q}_\theta + \partial_\theta Q^*_\theta) / b_e(X) + \epsilon^2 
    (\bar{Q}_\theta + \partial_\theta Q^*_\theta)^2/ b_e(X)^2 \bigg](X,R,Z,\Theta) + \mathcal{O}(\epsilon^3). 
\end{align}
We can integrate the system (\ref{polar vary start})-$\cdots$-(\ref{polar vary end}) in time, and the oscillating terms by parts. Remark that 
$| Q(x,\xi)/\vl{\xi}| \leq C(\mathbf{B}_e(x))|\xi|$ and $|v(\Xi)| <1 $, and thus for $\partial_\theta G^*\in \{\partial_\theta V^*,\vl{\xi}^{-1}\partial_\theta 
Q^*_r,\vl{\xi}^{-1}\partial_\theta Q^*_z\}$, noting that $ \vert \ddot{\Theta} \vert  \lesssim \epsilon^{-1} $, the oscillating terms become
\begin{align} \label{oscilating term}
  \int_0^t \partial_{\theta}G^*(X,R,Z,\Theta)ds &=  \frac{1}{\dot{\Theta}}G^*(X,R,Z,\Theta)|_{s=0}^{s=t} + \int_0^t \frac{\ddot{\Theta}}{\dot{\Theta}^2} G^*(X,R,Z,\Theta)ds
  \nn
  &-\int_0^t\frac{1}{\dot{\Theta}} (D_{x,r,z}G^*)(X,R,Z,\Theta)(\dot{X},\dot{R},\dot{Z})ds
 \nn
 &\lesssim \epsilon   C(|\xi|,x,\mathbf{B}_e,t).
\end{align}
Compare this with (\ref{ eps X estimate}). Thus we consider the first order in $\epsilon$ approximation of $(X,R,Z)(t)$. Let $(X_1,R_1,Z_1)$ be the solution to
\begin{align}
    \dot{X}_1 &= \bar{V}(X_1,R_1,Z_1), &&X_1(0) = x,
    \nn
    \dot{R}_1 &= \bar{Q}_r(X_1,R_1,Z_1), && R_1(0) = r,
    \nn
    \dot{Z}_1 &=\bar{Q}_z(X_1,R_1,Z_1), && Z_1(0) = z.
\end{align}
It easily follows that
\begin{align} \label{ X,R EST}
    |(X,R,Z)(t) - (X_1,R_1,Z_1)(t)| \leq \epsilon   C(|\xi|,x,\mathbf{B}_e,t).
\end{align}
Noting (\ref{V defin}), in the case when $\mathbf{B}_e || e_3$, then the above is given by
\begin{align}
    \dot{X}_1 &= \frac{Z_1}{\vl{\xi}}, && X_1(0) = x,
    \nn
    \dot{R}_1 &= 0, &&R_1(0) =r,
    \nn
    \dot{Z}_1 &= 0, && Z_1(0) = z.
\end{align}
This gives the trivial order $\epsilon$ approximation of $(X,R,Z)$ as $(x+ zt /\vl{\xi} e_3,r,z)$.

\smallskip

\textbf{Step 2: A second order approximation of $(X,R,Z)$:} Just as before, we need to approximate $(X,R,Z)$ up to order $\epsilon^2$, to get our order 
$\epsilon$ approximation for $\Theta$. For ease of notation, we combine the variables $(X,R,Z)(t)\in \mathbb{R}^3\times\mathbb{R}_+\times\mathbb{R}$ into one.
\begin{align}
	U := (X,R,Z), \quad \partial_\theta F^*(U,\Theta) := (\partial_\theta V^*, \partial_\theta Q_r^*, \partial_\theta Q^*_z)(U,\Theta), \quad \bar{F}(U) := (\bar{V}, \bar{Q}_r,\bar{Q}_z).
\end{align}
By definition we have that $(U, \Theta)$ solves 
\begin{align}
	\dot{U} &= \bar{F}(U) + \partial_\theta F^*(U,\Theta), && U(0) = u = (x,r,z) ,
	\nn
	\dot{\Theta} &= -\frac{b_e(X)}{\epsilon \vl{\xi}} + \vl{\xi}^{-1}[\bar{Q}_\theta(U) + \partial_\theta Q^*_\theta(U,\Theta)] && \Theta(0) = \theta .
\end{align}
 Remark the definition (\ref{V star def}), and the fact that $Q$ is quadratic in $\xi$, it follows that each component of $\partial_\theta F^*$ is a sum of terms with the separable form $F_i^*(U)\partial_\theta P_i(k_i \Theta)$ where $P_i \in \{\sin, \cos\}$ and $k_i=1,2$, and $i\in \{1,2,..., N\}$ for some integer $N$. We claim, we may decompose these components in the following way
\begin{align} \label{F decomp}
\partial_\theta F^*(U,\Theta) &= \sum_{i=1}^{N} F_i^*(U)\partial_\theta P_i(k_i \Theta)
\nn
& =
 \sum_{i=1}^N\bigg( \epsilon \dt{t}\big[ A_{i,1}(U,\Theta) \big] + \epsilon A_{i,2}(U) + \epsilon^2 \dt{t}\bigg[\sum_{j = 3}^M A_{i,j}(U)P_{j}(k_j\Theta)\bigg]\bigg) + \mathcal{O}(\epsilon^2)	
\end{align}
In the above $k_j\in\{1,2,3\}$. To demonstrate that this is possible, consider $\sin(\Theta)$ and the expansion (\ref{1 over theta dot}). Keeping only terms of size $1$ and $\epsilon$ we have the following.
\begin{align*}
	\sin (\Theta) &= -\frac{1}{\dot{\Theta}}\dt{t}[\cos(\Theta)]
	\nn
	&= \epsilon \frac{\vl{\xi}}{b_e(X)}\bigg[ 1 + \epsilon(\bar{Q}_\theta + \partial_\theta Q^*_\theta)/ b_e(X) \bigg]\dt{t}[\cos(\Theta)] + \mathcal{O}(\epsilon^2)
	\nn
	&= \epsilon \frac{\vl{\xi}}{b_e(X)} [1 + \epsilon \bar{Q}_\theta/ b_e(X) ]\dt{t}[\cos(\Theta)] 
	\nn
	&\ \ \ - \epsilon^2 \frac{\vl{\xi} \partial_\theta Q^*_\theta}{b_e(X)^2} \, \Bigl[-\frac{b_e(X)}{\epsilon \vl{\xi}} + \vl{\xi}^{-1}(\bar{Q}_\theta + \partial_\theta Q^*_\theta) 
	\Bigr] \, \sin(\Theta) + \mathcal{O}(\epsilon^2) .
 \end{align*}
Define new functions $ A $, $ \bar{P} $ and $ P^* $ according to
\begin{align}
	A(U) &:=  \frac{\vl{\xi}}{b_e(X)}  [1 + \epsilon \bar{Q}_\theta/ b_e(X) ] ,
	\nn
	\frac{ \partial_\theta Q^*_\theta(U,\Theta) \, \sin(\Theta)}{b_e(X)} &=: \bar{P}(U) - \partial_\theta P^*(U,\Theta).
\end{align}
Modulo $ \epsilon^2 $, there remains
 \begin{align}
	\sin (\Theta) &= \epsilon A(U)\dt{t}[\cos(\Theta)] + \epsilon\bar{P}(U)  + \epsilon^2\frac{\vl{\xi}}{b_e(X)} \dt{t}[P^*(U,\Theta)] +  \mathcal{O}(\epsilon^2).
 \end{align}
 Since $\dot{U}$ is bounded, we can commute the time derivative on the first term above 
\begin{align}
	\epsilon A(U)\dt{t}[\cos(\Theta)]  = \epsilon \dt{t}[A(U)\cos(\Theta)] - \epsilon [\dot{U}\cdot \nabla A(U)] \cos(\Theta).
\end{align}
In a similar way, we combine the periodic and zero mean terms of $\dot{U}\cos(\Theta)$. It follows that $\sin(\Theta)$ can be written in the following way 
\begin{align}
	\sin(\Theta) = \epsilon \dt{t}\big[ \tilde{A}_1(U,\Theta) \big] + \epsilon \tilde{A}_2(U) + \epsilon^2 \dt{t}\sum_{j=3} \tilde{A}_j(U)P_j(k_j\Theta) + \mathcal{O}(\epsilon^2).
\end{align}
Any multiple, depending only on $U$ can be filtered in a similar way. Thus the there exists a decomposition in the form of (\ref{F decomp}). The next step is to absorb the time derivatives as follows
\begin{align} \label{ U tilde}
	\dt{t}\big[ U - \epsilon \sum_{i=1}^n A_{i,1}(U,\Theta)\big] = \bar{F}(U) +  \sum_{i=1}^N\bigg(\epsilon A_{i,2}(U) + \epsilon^2 \dt{t}\bigg[\sum_{j = 3}^M A_{i,j}(U)P_j(k_j\Theta)\bigg]\bigg) + \mathcal{O}(\epsilon^2).
\end{align} 
Then define $\tilde{U}$ related to the left hand side above as follows
\begin{align} \label{U tilde defin}
	\tilde{U} :=  U - \epsilon \sum_{i=1}^n A_{i,1}(U,\Theta).
\end{align}
Next we Taylor expand the $A_{i,2}$ terms in (\ref{ U tilde}) using (\ref{U tilde defin}). This yields addition oscillating terms of size $\epsilon$. These can be absorbed into the $A_{i,j}$ terms for $j\geq 3$. That is for, appropriately defined $\tilde{A}_{i,j}$ we obtain
\begin{align} \label{ U tilde exp}
	\dt{t} \tilde{U} &= \bar{F}(U) + \sum_{i=1}^N\bigg(\epsilon A_{i,2}(U) + \epsilon^2 \dt{t}\bigg[\sum_{j = 3}^M A_{i,j}(U)P_k(U,\Theta)\bigg]\bigg) + \mathcal{O}(\epsilon^2)
	\nn
	&= \bar{F}(\tilde{U} ) + \sum_{i=1}^N\bigg(\epsilon A_{i,2}(\tilde{U}) + \epsilon^2 \dt{t}\bigg[\sum_{j = 3}^{\tilde{M}} \tilde{A}_{i,j}(U)P_k(U,\Theta)\bigg]\bigg) + \mathcal{O}(\epsilon^2).
\end{align}
We next construct an approximation for $\tilde{U}$. Let $\tilde{U}_2$ be a solution to the following ODE which is independent of $\Theta$.
\begin{align} \label{ U tilde 2 ODE}
	\dot{\tilde{U}}_2 = \bar{F}(\tilde{U}_2)+ \epsilon \sum_{i=1}^N  A_{i,2}(\tilde{U}_2), \quad \tilde{U}_2(0) =  U(0) - \epsilon \sum_{i=1}^n A_{i,1}(U(0),\theta).
\end{align}
Then integrating (\ref{ U tilde exp}) and (\ref{ U tilde 2 ODE}) and taking the difference, it follows from Gronwall's lemma that
\begin{align}
	\tilde{U} = U - \epsilon \sum_{i=1}^N A_{i,1}(U,\Theta) = \tilde{U}_2 + \mathcal{O}(\epsilon^2).
\end{align}
and therefore Taylor expanding in $U$ we have 
\begin{align} \label{U with U2}
	U(t) &= \tilde{U}_2(t) + \epsilon \sum_{i=1}^N A_{i,1}(U,\Theta) + \mathcal{O}(\epsilon^{2}).
	\nn
	&= \tilde{U}_2(t) + \epsilon \sum_{i=1}^N A_{i,1}\big(\tilde{U}_2+\epsilon \sum_{i=1}^N A_{i,1}(U,\Theta) ,\Theta\big) + \mathcal{O}(\epsilon^{2}).
	\nn
	&= \tilde{U}_2 + \epsilon \sum_{i=1}^N A_{i,1}(\tilde{U}_2,\Theta) + \mathcal{O}(\epsilon^{2}).
\end{align}

\smallskip

\textbf{Step 3: First order Approximation of $\Theta$:} Due to the presence of $ \Theta $ inside (\ref{U with U2}), we still do not have a complete 
approximation for $U$ of order $\epsilon^2$. To this end, we must use this result to approximate $\Theta$ to order $\epsilon^1$ and then Taylor 
expand (\ref{U with U2}) again, now in $\Theta$. Let $A_{i,1,x}$ be the $X$ component of $A_{i,1}$ and $\tilde{U}_2 = (\tilde{X}_2,\tilde{R}_2,\tilde{Z}_2)$. Then integrating the ODE on $\Theta$, and 
Taylor expanding, we obtain
\begin{align}
\Theta(t) &= \theta -\int_0^t \frac{b_e(X)}{\epsilon\vl{\xi}} + \vl{\xi}^{-1}\bar{Q}_\theta(X_1,R_1,Z_1)ds + \mathcal{O}(\epsilon)
\nn
&= \theta -\int_0^t \frac{b_e(\tilde{X}_2 + \epsilon \sum_{i=1}^N A_{i,1,x}(\tilde{U}_2,\Theta))}{\epsilon\vl{\xi}} + \vl{\xi}^{-1}\bar{Q}_\theta(X_1,R_1,Z_1)ds + \mathcal{O}(\epsilon)
\nn
& = \theta - \int_0^t \frac{b_e(\tilde{X}_2)}{\epsilon\vl{\xi}} +\vl{\xi}^{-1}\bar{Q}_\theta(X_1,R_1,Z_1)ds  + \sum_{i=1}^N\int_0^t \vl{\xi}^{-1} \nabla b_e(\tilde{X}_2)\cdot A_{i,1,x}(\tilde{U}_2,\Theta)ds +\mathcal{O}(\epsilon). 
\end{align}
Then since $A_{i,1,x}(\tilde{U}_2,\Theta)$ can be decomposed in the same way as (\ref{F decomp}), we may repeat the same computation as (\ref{oscilating term}), so that the time integration yields only terms of size $\epsilon$ and we have 
\begin{align}
\Theta &=\theta -\int_0^t \frac{b_e(\tilde{X}_2)}{\epsilon\vl{\xi}} + \vl{\xi}^{-1}\bar{Q}_\theta(X_1,R_1,Z_1)ds + \mathcal{O}(\epsilon).
\end{align}
Finally, given that $U_1$ and $U_2$ may be solved independently of $\Theta$, we have the order $\epsilon$ approximation for $\Theta$ is
\begin{align}
	\Theta_1 := \theta +\int_0^t -\frac{b_e(\tilde{X}_2)}{\epsilon\vl{\xi}} + \vl{\xi}^{-1}\bar{Q}_\theta(X_1,R_1,Z_1)ds.
\end{align}
 Returning to (\ref{U with U2}) we define 
\begin{align} \label{U 2 final define}
	(X_2,R_2,Z_2)(t) = U_2(t) &:= \tilde{U}_2(t) + \epsilon \sum_{i=1}^N A_{i,1}(\tilde{U}_2,\Theta_1).
\end{align}
Then estimates (\ref{char est 2}) hold.  

\textbf{Step 4: Derivatives of phase estimates :} By construction, $\tilde{U}_2$ and in particular $X_2$ have uniform Lipschitz norm's with respect to $\epsilon$. So differentiating (\ref{U 2 final define}), it is clear that (\ref{X 2 lip}) holds. Note that the 2nd order derivatives will not be uniform. Next we estimate the phase derivatives 

\begin{align}
\frac{1}{\partial_t\Theta_1 + \omega\cdot \nabla_x \Theta_1} &= \ \frac{1}{-\frac{b_e(X_2)}{\epsilon\vl{\xi}} - \omega\cdot \nabla_x \int_0^t \frac{b_e(X_2)}{\epsilon\vl{\xi}}ds + \mathcal{O}(1)}
\nn 
&= \frac{-\epsilon \vl{\xi}}{b_e(X_2)}\bigg( \frac{1}{1 + \omega \cdot b_e(X_2)^{-1}\nabla_x \int_0^t b_e(X_2)ds}\bigg) + \mathcal{O}(\epsilon^2)
\nn
&\lesssim \frac{\epsilon C}{(1 - t C)} .
\end{align} 
So we adjust $T$ such that
\begin{align}
	b_- - T||\nabla b_e||_{L^\infty} ||D_x X_2||_{L^\infty} \geq \frac{b_-}{2} > 0.
\end{align} 
Moreover, since $\Theta_1$ only depends on our definition of $U_1$ and $U_2$, the estimate (\ref{Theta 1 est 2}) is also clear. In the same way as Lemma \ref{f initial non stationary} we get (\ref{integral average}). 

\end{proof}

%%%%%%%%%%%%%%%%%%%%%%%%%%
%%%%%%%%%%%%%%%%%%%%%%%%%%

    \textit{Email address: christophe.cheverry@univ-rennes1.fr
    \\
    Email address: ibrahims@uvic.ca
    \\
    Email address: dpreissl@uvic.ca}

\end{document}